\documentclass[oneside,reqno,11pt,a4paper]{amsart}
\usepackage[T1]{fontenc}
\usepackage[margin=2.5cm]{geometry}
\usepackage{amsmath,amsfonts,amssymb,amscd,amsthm}
\usepackage[usenames,dvipsnames,svgnames]{xcolor}
\usepackage[utf8x]{inputenc}
\usepackage{graphicx}
\graphicspath{{figures/}}
\usepackage{caption}
\usepackage{subcaption}
\usepackage[numbers]{natbib}
\usepackage{doi}
\usepackage{autonum}
\usepackage{makecell}
\usepackage{hyperref}
\usepackage{acronym}
\usepackage{listings}
\usepackage{textgreek}
\usepackage{url}
\usepackage{color}
\usepackage{framed}
\usepackage{verbatim}
\usepackage{fancyvrb}
\usepackage{newtxtext}
\usepackage{newtxmath}
\usepackage{microtype} 
\usepackage[final]{pdfpages}
\usepackage{tikz}
\usepackage{booktabs}
\usepackage{multirow}

\graphicspath{{figures/}}

\hypersetup{
    breaklinks=true,
    bookmarksopen=true,
    pdftitle={Template article},    
    pdfauthor={Santiago Badia, Author Two},     
    colorlinks=true,       
    linkcolor=black,          
    citecolor=blue,        
    filecolor=black,      
    urlcolor=blue           
}

\definecolor{bg}{rgb}{0.93,0.93,0.93}

\newtheorem{theorem}{Theorem}[section]

\newtheorem{proposition}[theorem]{Proposition}

\newtheorem{definition}[theorem]{Definition}

\acrodef{pde}[PDE]{partial differential equation}
\acrodef{fe}[FE]{finite element}
\acrodef{fem}[FEM]{finite element method}
\acrodef{xfem}[XFEM]{extended finite element method}
\acrodef{dof}[DOF]{degree of freedom}
\acrodefplural{dof}[DOFs]{degrees of freedom}
\acrodef{agfe}[AgFE]{aggregated finite element}
\acrodef{agfem}[AgFEM]{aggregated finite element method}
\acrodef{cg}[CG]{continuous Galerkin}
\acrodef{dg}[DG]{discontinuous Galerkin}
\acrodef{gp}[GP]{\emph{ghost penalty}}
\acrodef{bgp}[B-GP]{Bulk ghost penalty}
\acrodef{fgp}[F-GP]{Face ghost penalty}
\acrodef{agp}[A-GP]{intra-aggregate ghost penalty}
\acrodef{wagL2}[W-Ag-$L^2$]{weak AgFEM by $L^2$ product}
\acrodef{wagH1}[W-Ag-$\boldsymbol{\nabla}$]{weak AgFEM by $H^1$ product}
\acrodef{bgpi}[B-GP-i]{B-GP stabilisation with interpolation}
\acrodef{sag}[S-Ag]{strong version of AgFEM}

\newcommand{\tnor}[1]{\| #1 \|}

\definecolor{shadecolor}{gray}{.92}
\definecolor{incolor}{rgb}{0,0,.7}
\definecolor{outcolor}{rgb}{.65,0,0}
\definecolor{syntaxcolor}{rgb}{.65,0,0}

\begin{document}

\title[Linking ghost penalty and aggregated unfitted methods]{Linking ghost penalty and aggregated unfitted methods}

\author[S. Badia]{Santiago Badia$^{1,2,*}$}

\author[E. Neiva]{Eric Neiva$^{2}$}

\author[F. Verdugo]{Francesc Verdugo$^{2}$}

\thanks{\null\\
$^{1}$ School of Mathematics, Monash University, Clayton, Victoria, 3800, Australia.\\
$^{2}$ Centre Internacional de M\`etodes Num\`erics en Enginyeria, Esteve Terrades 5, E-08860 Castelldefels, Spain.\\
$^*$ Corresponding author.\\
E-mails: {\tt santiago.badia@monash.edu} (SB) 
{\tt eneiva@cimne.upc.edu} (EN)
{\tt fverdugo@cimne.upc.edu} (FV)
}


\begin{abstract}
In this work, we analyse the links between ghost penalty stabilisation and aggregation-based discrete extension operators for the numerical approximation of elliptic partial differential equations on unfitted meshes. We explore the behavior of ghost penalty methods in the limit as the penalty parameter goes to infinity, which returns a strong version of these methods. We observe that these methods suffer locking in that limit. On the contrary, aggregated finite element spaces are locking-free because they can be expressed as an extension operator from well-posed to ill-posed degrees of freedom. Next, we propose novel ghost penalty methods that penalise the distance between the solution and its aggregation-based discrete extension. These methods are locking-free and converge to aggregated finite element methods in the infinite penalty parameter limit. We include an exhaustive set of numerical experiments in which we compare weak (ghost penalty) and strong (aggregated finite elements) schemes in terms of error quantities, condition numbers and sensitivity with respect to penalty coefficients on different geometries, intersection locations and mesh topologies.
\end{abstract}

\maketitle

\noindent{{\bf {Keywords}}: Embedded methods; unfitted finite elements; stabilisation techniques; ghost penalty; aggregated finite elements.

\section{Introduction}\label{sec:introduction}

Standard \acp{fem} requires cumbersome and time-consuming body-fitted mesh generation, which is hard to automatise and does not scale properly on distributed platforms. These methods are not suitable for problems with moving boundaries or interfaces. Conversely, unfitted \acp{fem} provide a great amount of flexibility at the geometrical discretisation step. They can embed the domain of interest in a geometrically simple background grid (usually a uniform or an adaptive Cartesian grid), which can be generated and partitioned much more efficiently. Analogously, they can easily capture embedded interfaces. As a result, unfitted \ac{fe} methods are generating interest in applications with moving interfaces,
{such as fracture
mechanics~\cite{Waisman2013,berger-vergiat_inexact_2012}, 
fluid-structure
interaction~\cite{schott2019monolithic,alauzet2016nitsche,zonca2018unfitted,Massing2015},
two-phase and free surface
flows~\cite{Sauerland2011,kirchhart2016analysis}, localised plastic deformation~\cite{Badia2020Dec}, and in applications
with varying domains, such as shape or topology
optimisation~\cite{Burman2018}, additive
manufacturing~\cite{neiva2020numerical,carraturo2020modeling}, and stochastic
geometry problems~\cite{badia2019embedded}}. In the numerical community, this family of methods is known by different names, e.g., \emph{unfitted}, \emph{embedded}, \emph{immersed}. Besides, many different schemes have been proposed;~see, e.g.~the \ac{xfem}~\cite{belytschko_arbitrary_2001}, the
cutFEM method~\cite{burman_cutfem_2015}, the \ac{agfem}~\cite{Badia2018}, the cutIGA method~\cite{Elfverson2018},
the immersed boundary method~\cite{Mittal2005}, the finite cell
method~\cite{Schillinger2015}, the shifted boundary
method~\cite{main2018shifted}, the immersogeometric
method~\cite{kamensky2015immersogeometric}, the Cartesian grid \ac{fem}~\cite{Navarro-Jimenez2020Jul} and \ac{dg} methods
with cell
aggregation~\cite{saye2017implicit,engwer2012dune,johansson2013high,muller2017high}.

In the case of unfitted boundaries, unfitted methods lead to unstable and severe ill-conditioned discrete problems~\cite{DePrenter2017,Badia2018}, unless a specific technique mitigates the problem. The intersection of a background cell with the physical domain can be arbitrarily small and with unbounded aspect ratio. It leads to the \emph{small cut cell problem}; \ac{fe} shape functions on the background (unfitted) mesh can have arbitrarily small support in the physical domain. This support depends on the intersection between the background mesh and the boundary (or interface), which in general cannot be controlled. This problem is also present on unfitted interfaces with a high contrast of physical properties~\cite{Neiva2021}.

Despite vast literature on the topic (see,  e.g.,~\cite{Kummer2017,lehrenfeld2016high,guzman2017finite,li2019shifted} and the references above), few formulations are fully robust and optimal regardless of cut location and material contrast. The ill-conditioning issue was addressed by the introduction of the so-called \ac{gp}. In the original article~\cite{burman2010ghost}, two variants of the methods were originally proposed: (1) A bulk penalty term that penalised the distance between the finite element solution in a patch and a polynomial of the order of the approximation; (2) a face penalty term that penalised normal derivatives of the function (up to the order of approximation) on faces that were in touch with cut cells. The face-based \ac{gp} has been the method of choice in the so-called CutFEM framework~\cite{burman_cutfem_2015}. These schemes were originally motivated for $\mathcal{C}^{0}$ finite element spaces on simplicial meshes and later used in combination with discontinuous Galerkin formulations.

The so-called \emph{cell aggregation} or \emph{cell agglomeration} techniques
are an alternative way to ensure robustness with respect to cut location. This approach {is very natural in
\ac{dg} methods, as they can be easily formulated on agglomerated
meshes~\cite{helzel_high-resolution_2005,Kummer2017,bastian2009unfitted}}, and the method is robust if each cell (now an aggregate of cells) has \emph{enough} support in the interior of the domain~\cite{muller2017high}. However, the application of these ideas to $\mathcal{C}^{0}$ Lagrangian finite elements is more involved. The question is how to keep $\mathcal{C}^{0}$ continuity after the agglomeration. 

This problem was addressed in~\cite{Badia2018}, and the resulting method was coined \ac{agfem}. \ac{agfem} combines an aggregation strategy with a map that assigns to every geometrical entity (e.g., vertex, edge, face) one of the aggregates containing it. With these two ingredients, \ac{agfem} constructs a \emph{discrete extension operator} from well-posed \acp{dof} (i.e., the ones related to shape functions with enough support in the domain interior) to ill-posed \acp{dof} (i.e., the ones with small support) that preserves continuity. As a result, the basis functions associated with badly cut cells are removed and the ill-conditioning issues solved. The formulation enjoys good numerical properties, such as stability, condition number bounds, optimal convergence, and continuity with
respect to data; detailed mathematical analysis of the method is included
in~\cite{Badia2018} for elliptic problems and in~\cite{Badia2018a} for the
Stokes equation. \ac{agfem} is amenable to arbitrarily complex 3D geometries, distributed implementations for large scale problems~\cite{Verdugo2019},  error-driven $h$-adaptivity and parallel tree-based meshes~\cite{Badia2020Jun}, explicit time-stepping for the wave equation~\cite{Burman2020Nov} and elliptic interface problems with high contrast~\cite{Neiva2021}.

In this work, we aim to explore the links between the \emph{weak} ghost penalty strategy and the \emph{strong} aggregation-based strategy. In order to do this, we analyse the \emph{strong} form limit of \ac{gp} methods. In this process, we discuss the \emph{locking} phenomenon of current \ac{gp} methods. Next, we make use of the \ac{agfem} machinery to define new \ac{gp} formulations that converge to the classical (strong) \ac{agfem} and thus, are locking-free. We propose two alternative  expressions of the penalty method. The stabilisation term penalises the distance between the solution and its  interpolation onto the \ac{agfem} space. This distance can be expressed using a weighted $L^{2}$ product or an $H^1$ product.

This work is structured as follows. We introduce the geometrical discretisation in Section~\ref{sec:cell-aggregation} and the problem statement in Section~\ref{sec:problem_statement}, which include notations and definitions that are required to implement \ac{gp} and \ac{agfem}. Next, we introduce some \ac{gp} formulations and analyse their \emph{strong} limit in Section~\ref{sec:ghost-penalty}. After that, we present the \ac{agfem} and its underlying discrete extension operator in Section~\ref{sec:agfem}. In this section, we propose a new family of \ac{gp} methods that are a weak version of \ac{agfem}. We comment on the implementation aspects of all the methods in Section~\ref{sec:implementation-aspects}. In Section~\ref{sec:numerical_experiments}, we provide a detailed comparison of all these different schemes, in terms of accuracy and condition number bounds, for different values of the penalty parameter, geometries, and intersection locations. We consider Poisson and linear elasticity problems on isotropic and anisotropic meshes. We draw some conclusions in Section~\ref{sec:conclusions}. The original contributions of the article are:

\begin{enumerate}
  \item A discussion about the links between strong (\ac{agfem}) and weak (\ac{gp}) methods for solving the ill-conditioning of $\mathcal{C}^{0}$ Lagrangian unfitted finite element methods;
  \item A discussion about the locking phenomenon of \ac{gp} strategies in the strong limit and corrective measures;
  \item Design and analysis of \ac{gp} schemes that are a weak versions of \ac{agfem} and locking-free;
  \item A thorough numerical experimentation comparing \ac{gp} methods and strong discrete extension methods in terms of error quantities, condition numbers, sensitivity with respect to penalty coefficients, etc.
\end{enumerate}

\section{Geometrical discretisation}\label{sec:cell-aggregation}

Let us consider an open bounded polyhedral Lipschitz domain $\Omega \subset \mathbb{R}^d$, $d$ being the space dimension, in which we pose our \ac{pde}. Standard \ac{fe} methods rely on a geometrical discretisation of $\Omega$ in terms of a partition of the domain (or an approximation of it). This step involves so-called unstructured mesh generation algorithms. The resulting partition is a \emph{body-fitted} mesh of the domain. Embedded discretisation techniques alleviate geometrical constraints, because they do not rely on body-fitted meshes. Instead, these techniques make use of a background partition $\mathcal{T}_{h}$ of an arbitrary artificial domain $\Omega_{h}^{\mathrm{art}}$ such that $\Omega \subset \Omega_{h}^{\mathrm{art}}$. The artificial domain can be trivial, e.g., it can be a bounding box of $\Omega$. Thus, the computation of $\mathcal{T}_h$ is much simpler (and cheaper) than a body-fitted partition of $\Omega$. For simplicity, we assume that $\mathcal{T}_{h}$ is conforming, quasi-uniform and shape-regular; we represent with $h_T$ the diameter of a cell $T \in \mathcal{T}_{h}$ and define the characteristic mesh size $h \doteq \max_{T \in \mathcal{T}_h} h_T$. We refer to~\cite{Badia2020Jun} for the extension to non-conforming meshes.

The definition of unfitted \ac{fe} discretisations requires some geometrical classifications of the cells in the background mesh $\mathcal{T}_{h}$ and their \emph{n-faces}. We use {n-face} to denote entities in any dimension. E.g., in 3D, 0-faces are vertices, 1-faces are edges, 2-faces are faces and 3-faces are cells. We use \emph{facet} to denote an n-face of dimension $D$-1, i.e., an edge in 2D and a face in 3D.  

The cells in the background partition with null intersection with $\Omega$ are \emph{exterior} cells. The set of exterior cells $\mathcal{T}_{h}^{\mathrm{out}}$  is not considered in the functional discretisation and can be discarded.  $\mathcal{T}_h^{\mathrm{act}} \doteq \mathcal{T}_{h}^{} \setminus \mathcal{T}_{h}^{\mathrm{out}}$ is the \emph{active} mesh. E.g., the \ac{xfem} relies on a standard \ac{fe} space on $\mathcal{T}_{h}^{\mathrm{act}}$. Unfortunately, the resulting discrete system can be singular (see the discussion below). This problem, a.k.a. \emph{small cut cell problem}, is due to cells with arbitrarily small support on $\Omega$. This fact motivates the further classification of cells in $\mathcal{T}_{h}^{\mathrm{act}}$. Let $\mathcal{T}_{h}^{\mathrm{in}}$ be the subset of cells in $\Omega$ and $\mathcal{T}_h^{\mathrm{cut}}$ the cut cells (see Figure~\ref{fig:defs}).\footnote{One can consider a more subtle classification as follows. For any cell $T \in \mathcal{T}_{h}$, we define the quantity $\eta_{T} \doteq \frac{|T \cap \Omega|}{|T|}$, $\left| \cdot \right|$ being the measure of the set. For a given threshold $\eta_{0} \in \left( 0,1 \right]$, we classify cells $T \in \mathcal{T}_{h}^{\mathrm{act}}$ as \emph{well-posed} if $\eta_T \geq \eta_{0}$ and \emph{ill-posed} otherwise. The interior-cut classification is recovered for $\eta_0 = 1$. In any case, the following discussion applies verbatim.} The interior of $\bigcup_{T \in \mathcal{T}_{h}^{\#}} T$ is represented with $\Omega_h^{\#}$ for $\# \in \left\{ \mathrm{act}, \mathrm{in}, \mathrm{cut} \right\}$.

Some of the methods below require the definition of \emph{aggregates}. Let us consider $\mathcal{T}_{h}^{\mathrm{ag}}$ (usually called an agglomerated or aggregated mesh) obtained after a cell aggregation of $\mathcal{T}_{h}^{\mathrm{act}}$, such that each cell in $\mathcal{T}_{h}^{\mathrm{in}}$ only belongs to one aggregate and each aggregate only contains one cell in $\mathcal{T}_{h}^{\mathrm{in}}$ (the root cell). Based on this definition, cell aggregation only acts on the boundary; interior cells that are not in touch with ill-posed cells remain the same after aggregation. Let $\mathcal{T}_h^{\partial,\mathrm{ag}} \doteq \mathcal{T}_{h}^{\mathrm{ag}} \setminus \mathcal{T}_h^{\mathrm{in}}$ be the non-trivial aggregates on the boundary (see Figure~\ref{fig:defs}). We refer the interested reader to \cite{Badia2018} for the definition of cell aggregation algorithms. The aggregation has been extended to non-conforming meshes in \cite{Badia2020Jun} and interface problems in \cite{Neiva2021}, and its parallel implementation described in \cite{Verdugo2019}. It is essential for convergence to minimise the aggregate size in these algorithms. In particular, the characteristic size of an aggregated cell must be proportional to the one of its root cell. The motivation of the cell aggregation is to end up with a new partition in which all cells (aggregates) have support in $\Omega$ away from zero, are connected and are shape-regular.

We represent with $\mathcal{C}_{h}^{\#}$ the (simplicial or hexahedral) \emph{exact complex} of $\mathcal{T}_h^\#$  for $\# \in \left\{  \mathrm{act}, \mathrm{in}, \mathrm{out} \right\}$, i.e., the set of all n-faces of cells in $\mathcal{T}^{\#}_{h}$. $\mathcal{C}_h^{\mathrm{ag}}$ is the subset of n-faces in $\mathcal{C}_h^{\mathrm{act}}$ that lay on the boundaries of aggregates in $\mathcal{T}_{h}^{\mathrm{ag}}$. 
We also define the set of \emph{ghost boundary} facets  $\mathcal{F}_{h}^{\mathrm{gh},\mathrm{cut}}$ as the facets in $\mathcal{C}_{h}^{\mathrm{cut}}$ that belong to two active cells in $\mathcal{T}_{h}^{\mathrm{act}}$ (cut or not). $\mathcal{F}_h^{\mathrm{gh},\mathrm{ag}} \doteq \mathcal{F}_{h}^{\mathrm{gh},\mathrm{cut}} \setminus \mathcal{C}_{h}^{\mathrm{ag}}$ is the subset of these facets that do not lay on aggregate boundaries, see Figure~\ref{fig:defs}.

\definecolor{FigBlue}{RGB}{0,0,255}
\definecolor{FigCyan}{RGB}{0,255,255}
\definecolor{FigPurple}{RGB}{222,135,170}
\definecolor{FigYellow}{RGB}{255,221,85}
\definecolor{FigRed}{RGB}{255,0,0}

\begin{figure}[ht!]
\centering

\begin{tabular}{ccc}
\tikz{\draw[fill=FigYellow] rectangle(1.5ex,1.5ex);} $\in\mathcal{T}^{\mathrm{in}}_h$ &
\tikz{\draw[fill=FigCyan] rectangle(1.5ex,1.5ex);} $\in\mathcal{T}^{\mathrm{cut}}_h$ &
\tikz{\draw[fill=FigPurple] rectangle(1.5ex,1.5ex);} $\in\mathcal{T}_h^{\partial,\mathrm{ag}}$ \\[1em]
\tikz{\draw[color=black,line width=1.5] (0,0)--(1.7ex,1.7ex);}$\in\partial\Omega$ &
\tikz{\draw[color=FigBlue,line width=1.5] (0,0)--(1.7ex,1.7ex);}$\in\mathcal{F}_{h}^{\mathrm{gh},\mathrm{cut}}$ &
\tikz{\draw[color=FigRed,line width=1.5] (0,0)--(1.7ex,1.7ex);}$\in\mathcal{F}_{h}^{\mathrm{gh},\mathrm{ag}}$ \\[1em]
\end{tabular}

\begin{subfigure}{0.24\textwidth}
\centering
\includegraphics[width=0.8\textwidth]{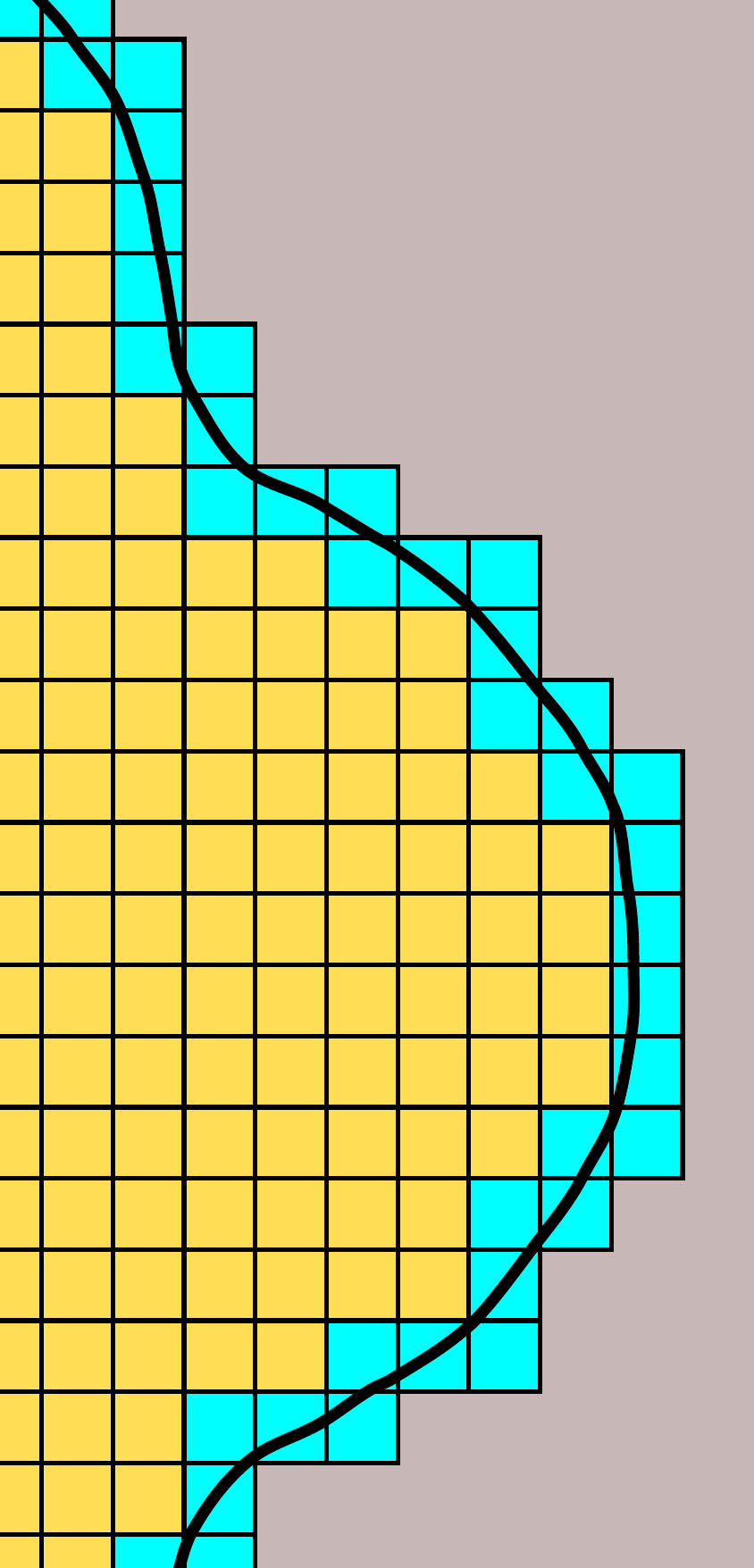}
\caption{}
\end{subfigure}
\begin{subfigure}{0.24\textwidth}
\centering
\includegraphics[width=0.8\textwidth]{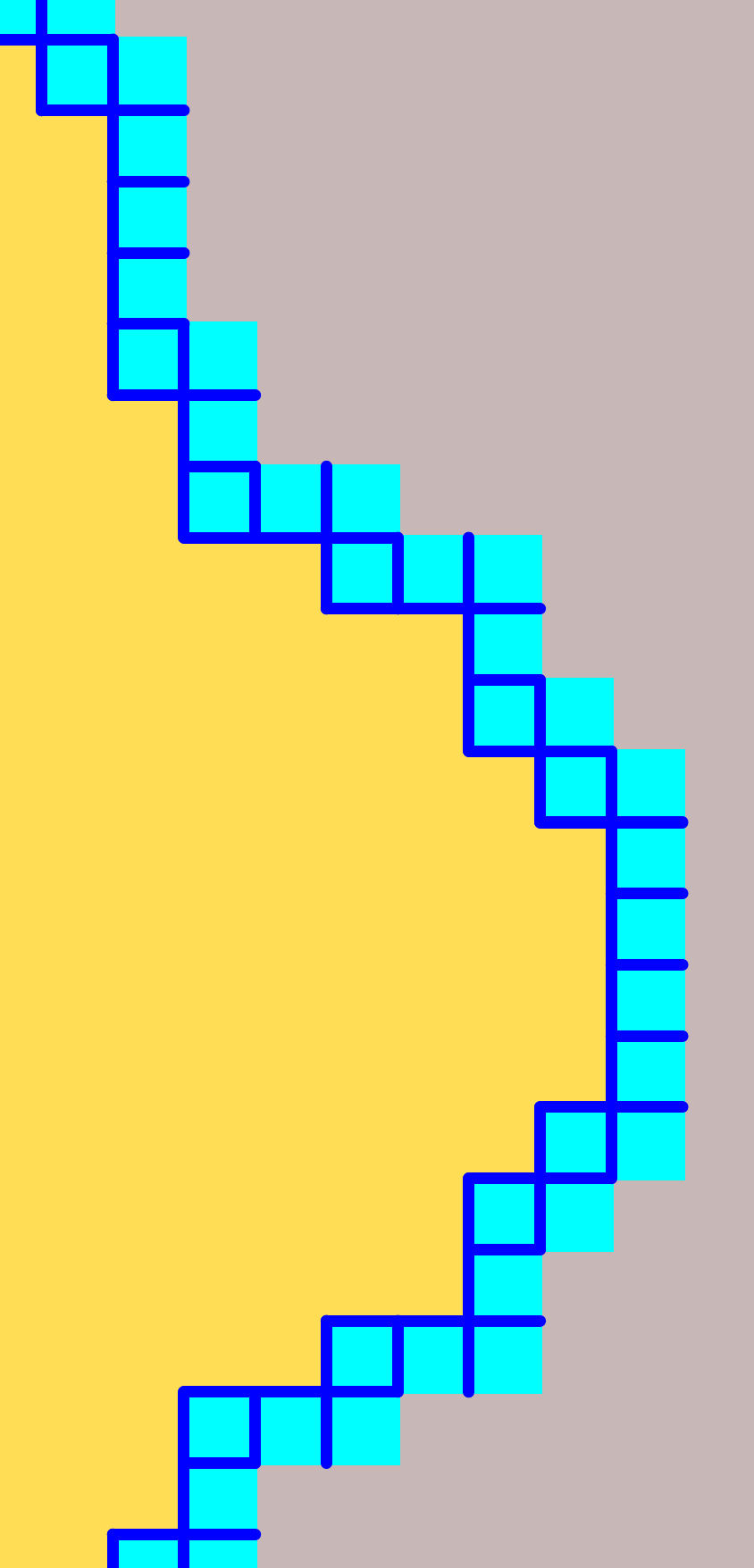}
\caption{}
\end{subfigure}
\begin{subfigure}{0.24\textwidth}
\centering
\includegraphics[width=0.8\textwidth]{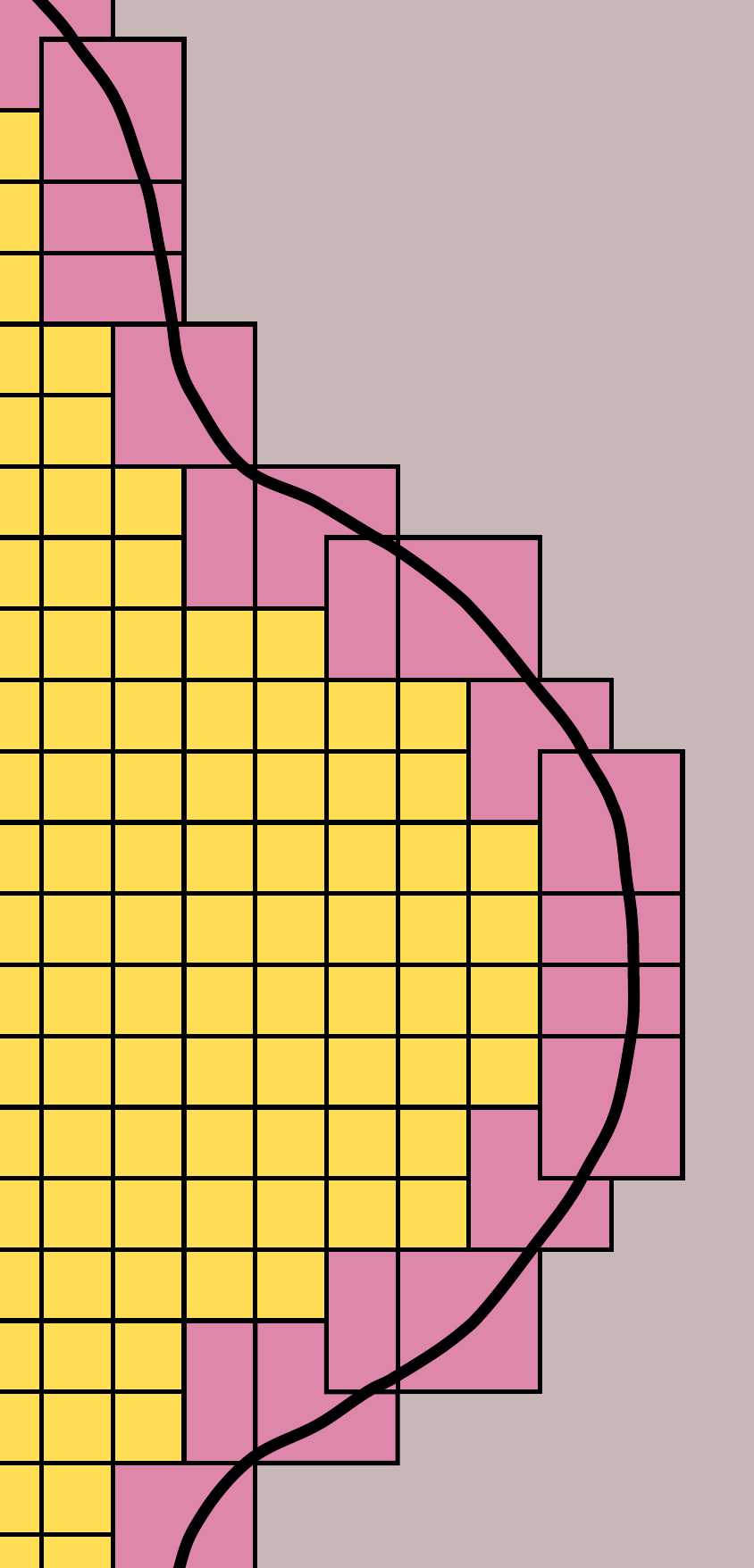}
\caption{}
\end{subfigure}
\begin{subfigure}{0.24\textwidth}
\centering
\includegraphics[width=0.8\textwidth]{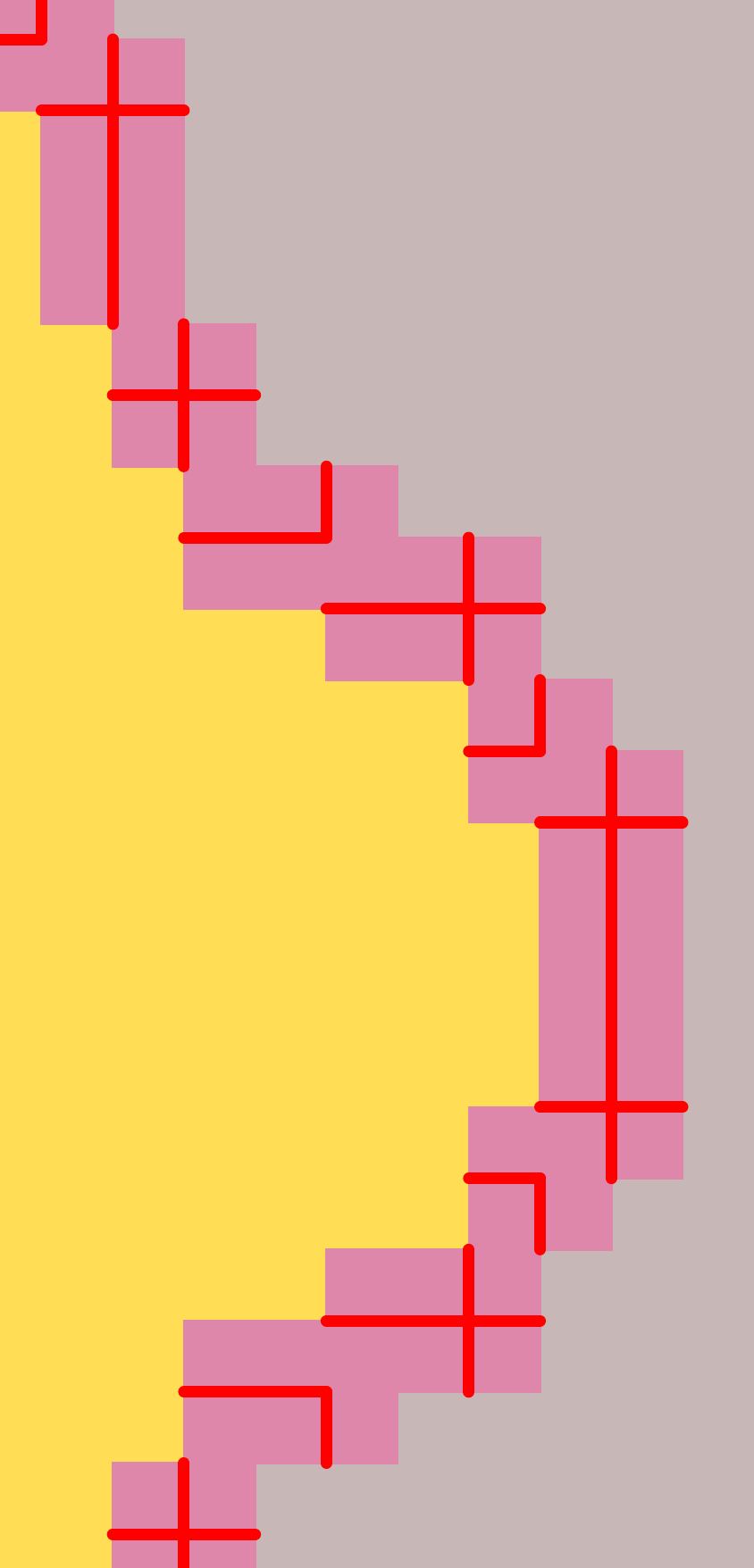}
\caption{}
\end{subfigure}
\caption{Illustration of the main geometrical sets introduced in Sect.~\ref{sec:cell-aggregation}.}
\label{fig:defs}
\end{figure}

\section{Problem statement}\label{sec:problem_statement}

Let us consider first the Poisson equation in $\Omega$ with Dirichlet boundary conditions on $\Gamma_{\mathrm{D}} \subset \partial \Omega$ and Neumann boundary conditions on $\Gamma_{\mathrm{N}} \doteq \partial \Omega \setminus \Gamma_{D}$. After scaling with the diffusion term, the equation reads:
\emph{find} $u \in H^1(\Omega)$ \emph{such that}
\begin{equation}\label{eq:poisson-strong}
	- \boldsymbol{\nabla} \cdot \boldsymbol{\nabla} u = f \quad \text{in} \ H^{-1}(\Omega), \qquad u = g \quad \text{in} \ H^{1/2}(\Gamma_{\mathrm{D}}),  \qquad \boldsymbol{n} \cdot \boldsymbol{\nabla} u = q \quad \text{in} \ H^{-\frac{1}{2}}(\Gamma_{\mathrm{N}}),
\end{equation}
where $f$ is the source term, $g$ is the prescribed value on the Dirichlet boundary and $q$ the prescribed flux on the Neumann boundary. 

The following presentation readily applies to other second-order elliptic problems. In particular, we will also consider the linear elasticity problem:
\emph{find} $\boldsymbol{u} \in \boldsymbol{H}^1(\Omega)$ \emph{such that}
\begin{equation}\label{eq:elasticity-strong}
	- \boldsymbol{\nabla} \cdot \boldsymbol{\sigma}(\boldsymbol{u}) = \boldsymbol{f} \quad \text{in} \ \boldsymbol{H}^{-1}(\Omega), \qquad \boldsymbol{u} = \boldsymbol{g} \quad \text{in} \ \boldsymbol{H}^{1/2}(\Gamma_{\mathrm{D}}),  \qquad \boldsymbol{n} \cdot \boldsymbol{\sigma}(\boldsymbol{u}) = \boldsymbol{q} \quad \text{in} \ \boldsymbol{H}^{-\frac{1}{2}}(\Gamma_{\mathrm{N}}),
\end{equation}
where $\boldsymbol{\epsilon}, \boldsymbol{\sigma} : \Omega \to \mathbb{R}^{d,d}$ are the strain tensor
$\boldsymbol{\epsilon}(\boldsymbol{u}) \doteq \frac{1}{2} (\boldsymbol{\nabla} \boldsymbol{u} + { \boldsymbol{\nabla} \boldsymbol{u}}^T)$ and stress tensor
$\boldsymbol{\sigma} (\boldsymbol{u}) = 2 \mu \boldsymbol{\epsilon}(\boldsymbol{u}) + \lambda \mathrm{tr} (\boldsymbol{\epsilon}(\boldsymbol{u})) \mathbf{Id}$; 
$\mathbf{Id}$ denotes the identity matrix in $\mathbb{R}^d$. $(\mu, \lambda)$ are the the Lamé coefficients. We consider the Poisson ratio $\nu \doteq
\lambda / ( 2 ( \lambda + \mu ) )$ is bounded away from
$1/2$, i.e.~the material is compressible. Since $\lambda = 2\nu
\mu / (1 - 2\nu)$, $\lambda$ is bounded above by
$\mu$, i.e.~$\lambda \leq C \mu$, $C > 0$. 

The simplification of the geometrical discretisation, in turn, complicates the functional discretisation. Standard (body-fitted) \acp{fem} cannot be straightforwardly used. First, the strong imposition of Dirichlet boundary conditions relies on the fact that the mesh is body-fitted. In an embedded setting, Dirichlet boundary conditions are weakly imposed instead. Second, the cell-wise integration of the \ac{fe} forms is more complicated; integration must be performed on the intersection between cells and $\Omega$ only. Third, naive discretisations can be arbitrarily ill-posed.

Let $\mathcal{V}_{h}^{\mathrm{act}}$ be a standard Lagrangian \ac{fe} space on $\mathcal{T}_{h}^{\mathrm{act}}$. As stated above, we consider a weak imposition of boundary conditions with Nitsche's method~\cite{Badia2018,Schillinger2015,burman_cutfem_2015}. This approach provides a consistent numerical scheme with optimal convergence for arbitrary order \ac{fe} spaces. According to this, we approximate the Poisson problem in \eqref{eq:poisson-strong} with the Galerkin method: 
find $u_h \in \mathcal{V}_h^{\rm act}$ such that $a_h(u_h,v_h)= b_h(v_h)$
 for any $v_h \in \mathcal{V}_h^{\rm act}$, with
 \begin{equation}\label{eq:poisson-weak}
 	\begin{array}{l}
 		\displaystyle {a}_h(u_h,v_h) \doteq \int_{\Omega} \boldsymbol{\nabla} u_h \cdot
 		\boldsymbol{\nabla} v_h \mathrm{\ d}\Omega \ + \int_{\Gamma_{\mathrm{D}}} \left( \tau u_h v_h  - u_h
 		\left( \boldsymbol{n} \cdot \boldsymbol{\nabla} v_h \right ) - v_h \left( \boldsymbol{n}
 		\cdot \boldsymbol{\nabla} u_h \right) \right) \mathrm{\ d}{\Gamma}, \quad \text{and} \\
 		\displaystyle {b}_h(v_h) \doteq \int_{\Omega} f v_h  \mathrm{\ d}\Omega \
 		+ \int_{\Gamma_{\mathrm{D}}} \left( \tau g v_h  - \left( \boldsymbol{n} \cdot \boldsymbol{\nabla} v_h
 		\right ) g \right) \mathrm{\ d}\Gamma + \int_{\Gamma_{\mathrm{N}}}^{} q v_h \ \mathrm{d}\Gamma,
 	\end{array}
 \end{equation}
with $\boldsymbol{n}$ being the outward unit normal on $\partial \Omega$. In the case of the linear elasticity problem in~(\ref{eq:elasticity-strong}), the bilinear form and linear functional read: 
\begin{align}\label{eq:elasticity-weak}
	a_h(\boldsymbol{u}_h,\boldsymbol{v}_h) & \doteq \int_{\Omega} \boldsymbol{\sigma}(\boldsymbol{u}_h) : \boldsymbol{\epsilon}(\boldsymbol{v}_h) \ \mathrm{d} \Omega \ +  
   \int_{\Gamma_{\mathrm{D}}} \left( \tau {\boldsymbol{u}_h} \cdot 
	 {\boldsymbol{v}_h} - \boldsymbol{n} \cdot \boldsymbol{\sigma}(\boldsymbol{v}_h)
	 \cdot {\boldsymbol{u}_h} - \boldsymbol{n} \cdot 
	 {\boldsymbol{\sigma}(\boldsymbol{u}_h)} \cdot {\boldsymbol{v}_h} \right) \ \mathrm{d} \Gamma  , \\
	 b_h(\boldsymbol{v}_h) &\doteq \int_{\Omega} \boldsymbol{f} \cdot \boldsymbol{v}_h \ \mathrm{d} \Gamma \ +
   \int_{\Gamma_{\mathrm{D}}} ( \tau \boldsymbol{g} \cdot 
	 {\boldsymbol{v}_h} - \boldsymbol{n} \cdot {\boldsymbol{\sigma}(\boldsymbol{v}_h)} 
	 \cdot \boldsymbol{g} ) \ \mathrm{d} \Gamma + \int_{\Gamma_{\mathrm{D}}} \boldsymbol{q} \cdot  \boldsymbol{v}_h \ \mathrm{d} \Gamma.
\end{align}

We note that the second term in all the forms above are associated with the weak imposition of Dirichlet boundary conditions with Nitsche's method~\cite{nitsche_uber_2013}.\footnote{Other weak imposition of boundary conditions involve the penalty method or a non-symmetric version of Nitsche's method \cite{Freund1995}. In any case, the penalty formulation is not weakly consistent for higher order methods and the non-symmetric formulation breaks the symmetry of the system and is not adjoint consistent.}  Stability of Lagrangian \acp{fem} relies, e.g., for the Poisson problem, on the following property:
\begin{equation}\label{eq:nistche-condition}
  \begin{array}{l}
    \int_{\Gamma_{\mathrm{D}} \cap T} \left( \tau_T u_h^2  - 2 u_h
    \left( \boldsymbol{n} \cdot \boldsymbol{\nabla} u_h \right ) \right) \mathrm{\ d}{\Gamma}
    \leq C \int_{\Gamma_{\mathrm{D}} \cap T}^{} \tau_T u_{h}^{2} \ \mathrm{d} \Gamma + \|\boldsymbol{\nabla}u_h\|^2_{\boldsymbol{L}^2(T)}, \ \ \forall T\in\mathcal{T}_h^{\mathrm{cut}}
  \end{array}
\end{equation}
for some $C>0$ independent of $h_T$. 

A value of $\tau_T$ that satisfies~\eqref{eq:nistche-condition} can be computed using a cell-wise eigenvalue problem. 
In shape-regular body-fitted meshes, one can simply use $\tau_T \doteq {\beta}{h_{T}^{-1}}$. {Here, $\beta = \tilde{\beta} m^2$, with $\tilde{\beta}$ a \emph{large enough} problem-dependent parameter and $m$ the order of $\mathcal{V}_{h}^{\mathrm{act}}$; while $h_T$ is the diameter of $T$}. For \ac{xfem}, we only have stability over $\|\boldsymbol{\nabla}u_h\|^2_{\boldsymbol{L}^2(T \cap \Omega)}$ in the right-hand side of~(\ref{eq:nistche-condition}). In this case, the minimum value of $\tau_T$ that makes the problem stable tends to infinity in some cell configurations as $|T \cap \Omega | \to 0$. As a result, unfitted \acp{fem} like \ac{xfem} are not robust to cell cut locations (boundaries or interfaces). 

Condition number bounds for the resulting linear system are strongly linked to the stability issue commented above. A shape function of the finite element space $\phi_{i}$ must satisfy, in at least one cell $T \in \mathcal{T}_{h}^{\mathrm{act}}$,
\begin{equation}\label{eq:mass-bound}
  C_- h_T^d \leq \int_{T \cap \Omega}^{} \phi_i^2 \ \mathrm{d} T \leq C_+ h_T^d,
\end{equation} 
for $C_+ > C_- > 0$ independent of $h_T$. This condition can be proved for Lagrangian \acp{fe} on body-fitted meshes ($T \cap \Omega = T$). It implies that the mass matrix is spectrally equivalent to the identity matrix for quasi-uniform and shape-regular partitions. However, it is obvious to check that this condition fails as $|T \cap \Omega | \to 0$ since $\int_{T}^{} \phi_i^2 \ \mathrm{d} T \to 0$. The plain vanilla \ac{fem} on the background mesh leads to severely ill-conditioned linear systems. Recent research works have tried to solve this issue at the preconditioner level (see, e.g.,~\cite{DePrenter2017,badia_robust_2017}). 

In the coming sections, we will present methods that solve the previous problems by providing stability over full background mesh cells $T \in \mathcal{T}_{h}^{\mathrm{act}}$ , i.e., $T$ instead of $T \cap \Omega$. With these methods, we can use the same expression of $\tau_T$ as in body-fitted meshes, $h_T$ being the background cell size. We will often make abuse of notation, using $u_h$ to refer to both the scalar unknown in~(\ref{eq:poisson-weak}) and the vector unknown in~(\ref{eq:elasticity-weak}), making distinctions where relevant. We use $A \gtrsim B$ (resp. $A \lesssim B$) to denote $A \geq C B$ (resp. $A  \leq CB$) for some positive constant $C$ that does not depend on $h$ and the location of the cell cuts. Uniform bounds irrespectively of boundary or interface locations is the driving motivation behind all these methods.

\section{Ghost penalty}\label{sec:ghost-penalty}

The ghost penalty formulation was originally proposed in \cite{burman2010ghost} to fix the ill-posedness problems of \ac{xfem} discussed above. The method adds a stabilising bilinear form $s_h$ to the formulation that provides an \emph{extended} stability while preserving the convergence rates of the body-fitted method. 

\begin{definition}\label{def:suitable-ghost-penalty}
  Let $n \leq m$,  where $m$ is the order of $\mathcal{V}_{h}^{\mathrm{act}}$, and $\mathcal{E}^{n}: H^n(\Omega) \rightarrow H^n(\Omega_{h}^{\mathrm{act}})$ be a continuous extension operator. Let $\mathcal{V}(h) \doteq \mathcal{V}_{h}^{\mathrm{act}} + H^2(\Omega_{h}^{\mathrm{act}})$. We endow $\mathcal{V}(h)$ with the extended stability norm:
  \begin{equation}\label{eq:norm-gp}
    \tnor{v}^{2}_{\mathcal{V}(h)} \doteq \|\boldsymbol{\nabla}v\|^{2}_{\boldsymbol{L}^2(\Omega_{h}^{\mathrm{act}})} +  \| \tau^{\frac{1}{2}} v \|_{L^2(\Gamma_{\mathrm{D}})}^2 
    + \sum_{T \in \mathcal{T}_{h}^{\mathrm{act}}}^{} h_T^{2} \| v \|_{H^2(T)}^{2}, \qquad \forall u \in \mathcal{V}(h).
    \end{equation}
  A suitable \ac{gp} term $s_h$ for problems~(\ref{eq:poisson-weak}) and~(\ref{eq:elasticity-weak}) is a positive semi-definite symmetric form that satisfies the following properties uniformly w.r.t.~the mesh size $h$ of the background mesh and interface intersection:
  \begin{enumerate}
	  \item[(i)] Extended stability: 
    \begin{equation}\label{eq:gp-extended-stability}
      s_{h}(u_{h},u_{h}) + \|\boldsymbol{\nabla}u_h\|^{2}_{\boldsymbol{L}^2(\Omega)} \gtrsim 
      \|\boldsymbol{\nabla}u_h\|^{2}_{\boldsymbol{L}^2(\Omega_{h}^{\mathrm{act}})}, \qquad \forall u_h \in \mathcal{V}_{h}^{\mathrm{act}},
    \end{equation}
\item[(ii)] Continuity:
\begin{equation}\label{eq:gp-continuity}
  s_h(u,v_h) \lesssim \tnor{u}_{\mathcal{V}(h)} \tnor{v_h}_{\mathcal{V}(h)}, \qquad \forall u \in \mathcal{V}(h), \ v_h \in \mathcal{V}_{h}^{\mathrm{act}}.
\end{equation}
\item[(iii)] Weak consistency: 
\begin{equation}\label{eq:gp-weak-consistency}
   s_h(\mathcal{E}^{n+1}(u),v_h) \lesssim h^{n} \| u \|_{H^{n+1}(\Omega)} \tnor{v_h}_{\mathcal{V}(h)}, \qquad  \forall u \in H^{n+1}(\Omega).
\end{equation} 
\end{enumerate}
\end{definition}

Under these conditions, the \ac{gp} unfitted formulation is well-posed and exhibits optimal convergence rates. The following abstract results have been proved in the literature for specific definitions of $s_h$ that satisfy Def.~\ref{def:suitable-ghost-penalty}. See, e.g., \cite[Lem.~4.2, Lem.~4.3, Prop.~4.4]{burman2010ghost}, \cite[Sect.~4]{Burman2012} and  \cite[Sec.~4]{Hansbo2017}. Note that the norm includes control, not only over $\| \boldsymbol{\nabla} v_h \|_{\boldsymbol{L}^2(\Omega)}$ (which comes from the standard Laplacian term in the formulation), but $\| \boldsymbol{\nabla} v_h \|_{\boldsymbol{L}^2(\Omega_h^{\mathrm{act}})}$. This extra stabilisation in~(\ref{eq:gp-extended-stability}) is provided by the \ac{gp} term and essential for well-posedness independent of the boundary location. 

\begin{proposition}\label{prop:well-posedness-gp}
  Let $s_h$ satisfy Def.~\ref{def:suitable-ghost-penalty}. It holds:
  \begin{align}\label{eq:coercive-continuous}
    a_h(u_h,u_h) + s_h(u_h,u_h) \gtrsim \| u_h \|^2_{\mathcal{V}(h)}, \qquad
    a_h(u,v_h)  + s_h(u,v_h) \lesssim \|u\|_{\mathcal{V}(h)} \|v_h\|_{\mathcal{V}(h)},
    \end{align} 
    for any $u_h, v_h \in  \mathcal{V}_{h}^{\rm act}$, $u \in \mathcal{V}(h)$. Thus, there is a unique 
  \begin{equation}\label{eq:gp-weak-poisson}
    u_h \in \mathcal{V}_h^{\rm act} \ : \ a_h(u_h,v_h) + s_h(u_h,v_h)= b(v_h), \quad \forall v_h \in \mathcal{V}_{h}^{\mathrm{act}}.
  \end{equation}
  Furthermore, the condition number $\kappa$ of the resulting linear system when using a standard Lagrangian basis for $\mathcal{V}_{h}^{\mathrm{act}}$  holds $\kappa \lesssim h^{-2}$. 
\end{proposition}
\begin{proof}
The main ingredient is the following trace inequality for continuous functions on cut cells (see~\cite{hansbo2002unfitted}):
\begin{equation}\label{eq:trace-inequality}
  \| \psi \|^2_{L^2(\partial(\Omega \cap T))} \lesssim h_{T}^{-1} \| \psi 
  \|^2_{L^2(\Omega \cap T)} + h_{T} \left| \psi \right|^{2}_{H^1
  (\Omega \cap T)}, \qquad \forall \psi \in H^1(\Omega \cap T),
\end{equation}
where $\partial(\Omega \cap T)$ is the boundary of $\Omega\cap T$. Using a standard discrete inequality on $\mathcal{T}_{h}^{\mathrm{act}}$, namely $\| \boldsymbol{\nabla} u_h \|_{\boldsymbol{H}^1(T)} \lesssim h_T^{-1} \| u_h \|_{L^2(T)}$ for any $u_h \in \mathcal{V}_{h}^{\mathrm{act}}$ and $T \in \mathcal{T}_{h}^{\mathrm{act}}$,   we readily get
\begin{equation}\label{eq:inverse-trace-inequality}
  \| \boldsymbol{n} \cdot \boldsymbol{\nabla} u_h \|_{L^2( \Gamma_{\rm D} \cap T )}^2 \lesssim h_T^{-1} \| \boldsymbol{\nabla} u_h \|_{\boldsymbol{L}^2( T )}^2, \qquad \forall u_h \in \mathcal{V}_{h}^{\mathrm{act}}.
\end{equation}
Using~\eqref{eq:gp-extended-stability}, Cauchy-Schwarz and Young inequalities, (\ref{eq:inverse-trace-inequality}) and the expression for $\tau$, we can prove~(\ref{eq:nistche-condition}). Thus, the \ac{gp} provides a stronger coercivity result, namely
\begin{equation}\label{gp-coercivity}
  a(u_h,u_h) + s(u_h,u_h) \gtrsim \tnor{u_h}_{\mathcal{V}(h)}^{2}, \qquad \forall u_h \in \mathcal{V}_{h}^{\mathrm{act}},
\end{equation} 
for the same definition of $\tau$ that is needed in body-fitted meshes, i.e., the value of $\beta$ does not depend on the location of the cell cut. Continuity in (\ref{eq:coercive-continuous}) is obtained invoking (\ref{eq:trace-inequality}), the standard inverse inequality $\|\boldsymbol{\nabla} u_h\|_{L^2(T)} \lesssim h_T^{-1} \|u_h\|_{L^2(T)}$ and (\ref{eq:gp-continuity}). The cut-independent condition number bound is a by-product of the enhanced coercivity in~(\ref{eq:coercive-continuous}). Using the Poincar\'e inequality, one can readily prove that $\| u_h \|_{L^2(\Omega_h^{\mathrm{act}})} \lesssim \| u_h \|_{\mathcal{V}(h)}$. Using the inverse inequality on the background mesh and~(\ref{eq:inverse-trace-inequality}) one can prove that $\| u_h \|_{\mathcal{V}(h)} \lesssim h^{-1} \| u_h \|_{L^2(\Omega_h^{\mathrm{act}})}$. Finally, one can use the standard bounds for the mass matrix on the background mesh~(\ref{eq:mass-bound}) to prove the result (see, e.g., \cite[Lem.~4.3]{burman2010ghost}).
\end{proof} 
  \begin{proposition}\label{eq:error-estimates-gp}
  Let $m$ be the order of $\mathcal{V}_{h}^{\mathrm{act}}$. If $u \in H^{n+1}(\Omega)$, $m \geq n \geq 1$, is the solution of~(\ref{eq:poisson-weak}) or (\ref{eq:elasticity-weak}) and $u_h \in \mathcal{V}^{\mathrm{act}}_h$ is the
	solution of \eqref{eq:gp-weak-poisson},  then:
	\begin{equation}
		\|u-u_h\|_{\mathcal{V}(h)} \lesssim h^{n} 
		| u |_{H^{n+1}(\Omega)}, \qquad 
		\|{u-u_h}\|_{L^2(\Omega)} \lesssim h^{n+1} 
		| u |_{H^{n+1}(\Omega)}.
	\end{equation}
\end{proposition}
\begin{proof}
The convergence analysis relies on standard approximation error bounds for $\mathcal{E}^{n+1}(u)$ in $\mathcal{V}_{h}^{\mathrm{act}}$ and the weak consistency in (\ref{eq:gp-continuity}).     
\end{proof}

In the following, we show some definitions of $s_h$ with these properties.

\subsection{\ac{bgp}}\label{sub:bulk-gp}

The \emph{bulk} ghost penalty stabilisation was originally proposed in \cite{burman2010ghost}. The idea of this formulation is to add the stabilisation term
\begin{equation}\label{eq:bulk-penalty}
s_{h}(u_{h},v_{h}) \doteq \sum_{U \in \mathcal{T}_{h}^{\partial,\mathrm{ag}}}   \left( \gamma h_{U}^{-2}(u_h - \pi_U(u_h)) , v_h \right)_U,       
\end{equation}
where $\pi_U(\cdot)$ represents the $L^2(U)$ projection onto the polynomial space $\mathcal{P}_p(U)$ (for simplicial meshes) or $\mathcal{Q}_p(U)$ (for hexahedral meshes). $\gamma > 0$ is a numerical parameter that must be chosen. This method penalises the difference between the restriction of the \ac{fe} solution in the aggregate and a local polynomial space on the aggregate. Since the aggregate-wise polynomial spaces have enough support in $\Omega$ by construction, it provides stability on ill-posed cells.

\begin{proposition}\label{prop:bulk-gp-suitable}
The expression of $s_h$ in~(\ref{eq:bulk-penalty}) is a suitable \ac{gp} stabilisation that satisfies the conditions (\ref{eq:gp-extended-stability})-(\ref{eq:gp-weak-consistency}) in Def.~\ref{def:suitable-ghost-penalty}.  
\end{proposition}
\begin{proof}
The analysis of the \ac{bgp} stabilisation, which indirectly proves the conditions in Def.~\ref{def:suitable-ghost-penalty}, can be found in \cite[Lem.~4.2, Prop.~4.4]{burman2010ghost}.
\end{proof}

\subsection{\ac{fgp}}\label{sub:face-gp}

The most popular ghost stabilisation technique relies on the penalisation of inter-element $m$  normal derivatives on the facets in $\mathcal{F}_{h}^{\mathrm{gh},\mathrm{cut}}$, where $m$ stands for the \ac{fe} space order. This is the stabilisation that is being used in the so-called CutFEM framework \cite{burman_cutfem_2015}. Let us consider linear \acp{fe} in the following exposition. Given a face $F \in \mathcal{F}_{h}^{\mathrm{gh},\mathrm{cut}}$ and the two cells $K$ and $K'$ sharing this face, we define $\llbracket \partial_n u \rrbracket \doteq \boldsymbol{n}_K \cdot \boldsymbol{\nabla} u|_K + \boldsymbol{n}_{K'} \cdot \boldsymbol{\nabla} u|_{K'}$. $h_F$ is some average of $h_K$ and $h_{K'}$. We define the \ac{gp} stabilisation term: 
\begin{equation}\label{eq:face-penalty}
  s_{h}(u_{h},v_{h}) = \sum_{F \in \mathcal{F}_h^{\mathrm{gh},\mathrm{cut}}}^{} 
  \left( \gamma h_F \llbracket \partial_n u_h \rrbracket , \llbracket \partial_n v_h \rrbracket  \right)_F.
\end{equation}
The method weakly enforces $\mathcal{C}^{m}$ continuity across cells as soon as one of the cells is cut. This weak constraint of cut cells also provides the desired extended stability described above.

\begin{proposition}\label{prop:face-gp-suitable}
  The expression of $s_h$ in~(\ref{eq:bulk-penalty}) is a suitable \ac{gp} stabilisation that satisfies the conditions (\ref{eq:gp-extended-stability})-(\ref{eq:gp-weak-consistency}) in Def.~\ref{def:suitable-ghost-penalty}.  
  \end{proposition}
  \begin{proof}
    The analysis of the \ac{fgp} stabilisation can be found in \cite[Lem.~4.2, Lem.~4.3, Prop.~4.4]{burman2010ghost} for linear elements and \cite[Sect.~4]{Burman2012} and  \cite[Sec.~4]{Hansbo2017} in the general case and applied to linear elasticity problems. 
  \end{proof}

\subsection{The limit $\gamma \to \infty$}\label{sub:gp-limit}
The \ac{gp} method underlying idea is 
to \emph{rigidise} the shape functions on the boundary zone. 
One question that arises is: what happens when the penalty parameter $\gamma \to \infty$? Can we use the previous methods in a \emph{strong} way? Can we build spaces that exactly satisfy the constraints in that limit?

Let us consider the \ac{fgp} method first. In the limit $\gamma \to \infty$, the stabilisation term in that case enforces full $\mathcal{C}^{m}$ continuity of the \ac{fe} solution on $\mathcal{F}_{h}^{\mathrm{gh},\mathrm{cut}}$. The only functions in $\mathcal{V}_{h}^{\mathrm{act}}$ that satisfy this continuity are \emph{global} polynomials in $\Omega_{h}^{\mathrm{cut}}$. As a result, the method is not an accurate method that converges to the exact solution as $h \to 0$. 

Let us look at the \ac{bgp} now. In the limit $\gamma \to \infty$, the solution must be a piecewise $\mathcal{C}^0$ polynomial on the aggregated mesh $\mathcal{T}_h^{\mathrm{ag}}$. The limit constraint is weaker than the one for \ac{fgp}. However, it is still too much to eliminate the locking phenomenon in general. In order to satisfy exactly the penalty term, the aggregate values are polynomial extensions of the root cell polynomials. In general, these extensions do not satisfy $\mathcal{C}^0$ continuity (Aggregates are not tetrahedra or hexahedra anymore, but an aggregation of these polytopes that can have more general shapes, e.g., L-shaped.) As a result, the \ac{bgp} term is not just constraining the values of the ill-posed cell (as one could expect) but the ones of the well-posed cells too. Furthermore, these constraints are not local but can propagate across multiple aggregates, depending on the mesh topology.

The previous observations motivate a slight improvement of \ac{fgp} (CutFEM). Instead of considering the face penalty on all faces in $\mathcal{F}_{h}^{\mathrm{gh},\mathrm{cut}}$, one can only penalise the subset of intra-aggregate faces $\mathcal{F}_h^{\mathrm{gh},\mathrm{ag}}$: 
\begin{equation}\label{eq:weak-face-penalty}
  s_{h}(u_{h},v_{h}) = \sum_{F \in \mathcal{F}_h^{\mathrm{gh},\mathrm{ag}}}^{} 
  \left( \gamma h_F \llbracket \partial_n u_h \rrbracket , \llbracket \partial_n v_h \rrbracket  \right)_F.
\end{equation}
It is easy to check that this variant (\ref{eq:weak-face-penalty}) of the \ac{fgp} method, referred to as \textbf{\ac{agp}}, has the same behaviour in the limit as the bulk one, i.e., it is still affected by locking. The analysis in this case readily follows from Prop.~\ref{prop:face-gp-suitable}. 

\section{Aggregated finite elements}\label{sec:agfem}
The natural question that arises from this discussion is: how can we define a strong version of the \ac{gp} stabilisation that is free of the \emph{locking} described above? The answer to this question was the \ac{agfem} proposed in \cite{Badia2018}. The underlying idea of this method is to define a new \ac{fe} space that can be expressed in terms of an aggregate-wise \emph{discrete extension} operator $\mathcal{E}^{\mathrm{ag}}_h: \mathcal{V}_{h}^{\mathrm{in}} \longrightarrow \mathcal{V}_{h}^{\mathrm{act}}$ that extends \ac{fe} functions from the root cells to the cut cells. This definition solves two problems: the non-locality of the constraints and the constraining of interior cells \acp{dof}. The image of this extension is the \ac{agfe} space $\mathcal{V}_{h}^{\mathrm{ag}} \subset \mathcal{V}_{h}^{\mathrm{act}}$; $\mathcal{V}_{h}^{\mathrm{ag}}$ can be built by adding constraints to $\mathcal{V}_{h}^{\mathrm{act}}$. The new \ac{agfe} space is not affected by the small cut cell problem, since the ill-posed \acp{dof} are constrained by well-posed interior \acp{dof}.

Since $\mathcal{V}_{h}^{\mathrm{act}}$ is a nodal Lagrangian \ac{fe} space, there is a one-to-one map between shape functions, nodes and \acp{dof}. For each node in the mesh, we can define its owner as the lowest dimensional n-face (e.g., vertex, edge, face, cell) that contains it, to create a map $\mathcal{O}_h^{\mathrm{dof}\to \mathrm{nf}}$. We define the set of ill-posed \acp{dof} as the ones owned by cut/external n-faces, i.e., $\mathcal{C}_{h}^{\mathrm{cut}} \setminus \mathcal{C}_{h}^{\mathrm{in}}$. The rationale for this definition is the fact that the shape functions associated to these \acp{dof} are the ones that can have an arbitrarily small support on $\Omega$.

The definition of the discrete extension operator in the \ac{agfem} requires to define an ownership map $\mathcal{O}_h^{\mathrm{nf \to ag}}: \mathcal{C}_{h}^{\mathrm{cut}} \setminus \mathcal{C}_{h}^{\mathrm{in}} \rightarrow \mathcal{T}_{h}^{\partial,\mathrm{ag}}$  from ill-posed cut/external n-faces to aggregates. This mapping is not unique for inter-aggregate \acp{dof} and can be arbitrarily chosen. On the other hand, each aggregate has a unique root cell in $\mathcal{T}_{h}^{\mathrm{in}}$. Composing all these maps, we end up with an ill-posed to root cell map $\mathcal{O}_h^{\mathrm{dof} \to \mathrm{in}}: \mathcal{C}_{h}^{\mathrm{cut}} \setminus \mathcal{C}_{h}^{\mathrm{in}} \rightarrow \mathcal{T}_{h}^{\mathrm{in}}$.

Let $\mathcal{O}_{h}^{\mathrm{nf} \to \mathrm{dof}}$ be the inverse of  $\mathcal{O}_{h}^{\mathrm{dof} \to \mathrm{nf}}$, i.e., the map that returns the \acp{dof} owned by an n-face in $\mathcal{C}_{h}^{\mathrm{act}}$. We also need a \emph{closed} version $\overline{\mathcal{O}}_{h}^{\mathrm{nf} \to \mathrm{dof}}$ of this ownership map, which given an n-face $C$ in $\mathcal{C}_{h}^{\mathrm{act}}$ returns the owned \acp{dof} of all n-faces $C' \in \mathcal{C}_{h}^{\mathrm{act}}$ in the closure of $C$, $C' \subset C$. $\overline{\mathcal{O}}_{h}^{\mathrm{nf} \to \mathrm{dof}}$ is the map that describes the \emph{locality} of \ac{fe} methods. The only \acp{dof} that are \emph{active} in a cell $T \in \mathcal{T}_{h}^{\mathrm{act}}$ are the ones in $\overline{\mathcal{O}}_{h}^{\mathrm{nf} \to \mathrm{dof}}(T)$. Analogously, only the shape functions associated to these \acp{dof} have support on $T$.

We are in position to define the discrete extension operator as follows. An ill-posed \ac{dof} in $\mathcal{C}_{h}^{\mathrm{cut}} \setminus \mathcal{C}_{h}^{\mathrm{in}}$ is computed as a linear combination of the well-posed \acp{dof} in the closure of the root cell that owns it. Using the notation introduced so far, these \acp{dof} are the ones in $\mathcal{O}_{h}^{\mathrm{ipd} \to \mathrm{wpd}} \doteq \overline{\mathcal{O}}_{h}^{\mathrm{nf} \to \mathrm{dof}} \circ \mathcal{O}_{h}^{\mathrm{dof} \to \mathrm{in}}$. In particular, the value of $\sigma^\alpha \in \mathcal{C}_{h}^{\mathrm{cut}} \setminus \mathcal{C}_{h}^{\mathrm{in}}$ is computed as follows
\begin{equation}\label{eq:agfem-constraints}
  \sigma^{\alpha}(\cdot) \ = 
  \sum_{\sigma^{\beta} \in \mathcal{O}_{h}^{\mathrm{ipd}\to \mathrm{wpd}}(\alpha)} 
  \sigma^{\alpha}(\phi^{\beta}) \sigma^{\beta}(\cdot) \ = 
  \sum_{\sigma^{\beta} \in \mathcal{O}_{h}^{\mathrm{ipd}\to \mathrm{wpd}}(\alpha)} 
  \phi^{\beta}(\boldsymbol{x}^{\alpha}) \sigma^{\beta}(\cdot). 
\end{equation}
We can readily check that this expression extends the well-posed \ac{dof} values on interior cells to the ill-posed \ac{dof} values that only belong to cut cells. The application of~\eqref{eq:agfem-constraints} to a \ac{fe} function in $\mathcal{V}_{h}^{\mathrm{in}}$ provides the sought-after discrete extension operator $\mathcal{E}_{h}^{\mathrm{ag}}$ and the \ac{agfe} space $\mathcal{V}_{h}^{\mathrm{ag}}$.\footnote{One could consider alternative expressions for the constraints. E.g., one could skip the aggregate owner map $\mathcal{O}_h^{\mathrm{dof} \to \mathrm{ag}}$ and simply compute an average of the constraints from all aggregates that contain the ill-posed \ac{dof}. However, this choice would have a very negative impact on the sparsity pattern of the resulting linear system and was originally discarded.}

The implementation of \ac{agfem} simply requires the imposition of the constraints in~\eqref{eq:agfem-constraints}. These constraints are cell-local (much simpler than the ones in $h$-adaptive mesh refinement). On the other hand, the method does not require any modification of the forms in~\eqref{eq:poisson-weak}-(\ref{eq:elasticity-weak}) . Instead, it is the \ac{fe} space the one that changes. The method reads: find $u_h \in \mathcal{V}_{h}^{\mathrm{ag}}$ such that $a(u_h,v_h) = b(v_h)$ for any $v_h \in \mathcal{V}_{h}^{\mathrm{ag}}$. The key properties of the \ac{agfem} are:
\begin{equation}\label{eq:agfem-extended-stability}
  \|\boldsymbol{\nabla}u_h\|^{2}_{\boldsymbol{L}^2(\Omega_{h}^{\mathrm{act}})} \lesssim \|\boldsymbol{\nabla}u_h\|^{2}_{\boldsymbol{L}^2(\Omega)}, \qquad   
  \|u_h\|^{2}_{{L}^2(\Omega_{h}^{\mathrm{act}})} \lesssim \| u_h\|^{2}_{{L}^2(\Omega)}, \qquad \forall u_h \in \mathcal{V}_h^{\mathrm{ag}},
\end{equation}
It is also important to note that, by construction, $\mathcal{V}_{h}^{\mathrm{ag}}(U)$ for $U \in \mathcal{T}_{h}^{\mathrm{ag}}$ is a subspace of $\mathcal{P}_p(U)$ (for simplices) of $\mathcal{Q}_p(U)$ for hexahedra, solving the issues in the $\gamma \to \infty$ limit of \ac{fgp} and \ac{bgp}. As a result, the method preserves the convergence properties of $\mathcal{V}_{h}^{\mathrm{act}}$. We summarise these observations in the following definition.

\begin{definition}\label{def:discrete-extension-operator}
  Let $0 \leq s \leq n \leq m$,  where $m$ is the order of $\mathcal{V}_{h}^{\mathrm{in}}$. A suitable discrete extension operator $\mathcal{E}_{h}^{\mathrm{ag}}: \mathcal{V}_{h}^{\mathrm{in}} \rightarrow  \subset \mathcal{V}_{h}^{\mathrm{act}}$ must satisfy the following properties: 
  \begin{itemize}
         \item[(i)] Continuity: 
    \begin{equation}\label{eq:discrete-extension-continuity}
		{\| \mathcal{E}_{h}^{\mathrm{ag}}(v_h) \|_{L^2(\Omega_{h}^{\mathrm{act}})} \lesssim \|v_h\|_{L^2(\Omega_{h}^{\mathrm{in}})},} \qquad \| \boldsymbol{\nabla}\mathcal{E}_{h}^{\mathrm{ag}}(v_h) \|_{\boldsymbol{L}^2(\Omega_{h}^{\mathrm{act}})} \lesssim \|\boldsymbol{\nabla}  v_h\|_{\boldsymbol{L}^2(\Omega_{h}^{\mathrm{in}})}, \qquad \forall v_h \in \mathcal{V}_{h}^{\mathrm{in}}.
    \end{equation}
    \item[(ii)] Approximability:
\begin{equation}\label{eq:discrete-extension-approximability}
  \underset{v_h \in \mathcal{V}_{h}^{\mathrm{in}}}{\mathrm{inf}} \| u - \mathcal{E}_{h}^{\mathrm{ag}}(v_h) \|_{H^s(\Omega)} \lesssim h^{n-s+1} \| u \|_{H^{n+1}(\Omega)}, \qquad \forall u  \in H^{n+1}(\Omega).
\end{equation}
\end{itemize}
The image $\mathcal{V}_{h}^{\mathrm{ag}} \doteq \mathrm{Im}(\mathcal{E}_{h}^{\mathrm{ag}}) \subset \mathcal{V}_{h}^{\mathrm{act}}$ is a suitable \ac{agfe} space. 
\end{definition}

\begin{proposition}\label{prop:well-posedness-deo}
  Let $\mathcal{E}_{h}^{\mathrm{ag}}$ satisfy Def.~\ref{def:discrete-extension-operator}. Let $\mathcal{V}^{\mathrm{ag}}(h) \doteq H^2(\Omega) + \mathcal{V}_{h}^{\mathrm{ag}}$, endowed
  with the norm
  \begin{equation}\label{eq:norm-agfem}
    {\tnor{v}^{2}_{\mathcal{V}^{\mathrm{ag}}(h)} \doteq \|\boldsymbol{\nabla}v\|^{2}_{\boldsymbol{L}^2(\Omega)} +  \| \tau^{\frac{1}{2}} v \|_{L^2(\Gamma_{\mathrm{D}})}^2 
    + \sum_{T \in \mathcal{T}_{h}^{\mathrm{act}}}^{} h_T^{2} \| v \|_{H^2(T)}^{2}, \qquad \forall v \in \mathcal{V}^{\mathrm{ag}}(h).}
    \end{equation}
  It holds:
  \begin{align}\label{eq:deo-coercive-continuous}
    a_h(u_h,u_h) \gtrsim \| u_h \|^2_{\mathcal{V}^{\mathrm{ag}}(h)}, \qquad
    a_h(u,v_h) \lesssim \|u\|_{\mathcal{V}^{\mathrm{ag}}(h)} \|v_h\|_{\mathcal{V}^{\mathrm{ag}}(h)},
    \end{align} 
    for any $u_h, v_h \in  \mathcal{V}_{h}^{\mathrm{ag}}$, $u \in \mathcal{V}^{\mathrm{ag}}(h)$. Thus, there is a unique 
  \begin{equation}\label{eq:deo-weak-poisson}
    u_h \in \mathcal{V}_h^{\rm ag} \ : \ a_h(u_h,v_h) = b(v_h), \quad \forall v_h \in \mathcal{V}_{h}^{\mathrm{ag}}.
  \end{equation}
  Furthermore, the condition number $\kappa$ of the resulting linear system when using a standard Lagrangian basis for $\mathcal{V}_{h}^{\mathrm{in}}$  holds $\kappa \lesssim h^{-2}$. 
\end{proposition}
\begin{proof}
  In order to prove (\ref{eq:deo-coercive-continuous}), the Nitsche terms can readily be bounded as in Prop.~\ref{prop:well-posedness-gp} using (\ref{eq:discrete-extension-continuity}). The condition number proof follows the same lines as the one for the \ac{gp} in Prop.~\ref{prop:well-posedness-gp} (see \cite{Badia2018}).
\end{proof}
  
  \begin{proposition}\label{prop:error-estimates-deo}
  Let $m$ be the order of $\mathcal{V}_{h}^{\mathrm{act}}$. If $u \in H^{n+1}(\Omega)$, $m \geq n \geq 1$, is the solution of~(\ref{eq:poisson-weak}) or (\ref{eq:elasticity-weak})  and $u_h \in \mathcal{V}^{\mathrm{ag}}_h$ is the
	solution of \eqref{eq:deo-weak-poisson},  then:
	\begin{equation}
		\|u-u_h\|_{\mathcal{V}(h)} \lesssim h^{n} 
		| u |_{H^{n+1}(\Omega)}, \qquad 
		\|{u-u_h}\|_{L^2(\Omega)} \lesssim h^{n+1} 
		| u |_{H^{n+1}(\Omega)}.
	\end{equation}
\end{proposition}
\begin{proof}
The convergence analysis is standard and relies on the approximability property~(\ref{eq:discrete-extension-approximability}).     
\end{proof}

\begin{proposition}\label{prop:aggregation-suitable}
  The discrete extension operator $\mathcal{E}_{h}^{\mathrm{ag}}$ constructed with expression~(\ref{eq:agfem-constraints}) satisfies the conditions in Def.~\ref{def:discrete-extension-operator}.       
\end{proposition}
\begin{proof}
  The continuity in ~(\ref{eq:discrete-extension-continuity}) relies on the fact that the constrained \ac{dof} values in~(\ref{eq:agfem-constraints}) can be bounded by interior \acp{dof} under the assumption that the aggregate sizes are proportional to the root cell size. The complete proof can be found in \cite[Corollary 5.3]{Badia2018} for a general \ac{agfe} space, which can be either $\mathcal{V}_{h}^{\mathrm{in}}$ or the discontinuous FE space of its gradients. The proof of the approximability property~(\ref{eq:discrete-extension-approximability}) can be found in \cite[Lemma 5.10]{Badia2018}. 
\end{proof}

\subsection{Ghost penalty with discrete extension}\label{sub:weak-agfem}

In this section, we propose a novel scheme, which combines the discrete extension operator defined above for the \ac{agfem} and the \ac{gp} stabilisation ideas. This development is motivated by the \emph{locking} phenomenon of \ac{gp} methods as $\gamma \to \infty$. The idea is to penalise the distance between the solution in the (unconstrained) $\mathcal{V}_{h}^{\mathrm{act}}$ space and the one in $\mathcal{V}_{h}^{\mathrm{ag}}$. 

In Section~\ref{sec:agfem}, we have defined the discrete extension operator $\mathcal{E}^{\mathrm{ag}}_h: \mathcal{V}_{h}^{\mathrm{in}} \longrightarrow \mathcal{V}_{h}^{\mathrm{ag}}$. Let us define the projector $\mathcal{P}^{\mathrm{ag}}_h: \mathcal{V}_{h}^{\mathrm{act}} \longrightarrow \mathcal{V}_{h}^{\mathrm{ag}}$ that takes $v_h \in \mathcal{V}_{h}^{\mathrm{act}}$, computes its restriction to the interior $v_h|_{\Omega_{h}^{\mathrm{in}}} \in \mathcal{V}_{h}^{\mathrm{in}}$ and applies the extension $\mathcal{E}^{\mathrm{ag}}_h$, i.e., $\mathcal{P}_{h}^{\mathrm{ag}}(v_h) \doteq \mathcal{E}^{\mathrm{ag}}_h \left(  v_h|_{\Omega_{h}^{\mathrm{in}}} \right)$. Next, we define the following stabilisation term:
\begin{equation}\label{eq:weak-l2-agfem}
  s_h(u_h,v_h) \doteq \sum_{U \in \mathcal{T}_{h}^{\mathrm{ag}}}^{} \left( \gamma h_U^{-2} \left(   u_h - \mathcal{P}_{h}^{\mathrm{ag}}(u_h) \right), v_h - \mathcal{P}_{h}^{\mathrm{ag}}(v_h) \right)_{U^*},
\end{equation}
where $U^*$ can be chosen to be $U$ or $U \setminus \Omega$ (i.e., only adding stability in the exterior part). We refer to method (\ref{eq:weak-l2-agfem}) as \textbf{\ac{wagL2}}. We note that this term can be computed cell-wise, in the spirit of so-called local projection stabilisation methods \cite{Becker2001Dec,Badia2012Nov}. On the other hand, the term vanishes at the root cell of the aggregate by construction, i.e., it is only active on $\Omega_{h}^\mathrm{cut}$.  Alternatively, one could consider an $L^2$-projection at each aggregate $U$, as in the \ac{bgp}. It is also obvious to check that this method converges to the \ac{agfem} in the limit $\gamma \to \infty$. 
Alternatively, we can consider a stabilisation that relies on the $H^1$ product:
\begin{equation}\label{eq:weak-h1-agfem}
  s_h(u_h,v_h) \doteq \sum_{U \in \mathcal{T}_{h}^{\mathrm{ag}}}^{} \left( \gamma \boldsymbol{\nabla}\left(   u_h - \mathcal{P}_{h}^{\mathrm{ag}}(u_h) \right), \boldsymbol{\nabla} \left(  v_h - \mathcal{P}_{h}^{\mathrm{ag}}(v_h)  \right)\right)_{U^*},
\end{equation}
that is, a \textbf{\ac{wagH1}}, where again $U^*$ can be $U$ or $U \setminus \Omega$. 
Its implementation is as easy as the previous one, but does not require to compute a characteristic mesh size $h_U$.

\begin{proposition}\label{prop:weak-l2h1-gp-suitable}
  The expressions of $s_h$ in~(\ref{eq:weak-l2-agfem}) and (\ref{eq:weak-h1-agfem}) are suitable \ac{gp} stabilisation that satisfies the conditions (\ref{eq:gp-extended-stability})-(\ref{eq:gp-weak-consistency}) in Def.~\ref{def:suitable-ghost-penalty}.  
\end{proposition}

\begin{proof}
Let us start proving the extended stability~(\ref{eq:gp-extended-stability}) as follows. Using the stability of the discrete extension operator~(\ref{eq:discrete-extension-continuity}) and the triangle inequality, we get:
\begin{align}\label{eq:pr:weak-agfem-stability}
  \|\boldsymbol{\nabla}u_h\|_{\boldsymbol{L}_{2}(\Omega_{h}^{\mathrm{act}})}^{2} 
  & \leq \|\boldsymbol{\nabla}u_h\|_{\boldsymbol{L}_{2}(\Omega)}^{2} + \|\boldsymbol{\nabla} u_h - \boldsymbol{\nabla}\mathcal{P}_{h}^{\mathrm{ag}}(u_h)  \|_{\boldsymbol{L}_{2}(\Omega_{h}^{\mathrm{cut}} \setminus \Omega)}^{2} 
  + \|\boldsymbol{\nabla} \mathcal{P}_{h}^{\mathrm{ag}}(u_h) \|_{\boldsymbol{L}_{2}(\Omega_{h}^{\mathrm{cut}} \setminus \Omega)}^{2} \\ & \leq \|\boldsymbol{\nabla} u_h - \boldsymbol{\nabla}\mathcal{P}_{h}^{\mathrm{ag}}(u_h)  \|_{\boldsymbol{L}_{2}(\Omega_{h}^{\mathrm{cut}} \setminus \Omega)}^{2} + C \|\boldsymbol{\nabla}u_h\|_{\boldsymbol{L}_{2}(\Omega_{h}^{\mathrm{in}})}^{2},  
\end{align} 
for any $u_h \in \mathcal{V}_{h}^{\mathrm{act}}$. The result for $s_h$ in~(\ref{eq:weak-l2-agfem}) can be obtained by bounding the second term in~(\ref{eq:pr:weak-agfem-stability}) using the quasi-uniformity of the background mesh $\mathcal{T}_{h}^{\mathrm{act}}$, the standard inverse inequality in $\mathcal{T}_{h}^{\mathrm{act}}$ and the assumption that aggregate size $h_U$ is proportional to the root cell size times a constant that does not depend on $h$.

For weak consistency, we need to extend the definition of $s_h$ for $u \in H^{n+1}(\Omega_h^{\rm act})$. To this end, we invoke the standard Scott-Zhang interpolant $\pi_h^{\rm sz}: H^{n+1}(\Omega_h^{\rm act}) \longrightarrow \mathcal{V}_{h}^{\mathrm{act}}$, using the definition in~\cite{scott1990finite}. We let $\tilde{\mathcal{P}}^{\mathrm{ag}}_h = \mathcal{P}^{\mathrm{ag}}_h \circ \pi_h^{\mathrm{sz}}$; obviously $\tilde{\mathcal{P}}^{\mathrm{ag}}_h(u_h) = \mathcal{P}^{\mathrm{ag}}_h(u_h)$, if $u_h \in \mathcal{V}_{h}^{\mathrm{act}}$, since $\pi_{h}^{\mathrm{sz}}$ is a projection. Hence, weak consistency~(\ref{eq:gp-weak-consistency}) of the proposed methods can readily be proved using the optimal approximation properties of $\mathcal{V}_{h}^{\mathrm{ag}}$ in~(\ref{eq:discrete-extension-approximability}).\footnote{We note that \ac{agfem} does not require to define an extension of the continuous solution and the approximability property is only needed in $\Omega$. In any case, the proof in \cite[Lemma 5.10]{Badia2018} also applies for $\Omega_{h}^{\mathrm{act}}$.}
\begin{equation}\label{eq:pr:weak-agfem-convergence-h1}
  {\left| s_h(\mathcal{E}^{n+1}(u),v_h) \right| \lesssim | \mathcal{E}^{n+1}(u)  - \tilde{\mathcal{P}}_{h}^{\mathrm{ag}}(\mathcal{E}^{n+1}(u)) |_{{H}_{}^{1}(\Omega_{h}^{\mathrm{act}})} \| \boldsymbol{\nabla} v_h\|_{\boldsymbol{L}^2(\Omega_{h}^{\mathrm{act}})}} \lesssim h^{n} | u |_{H^{n+1}(\Omega)} \tnor{v_h}_{V(h)},
\end{equation}
for $s_h$ in~(\ref{eq:weak-h1-agfem})  and 
\begin{equation}\label{eq:pr:weak-agfem-convergence-l2}
  {\left| s_h(\mathcal{E}^{n+1}(u),v_h) \right| \lesssim \sum_{U \in \mathcal{T}_{h}^{\mathrm{ag}}}^{} h^{-2} \|  \mathcal{E}^{n+1}(u) - \tilde{\mathcal{P}}_{h}^{\mathrm{ag}}(\mathcal{E}^{n+1}(u)) \|^2_{{L}_{}^{2}(U)} \| v_h \|_{L^2(U)}} \lesssim h^{n} | u |_{H^{n+1}(\Omega)} \tnor{v_h}_{V(h)} 
\end{equation}
for $s_h$ in~(\ref{eq:weak-l2-agfem}). The proof of continuity~(\ref{eq:gp-continuity}) relies on the $H^1$-stability of the Scott-Zhang interpolant and the discrete extension operator in (\ref{eq:discrete-extension-continuity}). 
\end{proof}

We observe that, for $U^* \doteq U \setminus \Omega$, \ac{wagH1} in~(\ref{eq:weak-h1-agfem}) can also be understood as an improvement of the so-called finite cell method~\cite{Nguyen2017} that is both convergent and robust. Instead of substracting a projection, the finite cell method adds a grad-grad term on $\Omega_{h}^{\mathrm{act}}\setminus \Omega$, premultiplied by a numerical parameter. When the parameter is large, convergence is deteriorated. When the parameter is small, robustness is affected. Subtracting the projection, these problems are solved, since the resulting method enjoys weak consistency.

\section{Implementation aspects}\label{sec:implementation-aspects}

The implementation of the previous methods have some non-standard requirements that are not present in body-fitted \ac{fe} codes. First, there are some geometrical functions that have to be implemented. 

All methods, but the \ac{fgp}~(\ref{eq:face-penalty}), require the computation of an aggregated mesh with the properties described in Sec.~\ref{sec:cell-aggregation}. Furthermore, all the \ac{gp} methods, but the original bulk penalty in~(\ref{eq:bulk-penalty}), require the $n$-face to aggregate map described in Sec.~\ref{sec:agfem}. Besides, the \ac{gp} method needs to identify the set of faces that are in touch with cut cells and intersect the domain $\Omega$. In any case, the implementation of all these steps is quite straightforward and can be easily implemented in distributed memory machines (see \cite{Verdugo2019,Badia2020Jun}).

The implementation of \ac{fgp} involves jumps of normal derivatives up to the order of the \ac{fe} space. This functionality is uncommon in \ac{fe} software (see, e.g., \cite{badia-fempar,dealii2019design,Badia2020Aug}). The implementation of the \ac{bgp} methods requires the computation of aggregate-wise $L^2$ projections, which is a non-obvious task in \ac{fe} software, since it breaks the standard procedure, namely to compute cell-local matrices, assemble the local matrices in a global sparse array and solve a global linear system.

In order to avoid this implementation issue, we can consider an interpolator in the spirit of \ac{wagL2}. We replace the aggregate-wise $L^2$ projection in~(\ref{eq:bulk-penalty}) by the interpolation of the \ac{fe} function on the root cell over the aggregate. Let $\mathcal{V}_{h}^{\#,-}$ be the \ac{dg} counterpart of $\mathcal{V}_{h}^{\#}$, for $\# \in \{ \mathrm{in},\mathrm{cut}\}$. We can define the \ac{dg} discrete extension operator $\mathcal{E}_{h}^{\mathrm{ag},-}$ using the definition above for the \ac{dg} spaces. We note that the construction of the extension is much simpler in this case, since all \acp{dof} belong to cells. Since $\mathcal{V}_{h}^{\mathrm{in}} \subset \mathcal{V}_{h}^{\mathrm{in},-}$, we readily obtain  $\mathcal{E}_{h}^{\mathrm{ag},-}: \mathcal{V}_{h}^{\mathrm{in}} \longrightarrow \mathcal{V}_{h}^{\mathrm{act},-}$. Analogously, we can define the projector $\mathcal{P}^{\mathrm{ag},-}_h: \mathcal{V}_{h}^{\mathrm{act}} \longrightarrow \mathcal{V}_{h}^{\mathrm{ag},-}$ that takes $v_h \in \mathcal{V}_{h}^{\mathrm{act}}$, computes its restriction to the interior $v_h|_{\Omega_{h}^{\mathrm{in}}} \in \mathcal{V}_{h}^{\mathrm{in}}$ and applies the \ac{dg}  extension $\mathcal{E}^{\mathrm{ag},-}_h$, i.e., $\mathcal{P}_{h}^{\mathrm{ag},-}(v_h) \doteq \mathcal{E}^{\mathrm{ag},-}_h \left(  v_h|_{\Omega_{h}^{\mathrm{in}}} \right)$. Now, we can define an alternative \textbf{\ac{bgpi}} as follows:
\begin{equation}\label{eq:bulk-interp-agfem}
  s_h(u_h,v_h) \doteq \sum_{U \in \mathcal{T}_{h}^{\mathrm{ag}}}^{} \left( \gamma h_U^{-2} \left(   u_h - \mathcal{P}_{h}^{\mathrm{ag},-}(u_h) \right), v_h - \mathcal{P}_{h}^{\mathrm{ag},-}(v_h) \right)_{U}.
\end{equation}
The proof for the bulk penalty method~(\ref{eq:bulk-penalty}) readily applies for this version, since the aggregate-wise interpolant share the required continuity and error bounds of the aggregate-wise $L^2$ projection. On the other hand, it is penalising the component that does not belong to the same space, thus it has the same locking problems. Since the interpolant-based version simplifies the implementation, this is the one to be used in Sect.~\ref{sec:numerical_experiments}.

The \ac{gp} and \ac{agfe} methods make use of an ill-posed to well-posed \ac{dof} that can be readily computed with the geometrical machinery described above and the local-to-global \ac{dof} map. These methods require the implementation of a discrete extension operator. In general, it can be stored as the set of linear constraints in~(\ref{eq:agfem-constraints}) that can be implemented cell-wise (traversing all root cells). The computation of this constraints is much simpler than in $h$ and $hp$-adaptivity (see, e.g., \cite{badia-fempar,dealii2019design,Alnaes2015Dec}). It requires to evaluate the shape functions of the root cell in the nodes of the ill-posed \acp{dof} (see~(\ref{eq:agfem-constraints})). In most \ac{fe} codes, it requires the implementation of the geometrical map inverse, in order to map points from the physical to the reference space, which is not standard in general. {This problem is trivial for Cartesian background meshes, which is arguably the most natural and efficient choice.}

For the (strong) \ac{agfem}, one must impose these constraints in the assembly process. This is readily available (for more complex constraints) in any code that allows adaptive mesh refinement with hanging nodes (see, e.g., \cite{badia-fempar}). For the novel weak versions \ac{wagL2} and \ac{wagH1}, the assembly process does not involve the imposition of  constraints. However, the local matrix in cut cells does not only include the local \acp{dof} but also the local \acp{dof} of the root cell. E.g., for \ac{wagH1}, we have to assemble the following entries: 
\begin{align}
&  \left( \boldsymbol{\nabla}\phi_i, \boldsymbol{\nabla}\phi_j \right)_T, \quad && {i,j \in\overline{\mathcal{O}}_{h}^{\mathrm{nf}\to \mathrm{dof}}(T)}, \\ 
& \left( \boldsymbol{\nabla}\phi_i, \mathcal{E}_{h}^{\mathrm{ag}}(\boldsymbol{\nabla}\phi_k) \right)_T, \quad 
&& i \in \overline{\mathcal{O}}_{h}^{\mathrm{nf}\to \mathrm{dof}}(T), k \in \mathcal{O}_{h}^{\mathrm{\mathrm{ipd} \to \mathrm{wpd}}} \circ \overline{\mathcal{O}}_{h}^{\mathrm{nf}\to \mathrm{dof}}(T),
\end{align}
and the transpose of the second term, for any $T \in \mathcal{T}_h^{\mathrm{cut}}$. 

\section{Numerical experiments}\label{sec:numerical_experiments}

\subsection{Methods and parameter space} 
\label{ssub:Methods and parameter space}

We collect in Table~\ref{tab:methods} all the numerical methods that will be analysed in this work, in terms of the stabilisation term $s_h$ and the \ac{fe} space being used. We consider first the \ac{fgp} method defined in (\ref{eq:face-penalty}). We include the expression of $s_h$ for linear elements. Higher order approximations would require to evaluate inter-facet jumps of higher order derivatives that are not available in the numerical framework \texttt{Gridap} \cite{Badia2020Aug}. We also consider the modification of this method~(\ref{eq:weak-face-penalty}), that only penalises intra-aggregate facets (\ac{agp}). As indicated in Sect.~\ref{sec:implementation-aspects}, we consider the interpolation-based \ac{bgpi} method~(\ref{eq:bulk-interp-agfem}). We also analyse the results for both versions of the weak \ac{agfem}: \ac{wagL2} and \ac{wagH1}, i.e., using the $L^2$ product~(\ref{eq:weak-l2-agfem}) and $H^1$ product~(\ref{eq:weak-h1-agfem}). In both cases, we set $U^* = U$. All these methods make use of the standard \ac{fe} space $\mathcal{V}_{h}^{\mathrm{act}}$ on the active mesh $\mathcal{T}_{h}^{\mathrm{act}}$. Finally, we compare all these stabilised formulations against the \textbf{\ac{sag}} in Sect.~\ref{sec:agfem}, which makes use of the \ac{fe} space $\mathcal{V}_{h}^{\mathrm{ag}}$.

\begin{table}[ht!]
	\centering
	\begin{small}
		\begin{tabular}{llll}
			\toprule
			Acronym & Stabilisation form $s_h(u_h,v_h)$ & \ac{fe} space & Legend \\
			\midrule
			\ac{fgp}                       & $\sum_{F \in \mathcal{F}_h^{\mathrm{gh},\mathrm{cut}}}^{} \left( \gamma h_F \llbracket \partial_n u_h \rrbracket , \llbracket \partial_n v_h \rrbracket  \right)_F$ & $\mathcal{V}_{h}^{\mathrm{act}}$  & \includegraphics[width=0.07\textwidth]{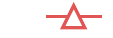} \\
			\ac{agp}                       & $\sum_{F \in \mathcal{F}_h^{\mathrm{gh},\mathrm{ag}}}^{} \left( \gamma h_F \llbracket \partial_n u_h \rrbracket , \llbracket \partial_n v_h \rrbracket  \right)_F$ & $\mathcal{V}_{h}^{\mathrm{act}}$  & \includegraphics[width=0.07\textwidth]{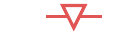} \\
			\ac{bgpi}                       & $\sum_{U \in \mathcal{T}_{h}^{\mathrm{ag}}}^{} \left( \gamma h_U^{-2} \left(   u_h - \mathcal{P}_{h}^{\mathrm{ag},-}(u_h) \right), v_h - \mathcal{P}_{h}^{\mathrm{ag},-}(v_h) \right)_{U}$ & $\mathcal{V}_{h}^{\mathrm{act}}$ & \includegraphics[width=0.07\textwidth]{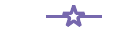}\\
			\ac{wagL2}                 & $\sum_{U \in \mathcal{T}_{h}^{\mathrm{ag}}}^{} \left( \gamma h_U^{-2} \left(   u_h - \mathcal{P}_{h}^{\mathrm{ag}}(u_h) \right), v_h - \mathcal{P}_{h}^{\mathrm{ag}}(v_h) \right)_{U}$ & $\mathcal{V}_{h}^{\mathrm{act}}$ & \includegraphics[width=0.07\textwidth]{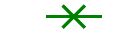} \\
			\ac{wagH1} & $\sum_{U \in \mathcal{T}_{h}^{\mathrm{ag}}}^{} \left( \gamma \boldsymbol{\nabla}\left(   u_h - \mathcal{P}_{h}^{\mathrm{ag}}(u_h) \right), \boldsymbol{\nabla} \left(  v_h - \mathcal{P}_{h}^{\mathrm{ag}}(v_h)  \right)\right)_{U}$ & $\mathcal{V}_{h}^{\mathrm{act}}$ & \includegraphics[width=0.07\textwidth]{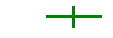} \\
			\ac{sag}                       & 0 & $\mathcal{V}_{h}^{\mathrm{ag}}$ & \includegraphics[width=0.07\textwidth]{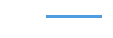} \\
			\bottomrule
		\end{tabular}
	\end{small}
	\caption{Summary of the different methods used and their corresponding symbols in the numerical examples.}
	\label{tab:methods}
\end{table}

We show in Table~\ref{tab:params} the parameter space being explored. We have reduced, as much as possible, this space, due to length constraints. We solve the Poisson problem~(\ref{eq:poisson-weak}) and linear elasticity~(\ref{eq:elasticity-weak}) discrete problems for two different sets of boundary conditions, namely weak imposition of Dirichlet boundary conditions on the whole boundary and Neumann and strong Dirichlet boundary conditions. This way, we can analyse the behaviour of the methods with respect to Nitsche terms and Neumann boundary conditions separately. {To compute errors analytically, we use polynomial manufactured solutions of the form $u(x,y) = (x+y)^{m+1}$ in 2D and $u(x,y,z) = (x+y+z)^{m+1}$ in 3D, with $m$ the order of the \ac{fe} space.}

\begin{table}[ht!]
	\centering
	\begin{small}
		\begin{tabular}{ll}
			\toprule
			Description & Considered methods/values \\
			\midrule
			Model problem & Poisson equation (Nitsche's formulation); \\ 
                    & Linear elasticity (Neumann and strong Dirichlet)
			\vspace{0.12cm} \\
			{Analytical solution} & {$u(x,y) = (x+y)^{m+1}$ (2D) and $u(x,y,z) = (x+y+z)^{m+1}$ (3D)} \\
					& {$m$ is the \ac{fe} interpolation order}
	  \vspace{0.12cm} \\
			Problem geometry & 2D: square; 3D: cube, tilted cube, sphere, 8-norm sphere 
			\vspace{0.12cm} \\
			Interpolation & $\mathcal{C}^0$ linear and quadratic \acp{fe} on simplicial meshes \vspace{0.12cm} \\
			Coef. in Nitsche's penalty term & $\beta = 10.0 \ m^2$ \\
			\bottomrule
		\end{tabular}
	\end{small}
	\caption{Summary of the main parameters and computational strategies used in
	the numerical examples.}
	\label{tab:params}
\end{table}

We have observed that the use of simplicial and hexahedral meshes leads to similar conclusions. So, we only show results for simplicial meshes below. We have considered the square in 2D and the cube, a tilted cube, the sphere and the 8-norm sphere in 3D. All geometries are represented in Figure~\ref{fig:geometries}. They are centred at the origin of coordinates. For weak Dirichlet tests, we simulate the whole geometry; while only the region in the first quadrant/octant for mixed boundary conditions. We have configured all geometries, in a similar way as in~\cite{de2018note}, to ensure that the small cut cell problem is present in all cases with $\min_{T \in \mathcal{T}_{h}^{\mathrm{act}}}\eta_T$ of the order of $10^{-8}$. In the case of the square, cube and 8-norm sphere, we also enforce sliver cuts. We have considered both linear and quadratic $\mathcal{C}^0$ Lagrangian \ac{fe} spaces. We pay particular attention to the penalty parameter $\gamma$; as we will see, some methods are highly sensitive to this value.  

\begin{figure}[!h]
	\centering
	\begin{subfigure}[t]{0.48\textwidth}
		\includegraphics[width=0.95\textwidth,trim=0.5cm 0.5cm 0.5cm 0.5cm,clip=true]{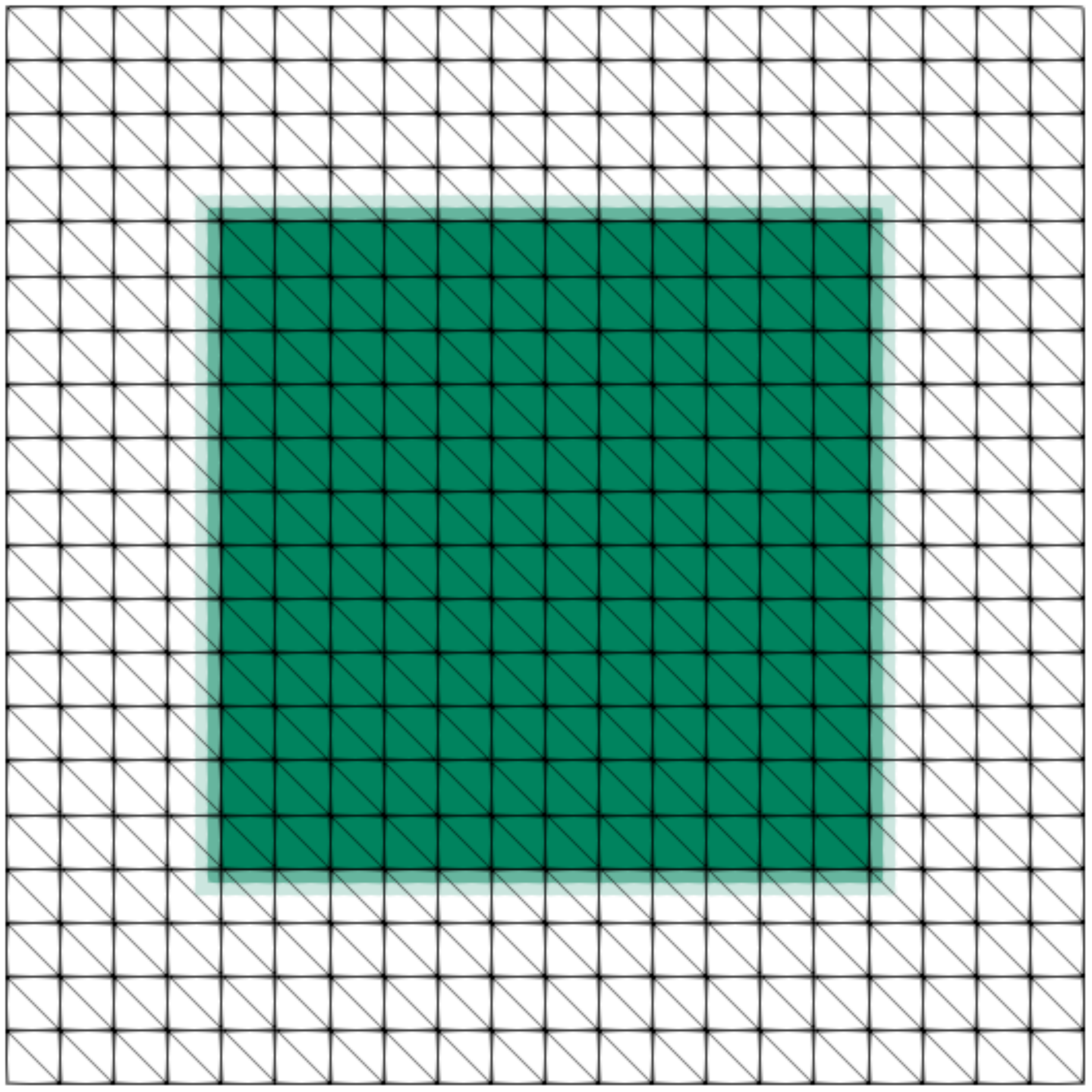}
	\caption{2D square and 3D cube on plane $y = 0$}\label{fig:geometries-a}
	\end{subfigure}
	\begin{subfigure}[t]{0.48\textwidth}
		\includegraphics[width=0.95\textwidth,trim=0.5cm 0.5cm 0.5cm 0.5cm,clip=true]{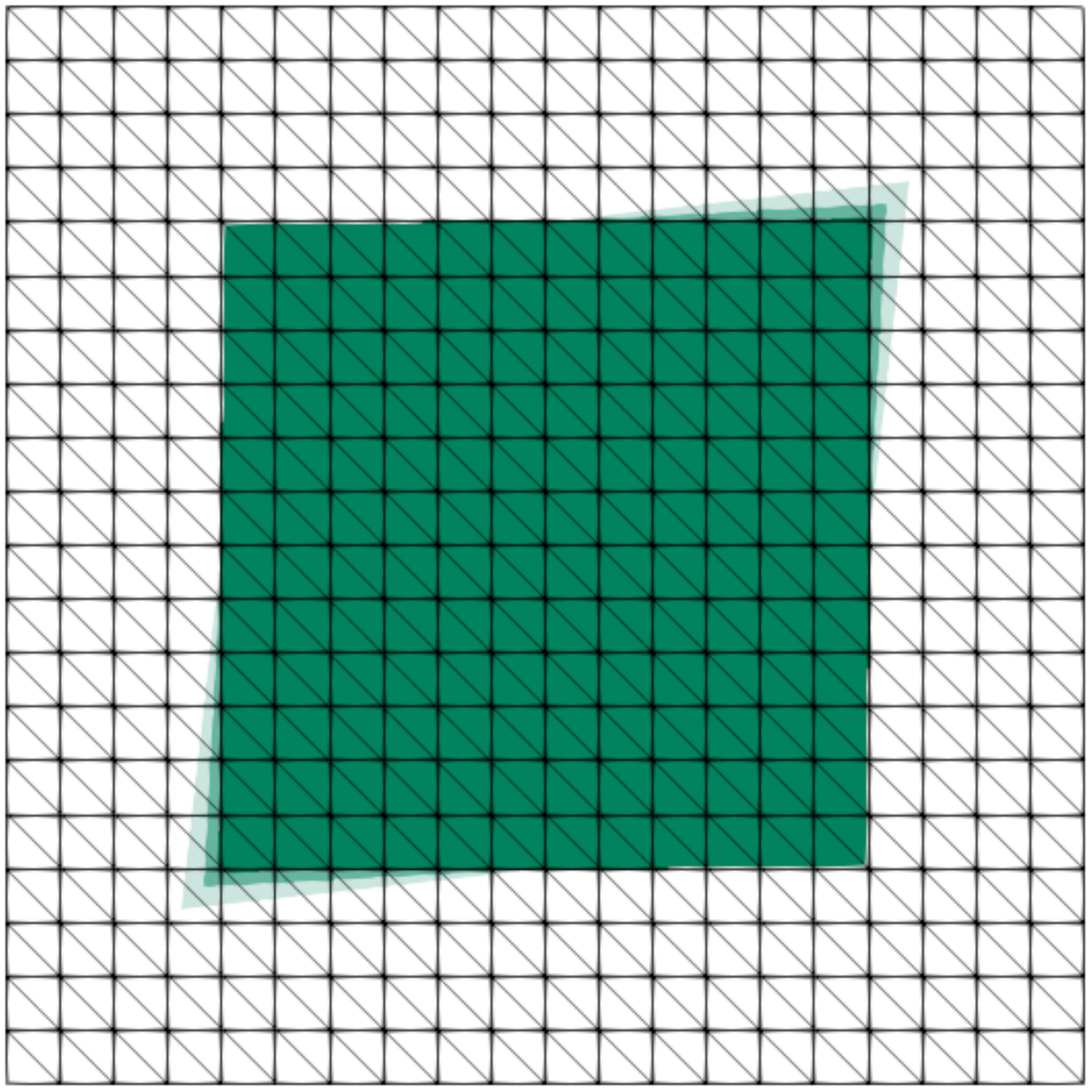}
    	\caption{3D tilted cube on plane $y = 0$}\label{fig:geometries-b}
	\end{subfigure} \\
	\vspace{0.6em}
	\begin{subfigure}[t]{0.48\textwidth}
		\includegraphics[width=0.95\textwidth,trim=0.5cm 0.5cm 0.5cm 0.5cm,clip=true]{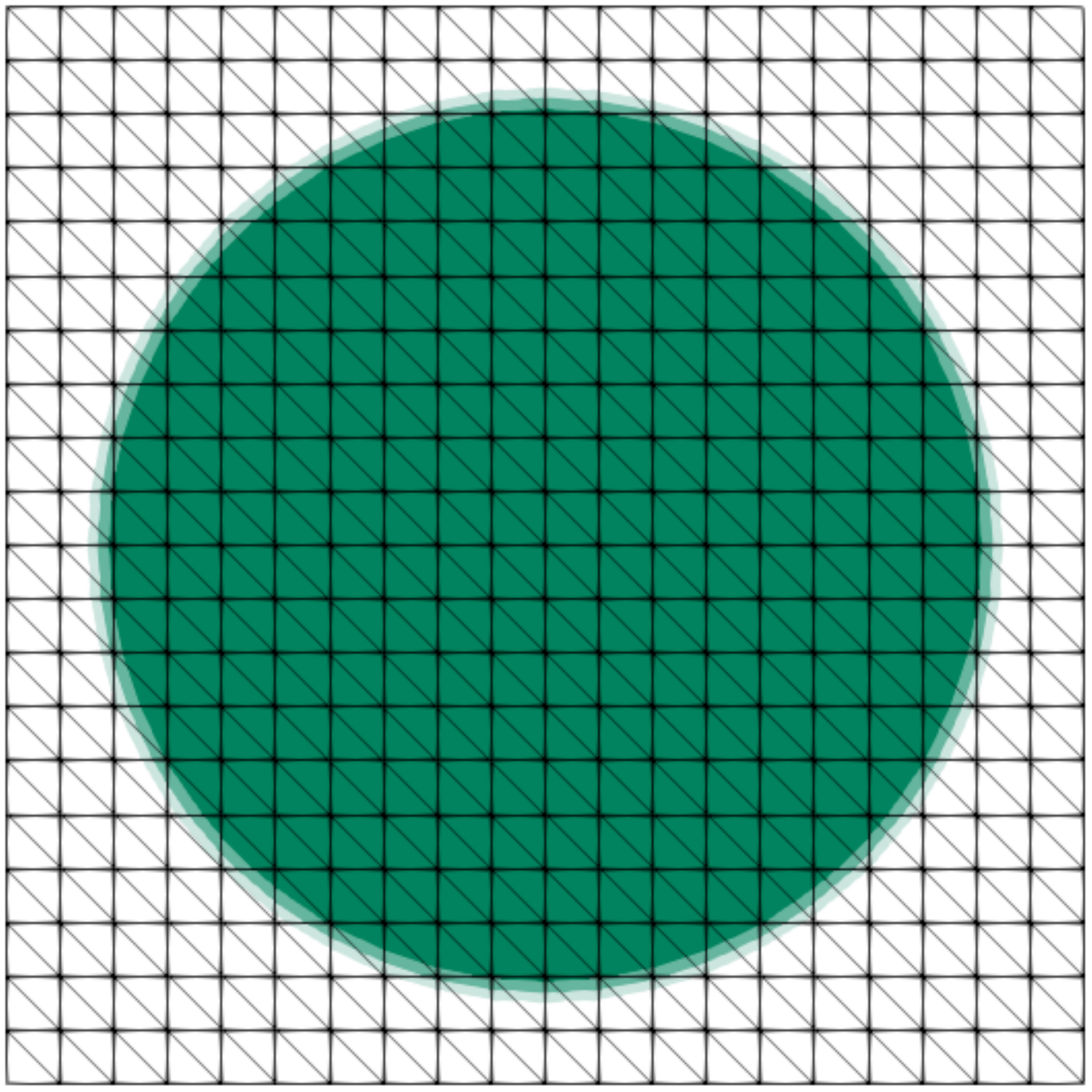}
		\caption{3D sphere on plane $y = 0$}
	\end{subfigure}
	\begin{subfigure}[t]{0.48\textwidth}
		\includegraphics[width=0.95\textwidth,trim=0.5cm 0.5cm 0.5cm 0.5cm,clip=true]{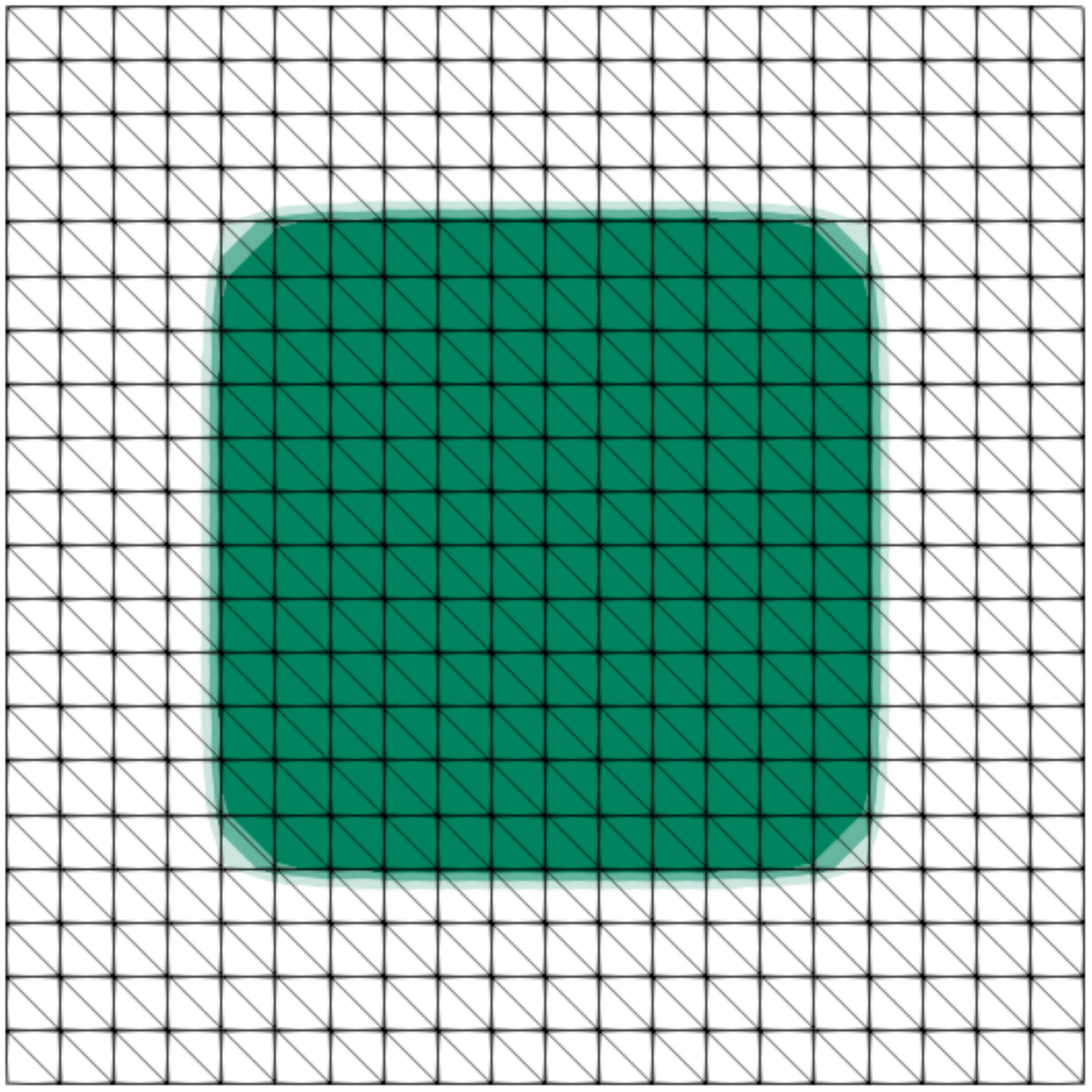}
    	\caption{3D $8$-norm sphere on plane $y = 0$}\label{fig:geometries-d}
	\end{subfigure}
	\caption[]{{Geometries considered in the sensitivity analysis. We superpose different configurations of each geometry, with increasing opacity, to expose how we have shrinked the geometries to force cut volumes of the order of $10^{-8}$. We refer the reader to~\cite{de2018note} for the level set descriptions of~\ref{fig:geometries-a},~\ref{fig:geometries-b} and~\ref{fig:geometries-d}.}}
	\label{fig:geometries}
\end{figure}

\subsection{Experimental environment} 
\label{sub:exp_env}

All the algorithms have been implemented in the \texttt{Gridap} open-source scientific software project \cite{Badia2020Aug}. \texttt{Gridap} is a novel framework for the implementation of grid-based algorithms for the discretisation of \acp{pde} written in the \texttt{Julia} programming language. \texttt{Gridap} has a user interface that resembles the whiteboard mathematical statement of the problem. The framework leverages the \texttt{Julia} just-in-time (JIT) compiler to generate high-performant code. \texttt{Gridap} is extensible and modular and has many available plugins. In particular, we have extensively used and extended the \texttt{GridapEmbedded} plugin, which provides all the mesh queries required in the implementation of the embedded methods under consideration, level set surface descriptions and constructive solid geometry. We use the \texttt{cond()} method provided by \texttt{Julia} to compute condition numbers. Condition numbers have been computed in the $1$-norm for efficiency reasons. It was not possible to compute the $2$-norm condition number for all cases with the available computational resources.

The numerical experiments have been carried out at the TITANI cluster of the Universitat Polit\`{e}cnica de Catalunya (Barcelona, Spain), and  NCI-Gadi \cite{NCIgadi}, hosted by the Australian National Computational Infrastructure Agency (NCI). NCI-Gadi is a petascale machine with 3024 nodes, each containing 2x 24-core Intel Xeon Scalable \textit{Cascade Lake} processors and 192 GB of RAM.  All nodes are interconnected via Mellanox Technologies' latest generation HDR InfiniBand technology.

\subsection{Sensitivity analysis} 
\label{sub:Experimental results}

We consider first log-log plots. They show values of the penalty parameter $\gamma = 10^{-\alpha}$ for $\alpha \in \left\{ -2, -1, 0, 1, 2, 4, 6, 8 \right\}$, in the X-axis, and the value of the condition number (in 1-norm) and the $L^2$ and $H^1$ errors, in the Y-axis. These plots allow us to easily compare the behaviour of the different methods and observe their dependency with respect to the penalty parameter $\gamma$. 

\subsubsection{2D results} 
\label{ssub:2D results}

We show in Figure~\ref{fig:k-l2e-h1e-vs-gamma-square-k1-k2-dir-poisson} the plots corresponding to the 2D square using both linear and quadratic \acp{fe} for the solution of the Poisson problem with weak imposition of Dirichlet boundary conditions. 

Let us consider first the condition number plots. For linear elements, we observe that all weak methods have the same behaviour. In all cases, the minimum is attained at $\gamma = 1$. The condition number increases as $\gamma \to 0$. The stabilisation is not enough to fix the small cut cell problem effect on the condition number. The condition number linearly increases as $\gamma \to \infty$. This phenomenon is well-understood. The minimum eigenvalue does not grow with $\gamma$, since eigenvectors can be in the kernel of the penalty term, i.e., they can belong to a subspace that cancels the penalty term. On the contrary, the maximum eigenvalue linearly increases with $\gamma$. The results for second order elements are similar, even though there is a more erratic behaviour for values of $\gamma$ below 1. The condition number of the strong \ac{agfem} (\ac{sag}) is constant (the method does not depend on $\gamma$) and is below weak schemes in all cases.

Now, we analyse the $L^2$ and $H^1$ error of the methods. For linear elements, we can observe the \emph{locking} phenomenon predicted in Sect.~\ref{sub:gp-limit} in the limit $\gamma \to \infty$. We can observe that we have three types of weak methods. The worst method is \ac{fgp}, because it is the one that rigidises on all ghost skeleton facets (see Sect.~\ref{sub:gp-limit}). The methods that only introduce intra-aggregate rigidisation, i.e., the weaker facet-stabilisation \ac{agp} and the bulk-based method \ac{bgpi}, show slightly lower errors, but still exhibit \emph{locking} and are not convergent methods in the limit. It is more serious for practical purposes the strong sensitivity of these methods to $\gamma$. The behavior is highly erratic. For linear elements, the best result is for $\gamma = 10^{-2}$, but this is not the case for quadratic elements. Besides, \ac{fgp}, \ac{agp} and \ac{bgpi} are worse than \ac{wagL2}, \ac{wagH1} and \ac{sag} in all cases but $\gamma = 1$, where \ac{bgpi} has a similar behavior as the Ag-based methods. The novel weak versions of \ac{agfem} are far less sensitive to $\gamma$. The methods can exhibit some instability as $\gamma \to 0$, since the stabilisation is not enough. However, the methods are very robust for $\gamma \geq 1$. On the other hand, as stated in Sect.~\ref{sub:weak-agfem}, W-Ag-* show the same error as \ac{sag}, indicating that these algorithms do converge to \ac{sag} in the limit $\gamma \to \infty$. Thus, these methods are not affected by any \emph{locking} in this limit.  

\newcommand{\addlegend}{
\begin{footnotesize}
\begin{tabular}{cccccc}
\includegraphics[width=0.06\textwidth]{fig_legend_m=2.pdf} \ac{fgp} &
\includegraphics[width=0.06\textwidth]{fig_legend_m=3.pdf} \ac{agp} &
\includegraphics[width=0.06\textwidth]{fig_legend_m=5.pdf} \ac{bgpi} &
\includegraphics[width=0.06\textwidth]{fig_legend_m=4.pdf} \ac{wagL2} &
\includegraphics[width=0.06\textwidth]{fig_legend_m=6.pdf} \ac{wagH1} &
\includegraphics[width=0.06\textwidth]{fig_legend_m=1.pdf} \ac{sag}
\end{tabular}
\end{footnotesize}
}

\newcommand{\addlegendref}{
\begin{footnotesize}
\begin{tabular}{ccccccc}
\includegraphics[width=0.06\textwidth]{fig_legend_m=2.pdf} \ac{fgp} &
\includegraphics[width=0.06\textwidth]{fig_legend_m=3.pdf} \ac{agp} &
\includegraphics[width=0.06\textwidth]{fig_legend_m=5.pdf} \ac{bgpi} &
\includegraphics[width=0.06\textwidth]{fig_legend_m=4.pdf} \ac{wagL2} &
\includegraphics[width=0.06\textwidth]{fig_legend_m=6.pdf} \ac{wagH1} &
\includegraphics[width=0.06\textwidth]{fig_legend_m=1.pdf} \ac{sag} \\
\multicolumn{6}{c}{\includegraphics[width=0.06\textwidth]{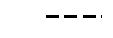}~Theoretical~slope}
\end{tabular}
\end{footnotesize}
}

\begin{figure}[!h]
	\centering
	\addlegend
	
	\vspace*{0.5em}	
	\begin{subfigure}[t]{0.32\textwidth}
\includegraphics[width=1.00\textwidth]{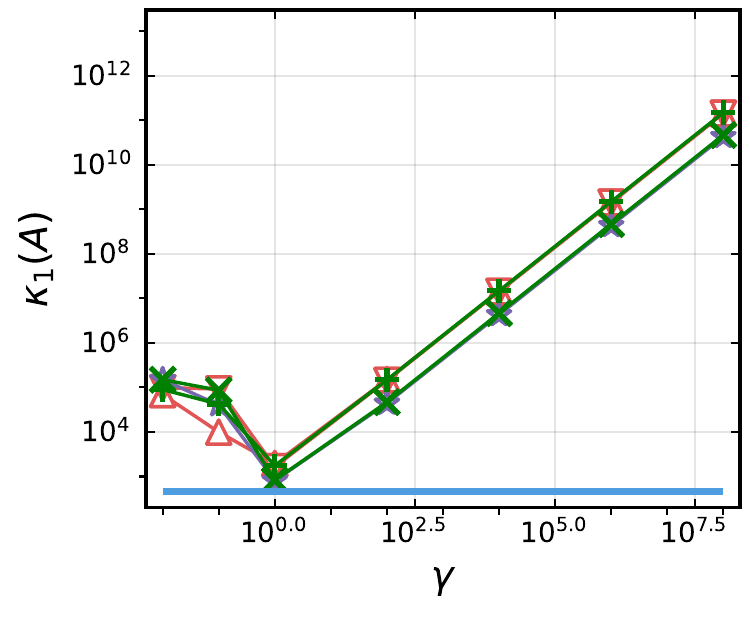}
    \caption{$\kappa_1(A)$ for order 1}
	\end{subfigure}
	\begin{subfigure}[t]{0.32\textwidth}
\includegraphics[width=1.00\textwidth]{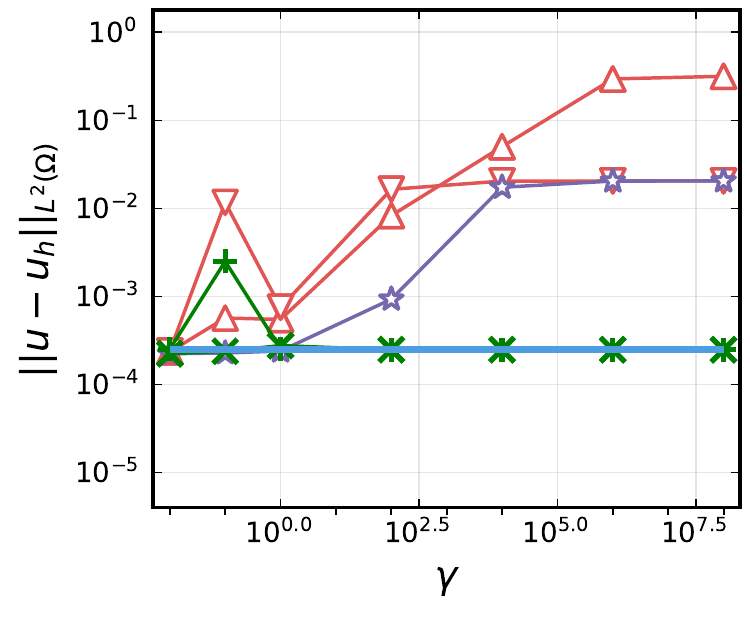}
    \caption{$\|u - u_h\|_{L^2(\Omega)}$ for order 1}
	\end{subfigure}
	\begin{subfigure}[t]{0.32\textwidth}
\includegraphics[width=1.00\textwidth]{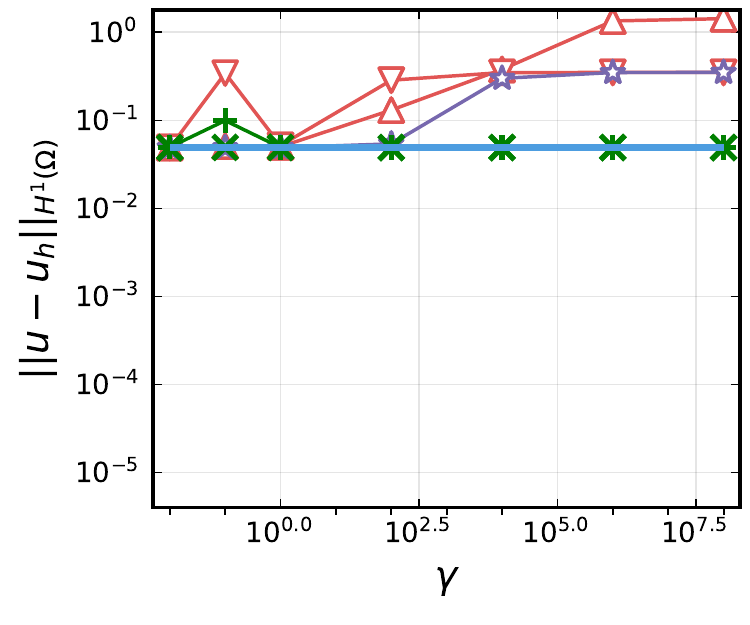}
    \caption{$\|u - u_h \|_{H^1(\Omega)}$ for order 1}
	\end{subfigure} \\
\vspace{0.6em}
\begin{subfigure}[t]{0.32\textwidth}
\includegraphics[width=1.00\textwidth]{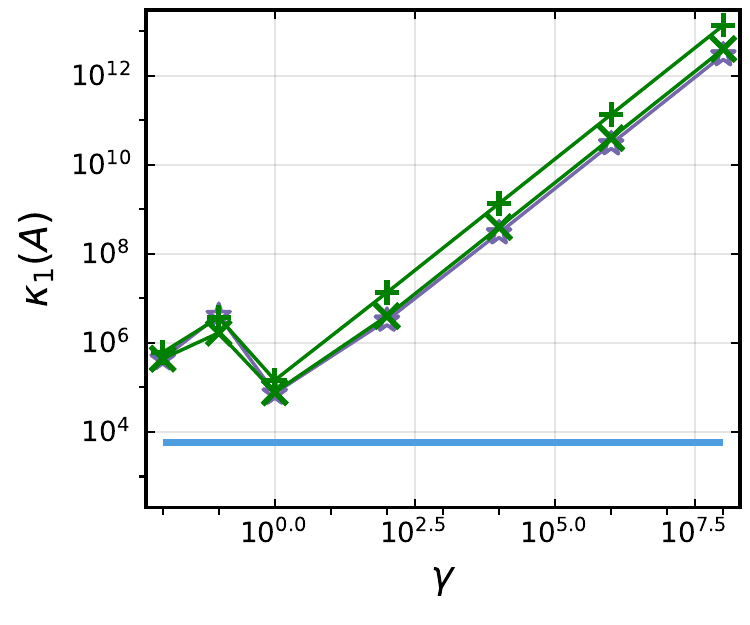}
    \caption{$\kappa_1(A)$ for order 2}
	\end{subfigure}
	\begin{subfigure}[t]{0.32\textwidth}
\includegraphics[width=1.00\textwidth]{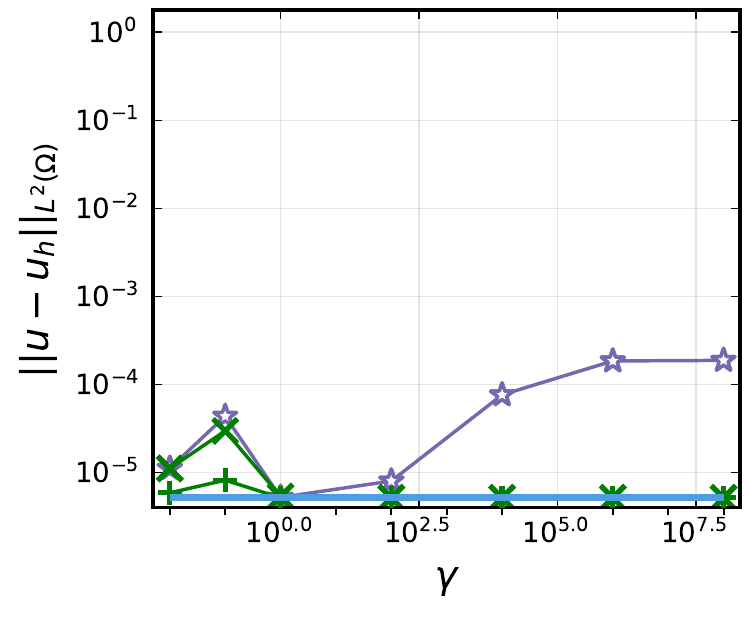}
    \caption{$\|u - u_h\|_{L^2(\Omega)}$ for order 2}
	\end{subfigure}
	\begin{subfigure}[t]{0.32\textwidth}
\includegraphics[width=1.00\textwidth]{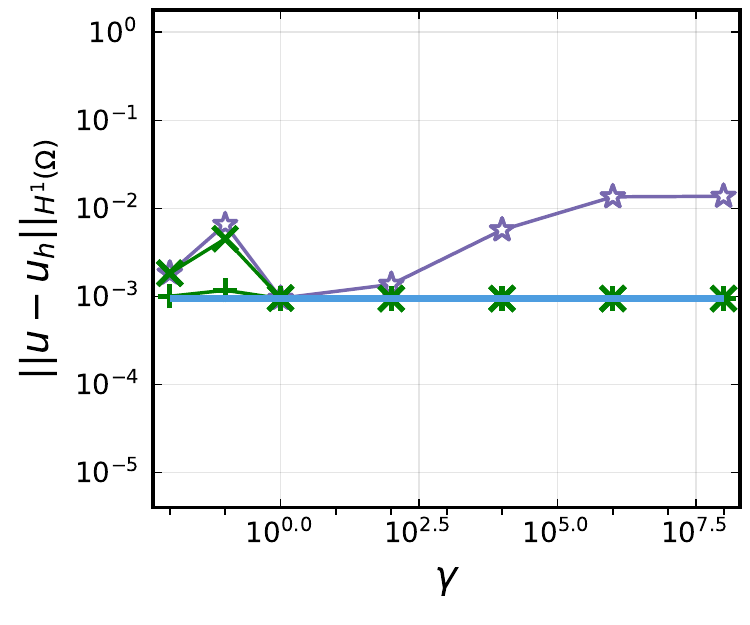}
    \caption{$\|u - u_h \|_{H^1(\Omega)}$ for order 2}
	\end{subfigure}
	\caption[]{Condition number $\kappa_1(A)$, and error norms $\|u - u_h\|_{L^2(\Omega)}$ and $\|u - u_h \|_{H^1(\Omega)}$ vs. penalty parameter $\gamma$ for the Poisson problem with Nitsche's method on the square using linear and quadratic elements {in a $40\times40$ structured triangular mesh}.}
	\label{fig:k-l2e-h1e-vs-gamma-square-k1-k2-dir-poisson}
\end{figure}

{Before reporting the 3D results, we take a closer look at the \emph{locking} phenomenon. In Figure~\ref{fig:stiffness}, we consider plots of the solution in the square for $\gamma = 10000$, using linear elements in a $20 \times 20$ structured triangular mesh. As expected, the \ac{fgp} solution is affected by severe locking and completely wrong. The \ac{agp} and \ac{bgpi} solutions are barely indistinguishable; they are still polluted by locking, which constraints the values of interior \acp{dof} away from the analytical solution. Finally, W-Ag-* are completely free of locking and coincide with the \ac{sag} solution.}

\begin{figure}[!h]
	\centering
	\begin{subfigure}[t]{0.26\textwidth}
		\centering
		\includegraphics[width=0.95\textwidth, trim=5.25cm 15.0cm 5.25cm 0.75cm, clip=true]{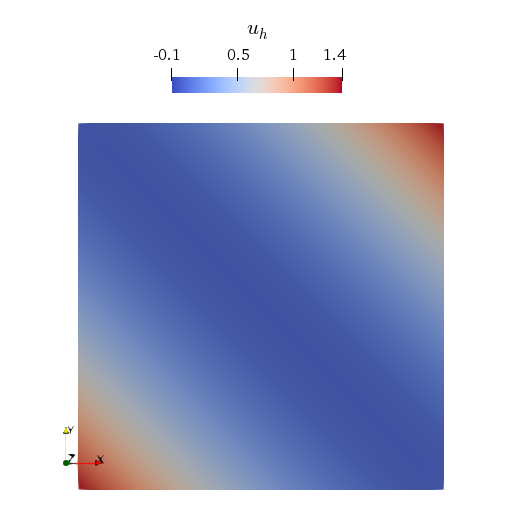}
	\end{subfigure} \\
	\vspace{0.6em}
	\begin{subfigure}[t]{0.28\textwidth}
		\centering
		\includegraphics[width=0.95\textwidth]{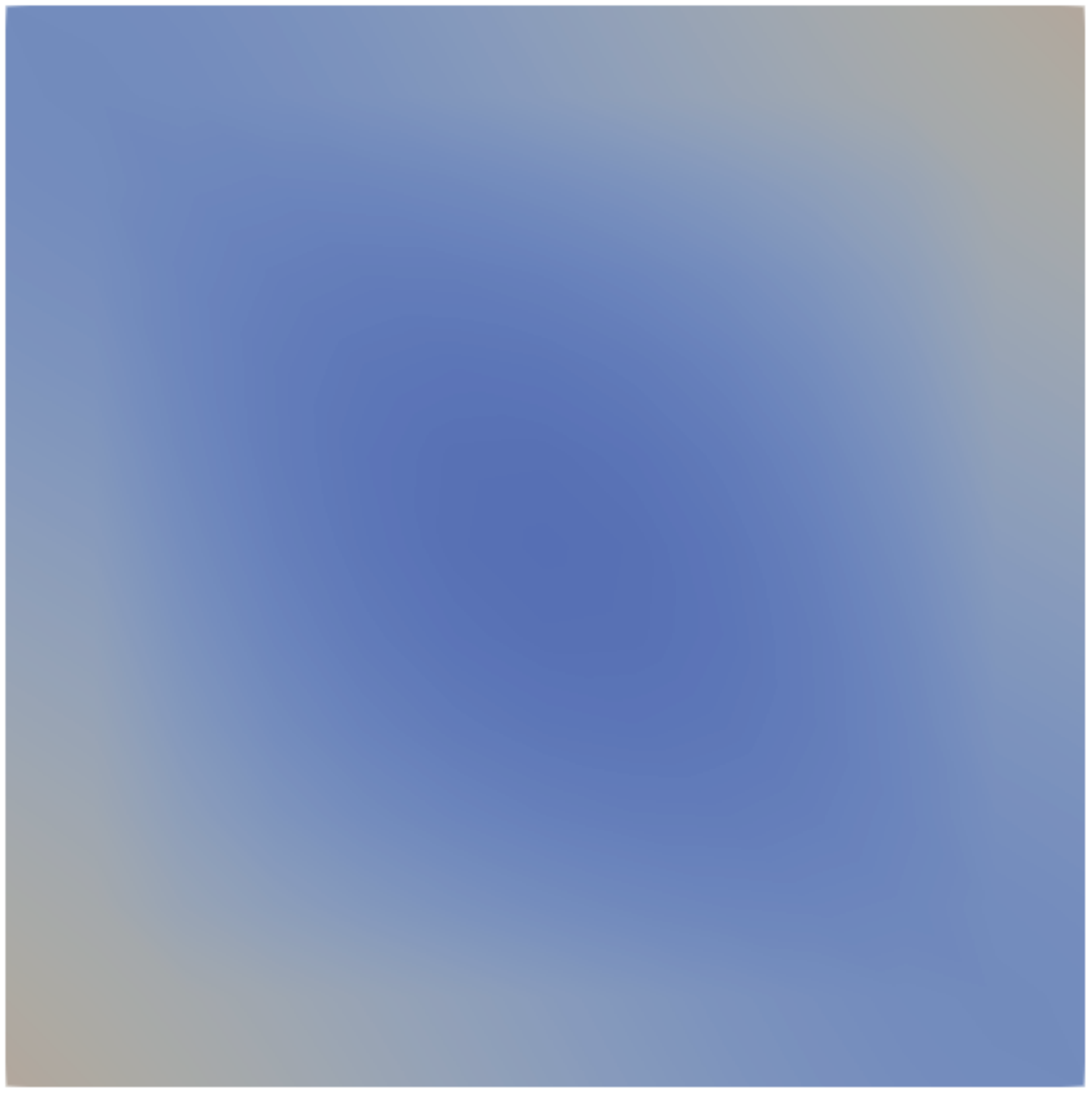}
		\caption{\ac{fgp}}
	\end{subfigure}
	\begin{subfigure}[t]{0.28\textwidth}
		\centering
		\includegraphics[width=0.95\textwidth]{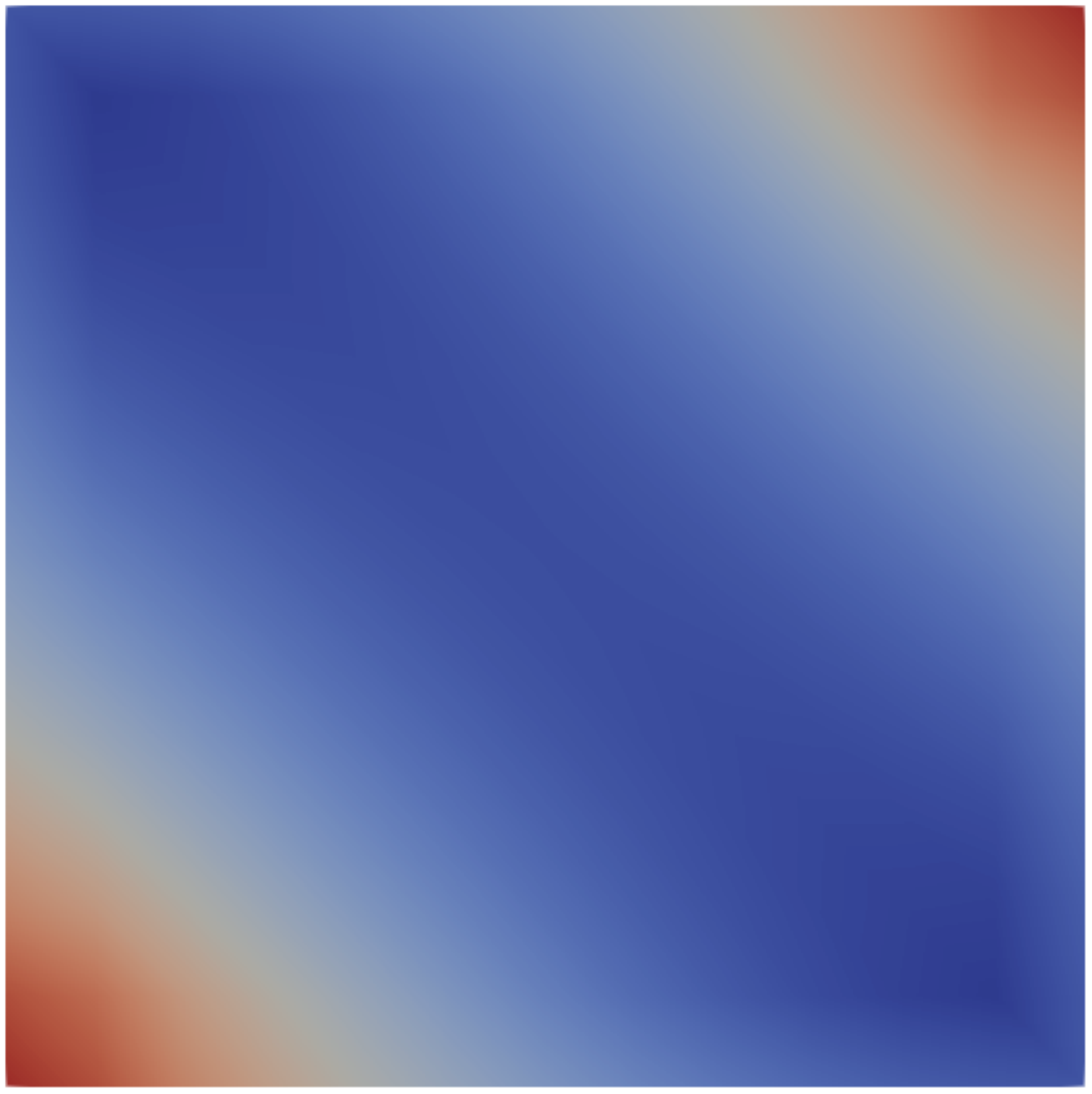}
		\caption{\ac{agp}}
	\end{subfigure}
	\begin{subfigure}[t]{0.28\textwidth}
		\centering
		\includegraphics[width=0.95\textwidth]{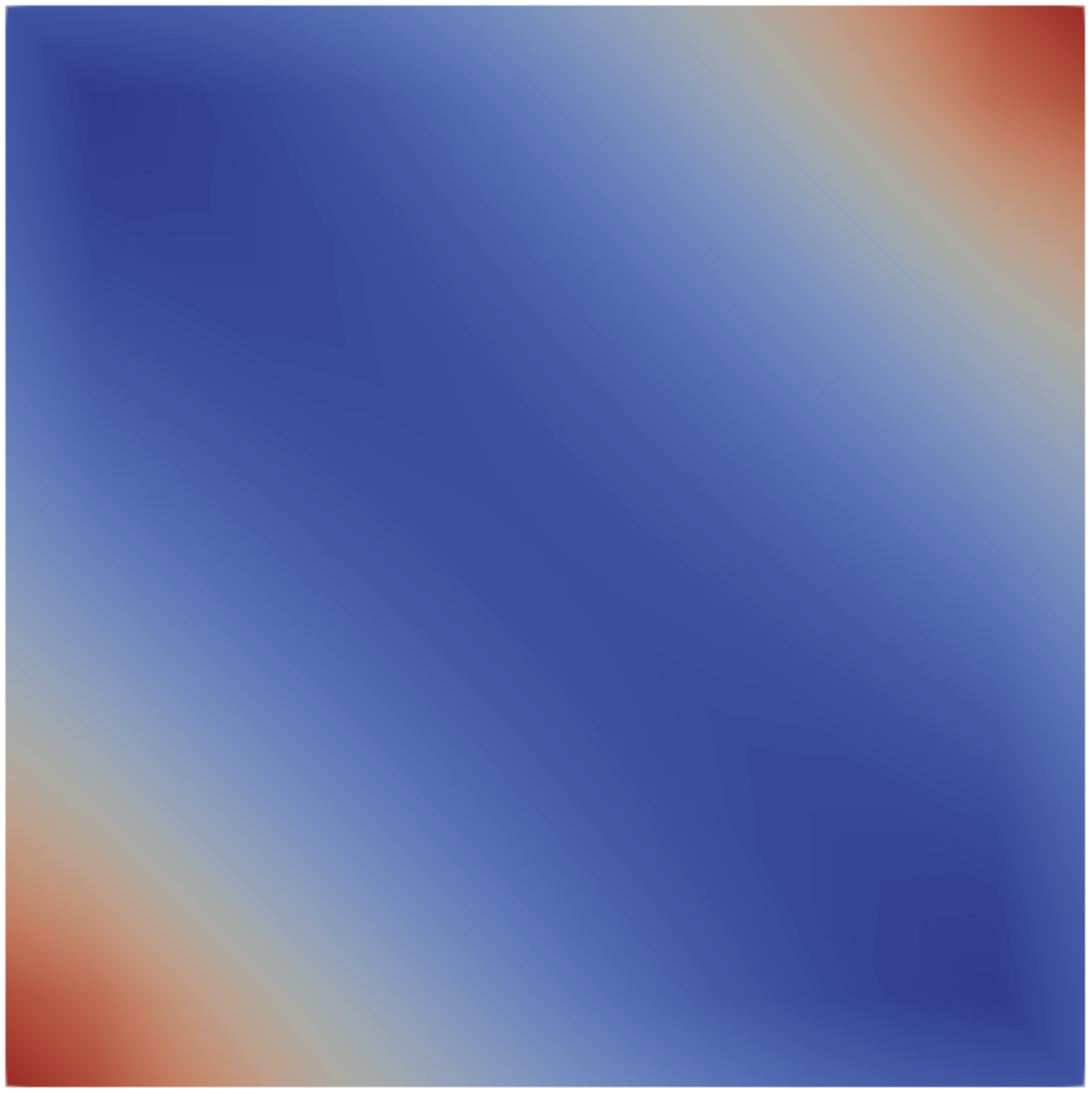}
    	\caption{\ac{bgpi}}
	\end{subfigure} \\
	\vspace{0.6em}
	\begin{subfigure}[t]{0.28\textwidth}
		\centering
		\includegraphics[width=0.95\textwidth]{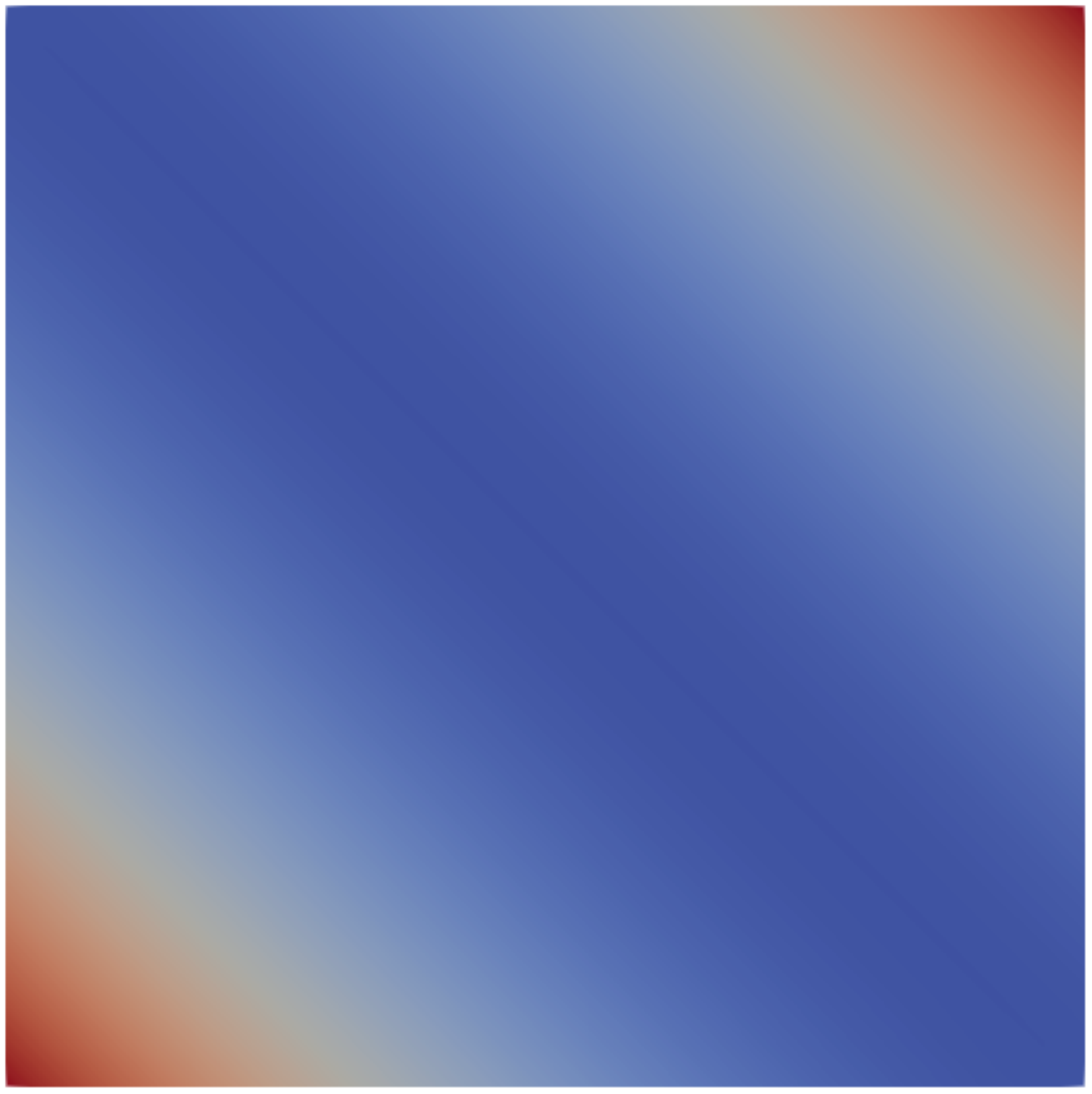}
		\caption{\ac{wagL2}}
	\end{subfigure}
	\begin{subfigure}[t]{0.28\textwidth}
		\centering
		\includegraphics[width=0.95\textwidth]{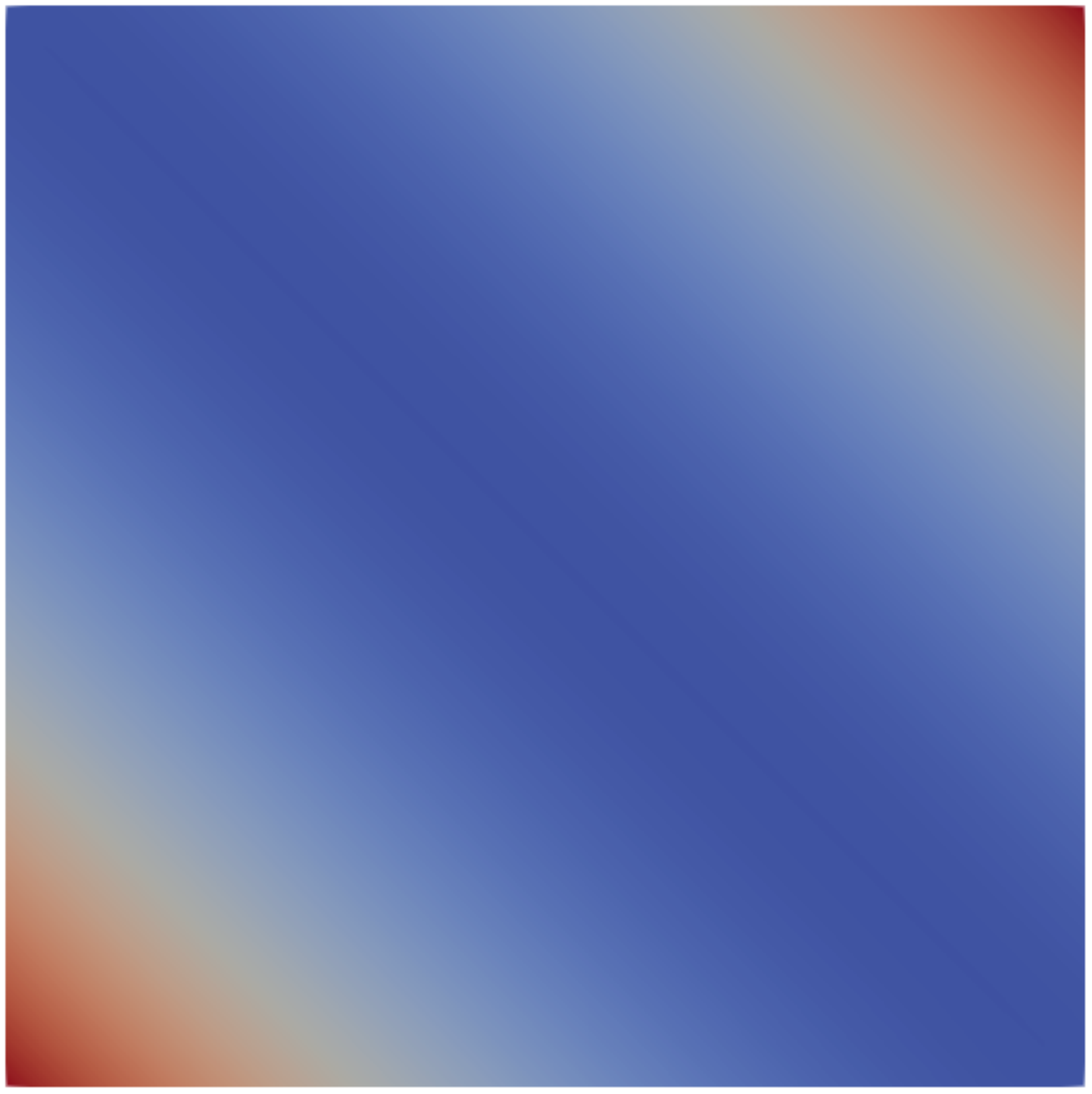}
		\caption{\ac{wagH1}}
	\end{subfigure}
	\begin{subfigure}[t]{0.28\textwidth}
		\centering
		\includegraphics[width=0.95\textwidth]{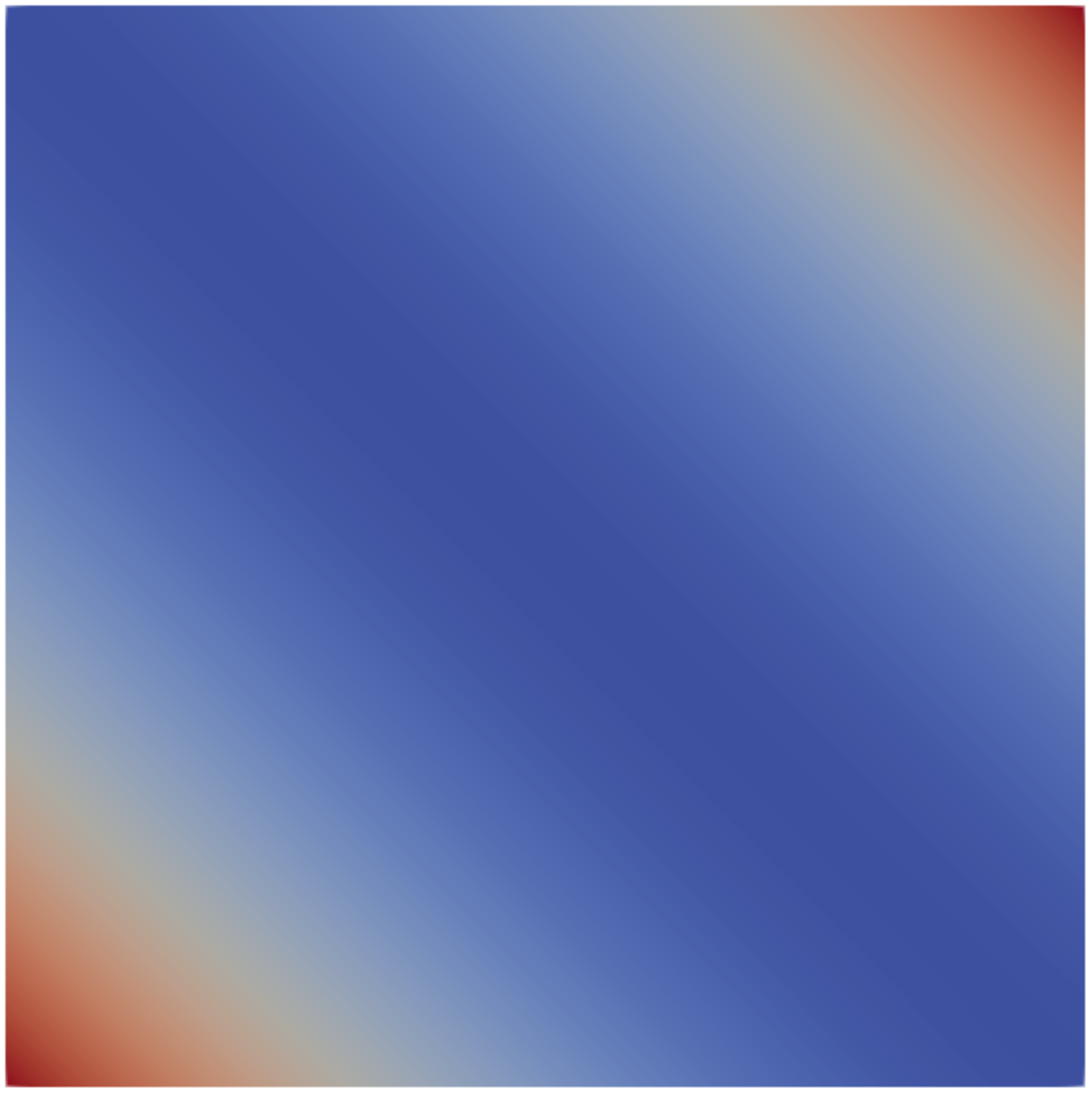}
    	\caption{\ac{sag}}
	\end{subfigure} 
	\caption[]{Solution of the Poisson problem with Nitsche's method on the square, using linear elements in a $20 \times 20$ structured triangular mesh. All \emph{weak} methods take $\gamma = 10000$.}
	\label{fig:stiffness}
\end{figure}

\subsubsection{3D results}\label{sub:3d-results}
We show in Figure~\ref{fig:k-l2e-h1e-vs-gamma-cube-k1-dir-poisson-cube-tilted-sphere} the same kind of plots for 3D geometries. In this test, we consider three different geometries, in order to evaluate whether the geometry affects the behavior of the different methods. We have considered the cube, a tilted cube and a sphere. We consider linear elements, only. We observe essentially the same behavior as in 2D. The results for all methods are quite insensitive to the geometry. The sphere is an \emph{easier} geometry in general, since it is smoother and does not involve corners. However, the results are very similar to the ones for the cube and tilted cube. The only difference would be the slightly better results of weak methods for $\gamma$ below 1. Anyway, the more important observations clearly apply. \ac{fgp}, \ac{agp} and \ac{bgpi} are very sensitive to $\gamma$ and exhibit \emph{locking}. W-Ag-* methods are very insensitive to $\gamma$. The \ac{sag} method is close to the optimal value in most situations, while parameter-free.     

\begin{figure}[!h]
	\centering
    \addlegend
	
	\vspace*{0.5em}	
	\begin{subfigure}[t]{0.32\textwidth}
\includegraphics[width=1.00\textwidth]{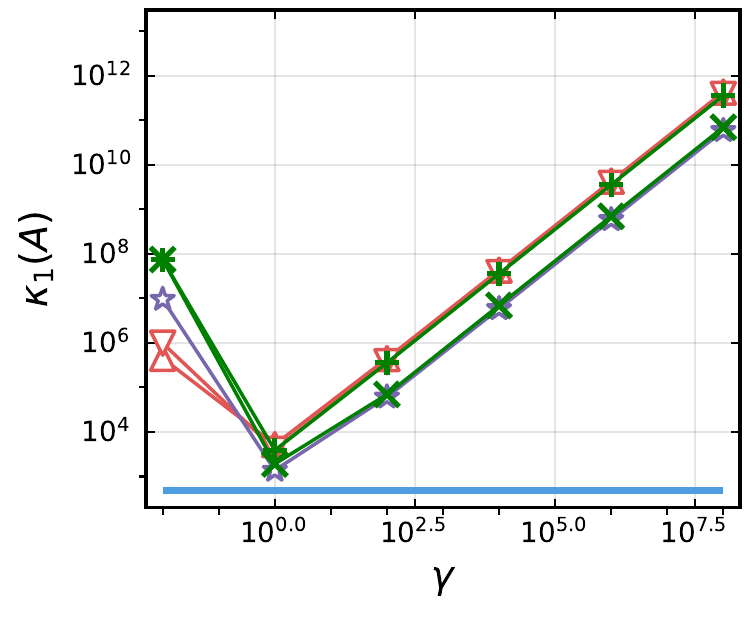}
    \caption{$\kappa_1(A)$ for cube}
	\end{subfigure}
	\begin{subfigure}[t]{0.32\textwidth}
\includegraphics[width=1.00\textwidth]{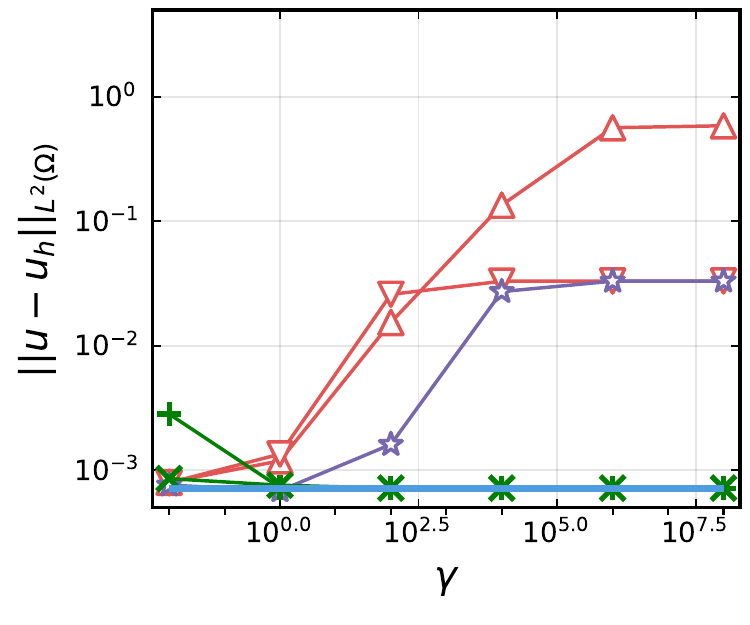}
    \caption{$\|u - u_h\|_{L^2(\Omega)}$ for cube}
	\end{subfigure}
	\begin{subfigure}[t]{0.32\textwidth}
\includegraphics[width=1.00\textwidth]{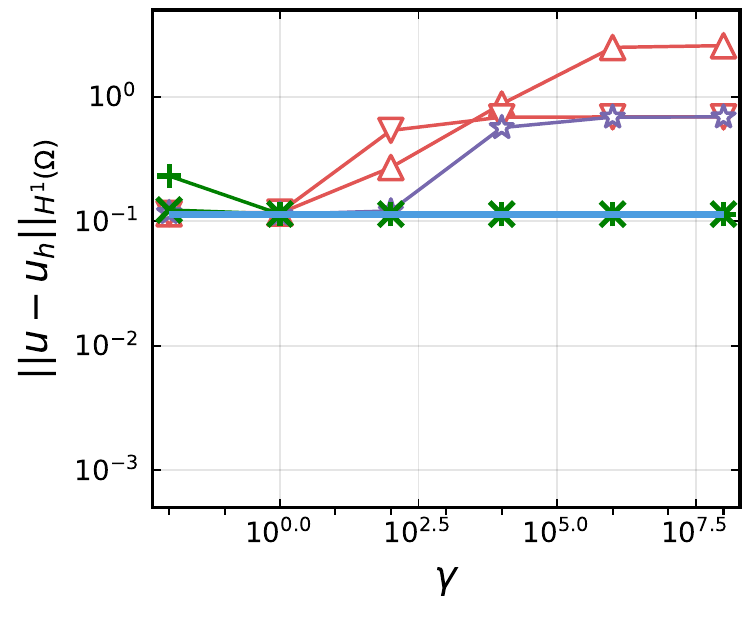}
    \caption{$\|u - u_h\|_{H^1(\Omega)}$ for cube}
	\end{subfigure} \\
\vspace{0.6em}
\begin{subfigure}[t]{0.32\textwidth}
\includegraphics[width=1.00\textwidth]{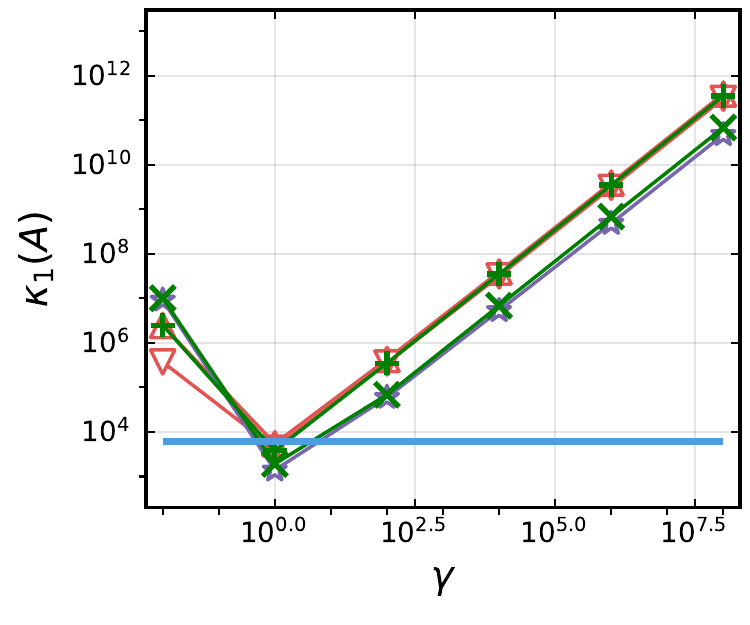}
    \caption{$\kappa_1(A)$ for tilted cube}
	\end{subfigure}
	\begin{subfigure}[t]{0.32\textwidth}
\includegraphics[width=1.00\textwidth]{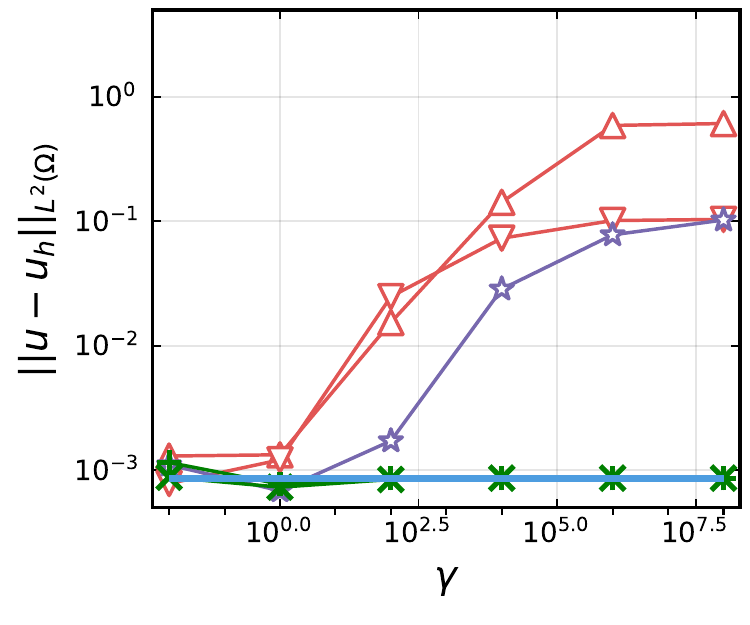}
    \caption{$\|u - u_h\|_{L^2(\Omega)}$ for tilted cube}
	\end{subfigure}
	\begin{subfigure}[t]{0.32\textwidth}
\includegraphics[width=1.00\textwidth]{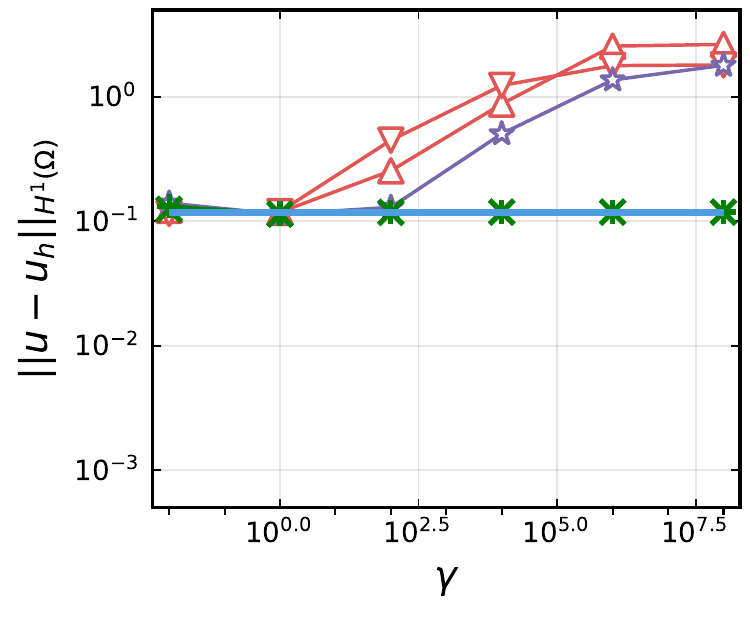}
    \caption{$\|u - u_h\|_{H^1(\Omega)}$ for tilted cube}
	\end{subfigure} \\
\vspace{0.6em}
\begin{subfigure}[t]{0.32\textwidth}
\includegraphics[width=1.00\textwidth]{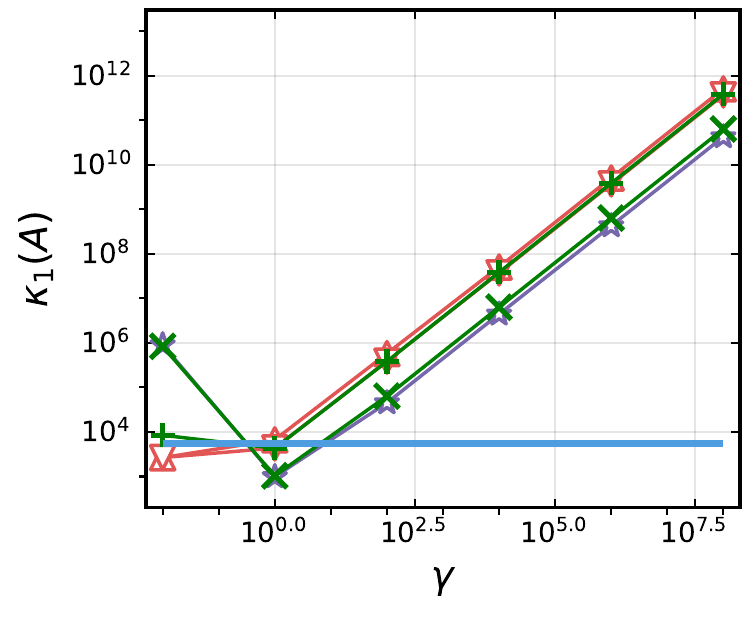}
    \caption{$\kappa_1(A)$ for sphere}
	\end{subfigure}
	\begin{subfigure}[t]{0.32\textwidth}
\includegraphics[width=1.00\textwidth]{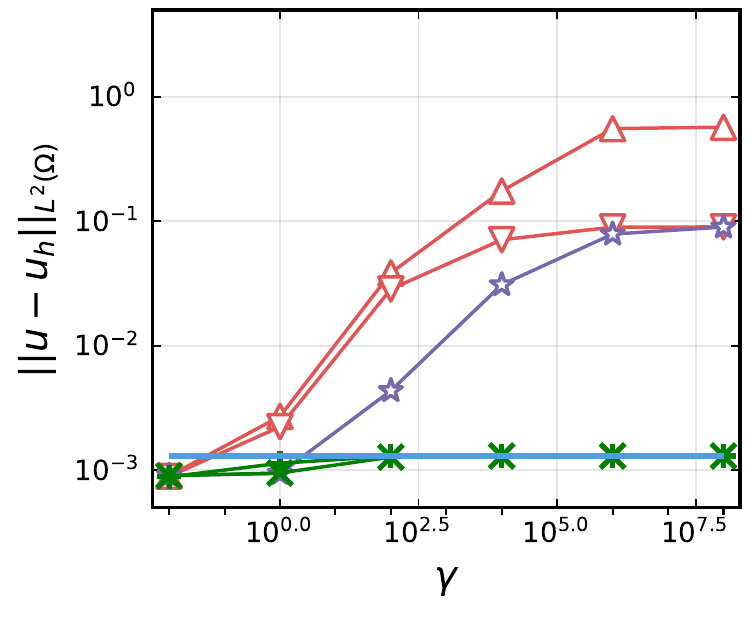}
    \caption{$\|u - u_h\|_{L^2(\Omega)}$ for sphere}
	\end{subfigure}
	\begin{subfigure}[t]{0.32\textwidth}
\includegraphics[width=1.00\textwidth]{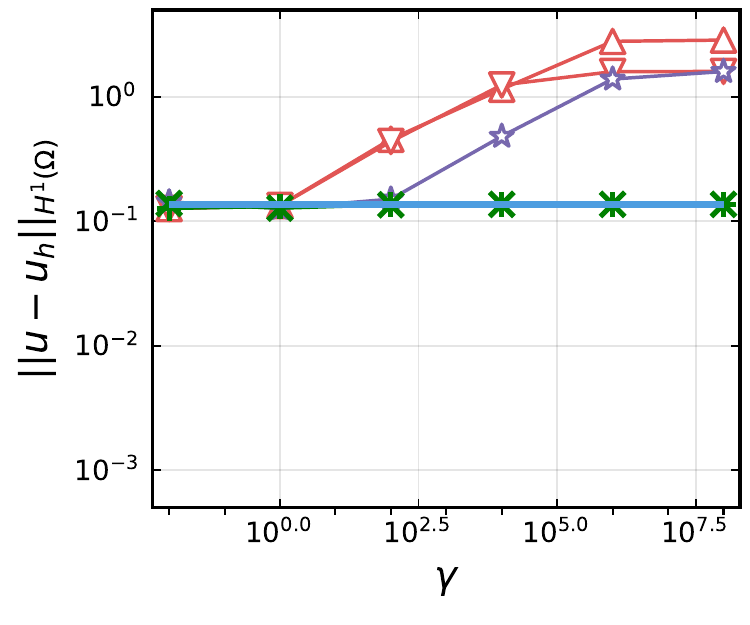}
    \caption{$\|u - u_h\|_{H^1(\Omega)}$ for sphere}
	\end{subfigure}
	\caption[]{Condition number $\kappa_1(A)$, and error norms $\|u - u_h\|_{L^2(\Omega)}$ and $\|u - u_h \|_{H^1(\Omega)}$ vs. penalty parameter $\gamma$ for the Poisson problem with Nitsche's method on the cube, tilted cube and sphere, {for order 1 in a $40\times40\times40$ structured tetrahedral mesh}.}
	\label{fig:k-l2e-h1e-vs-gamma-cube-k1-dir-poisson-cube-tilted-sphere}
\end{figure}

We compare now the linear elasticity problem and Neumann boundary conditions and consider both linear and quadratic elements on the 3D cube. In this case, W-Ag-* errors are totally independent of $\gamma$, also for low values of $\gamma$. Condition number plots have the same behavior as above; for the largest values of $\gamma$ the problem is so ill-conditioned that the solution is severely affected by rounding errors, but such a value of $\gamma$ is for presentation purposes and certainly not a value to be used in practise. The W-GP, \ac{agp} and \ac{bgpi} methods show the same problems discussed above. \ac{sag} exhibits optimal condition numbers and errors in all cases. The results are very similar for linear and quadratic \acp{fe}. The only difference is that the condition number of \ac{sag} is about one order of magnitude lower than the weak methods. 

\begin{figure}[!h]
	\centering
	\addlegend
	
	\vspace*{0.5em}	
	\begin{subfigure}[t]{0.32\textwidth}
\includegraphics[width=1.00\textwidth]{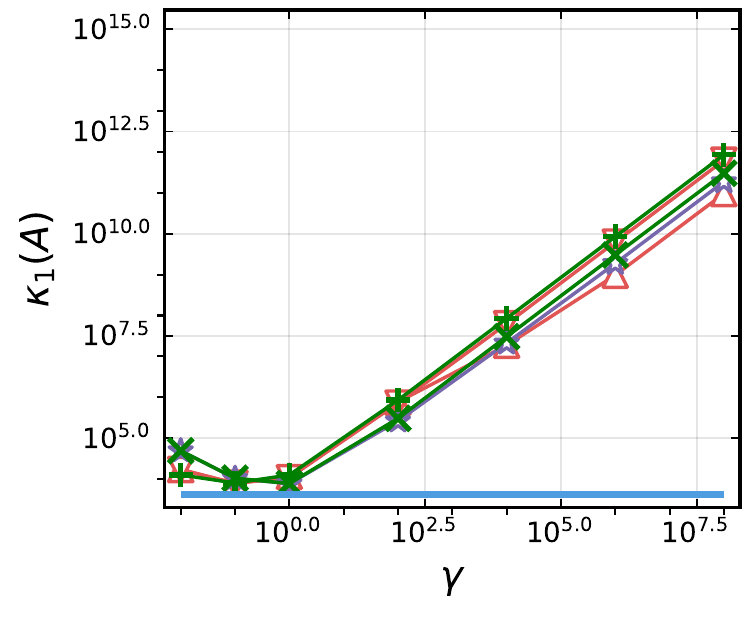}
    \caption{$\kappa_1(A)$ for order 1}
	\end{subfigure}
	\begin{subfigure}[t]{0.32\textwidth}
\includegraphics[width=1.00\textwidth]{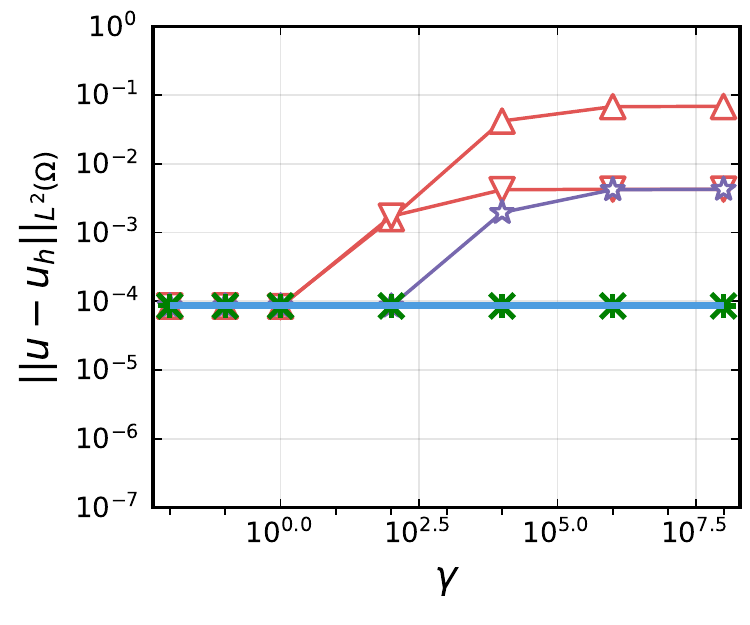}
    \caption{$\|u - u_h\|_{L^2(\Omega)}$ for order 1}
	\end{subfigure}
	\begin{subfigure}[t]{0.32\textwidth}
\includegraphics[width=1.00\textwidth]{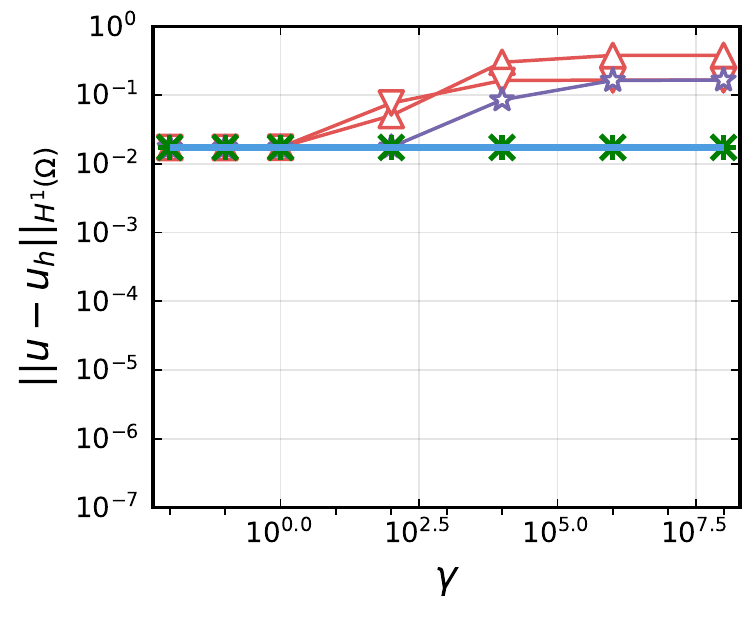}
    \caption{$\|u - u_h\|_{H^1(\Omega)}$ for order 1}
	\end{subfigure} \\
\vspace{0.6em}
\begin{subfigure}[t]{0.32\textwidth}
\includegraphics[width=1.00\textwidth]{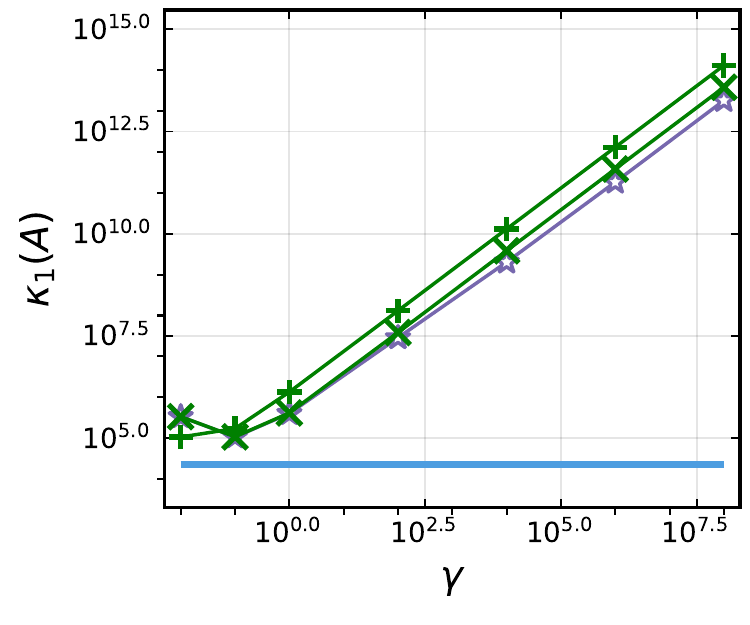}
    \caption{$\kappa_1(A)$ for order 2}
	\end{subfigure}
	\begin{subfigure}[t]{0.32\textwidth}
\includegraphics[width=1.00\textwidth]{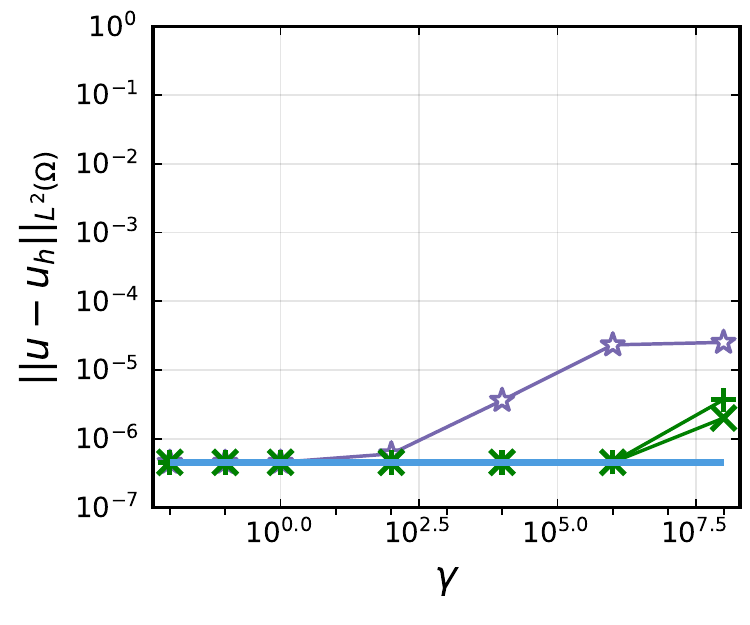}
    \caption{$\|u - u_h\|_{L^2(\Omega)}$ for order 2}
	\end{subfigure}
	\begin{subfigure}[t]{0.32\textwidth}
\includegraphics[width=1.00\textwidth]{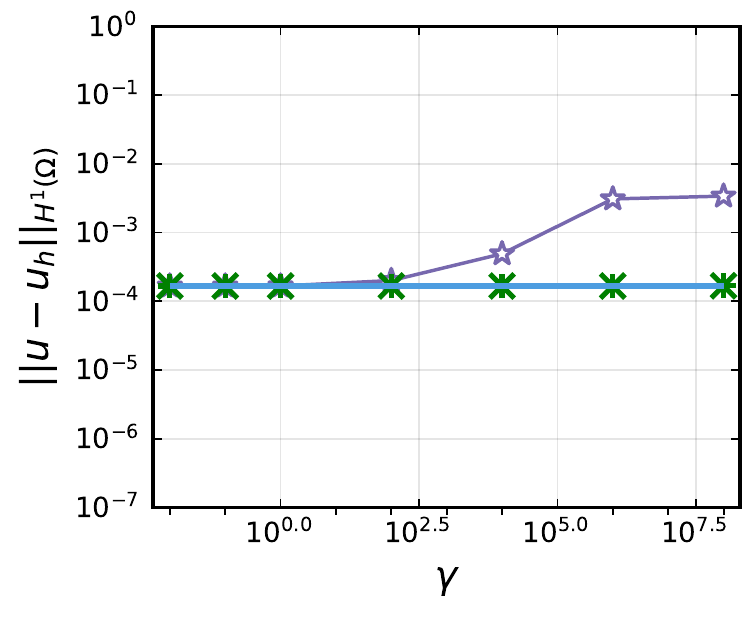}
    \caption{$\|u - u_h\|_{H^1(\Omega)}$ for order 2}
	\end{subfigure}
	\caption[]{Condition number $\kappa_1(A)$, and error norms $\|u - u_h\|_{L^2(\Omega)}$ and $\|u - u_h \|_{H^1(\Omega)}$ vs. penalty parameter $\gamma$ for the linear elasticity problem with Neumann and strong Dirichlet boundary conditions on the cube using linear and quadratic elements {in a $40\times40\times40$ structured tetrahedral mesh}.}
	\label{fig:k-l2e-h1e-vs-gamma-cube-k1-k2-dir-elasticity}
\end{figure}

In Figure~\ref{fig:k-l2e-h1e-vs-gamma-sphere-k1-k2-dir-elasticity}, we analyse the behavior of the methods for linear elasticity and Neumann boundary conditions on the sphere for both linear and quadratic \acp{fe}. Again, W-Ag-* methods turn to be very insensitive to $\gamma$. \ac{fgp}, \ac{agp} and \ac{bgpi} show the same bad behavior indicated in the previous examples. The main difference with respect to above is the fact that the condition number of \ac{sag} for quadratic elements is about one order of magnitude larger than the minimum value obtained with weak schemes. 

\begin{figure}[!h]
	\centering
	\addlegend
	
	\vspace*{0.5em}	
	\begin{subfigure}[t]{0.32\textwidth}
\includegraphics[width=1.00\textwidth]{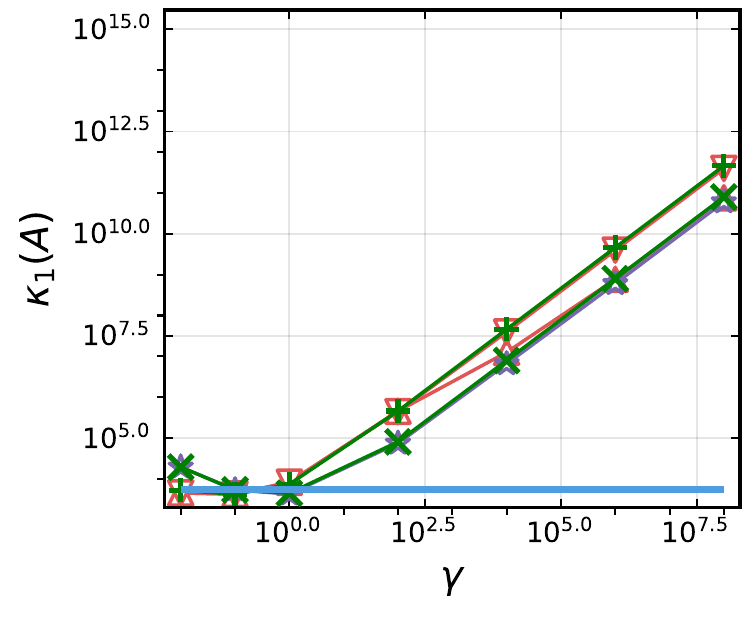}
    \caption{$\kappa_1(A)$ for order 1}
	\end{subfigure}
	\begin{subfigure}[t]{0.32\textwidth}
\includegraphics[width=1.00\textwidth]{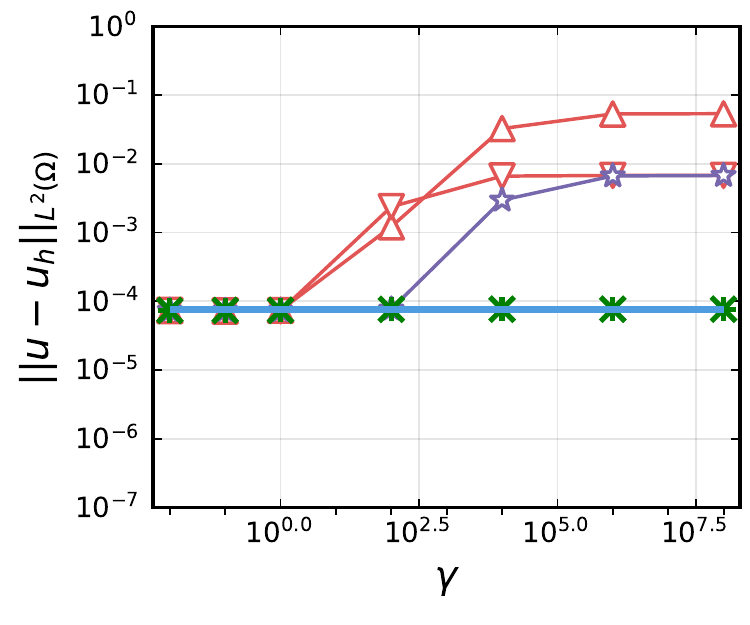}
    \caption{$\|u - u_h\|_{L^2(\Omega)}$ for order 1}
	\end{subfigure}
	\begin{subfigure}[t]{0.32\textwidth}
\includegraphics[width=1.00\textwidth]{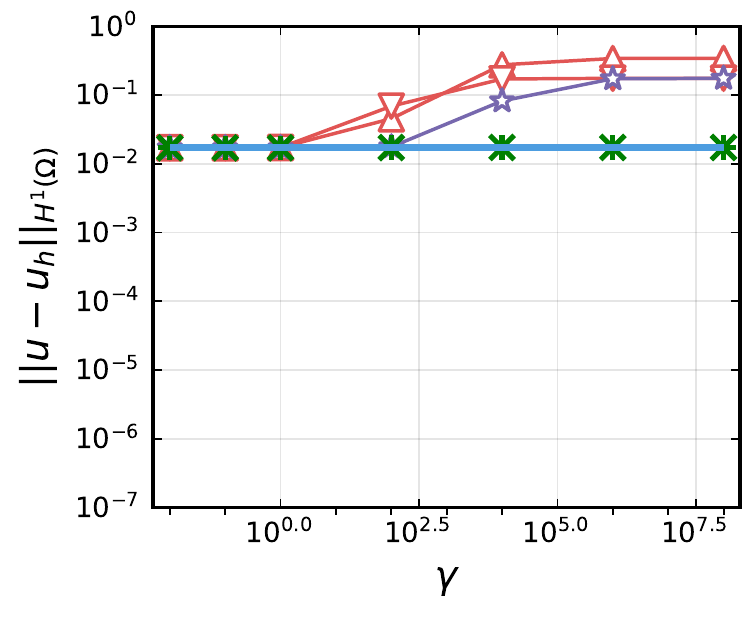}
    \caption{$\|u - u_h\|_{H^1(\Omega)}$ for order 1}
	\end{subfigure} \\
\vspace{0.6em}
\begin{subfigure}[t]{0.32\textwidth}
\includegraphics[width=1.00\textwidth]{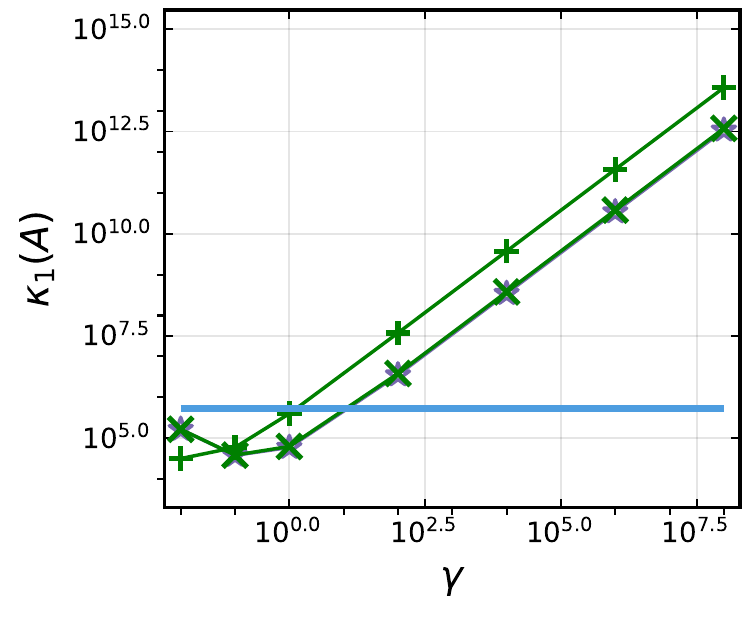}
    \caption{$\kappa_1(A)$ for order 2}
	\end{subfigure}
	\begin{subfigure}[t]{0.32\textwidth}
\includegraphics[width=1.00\textwidth]{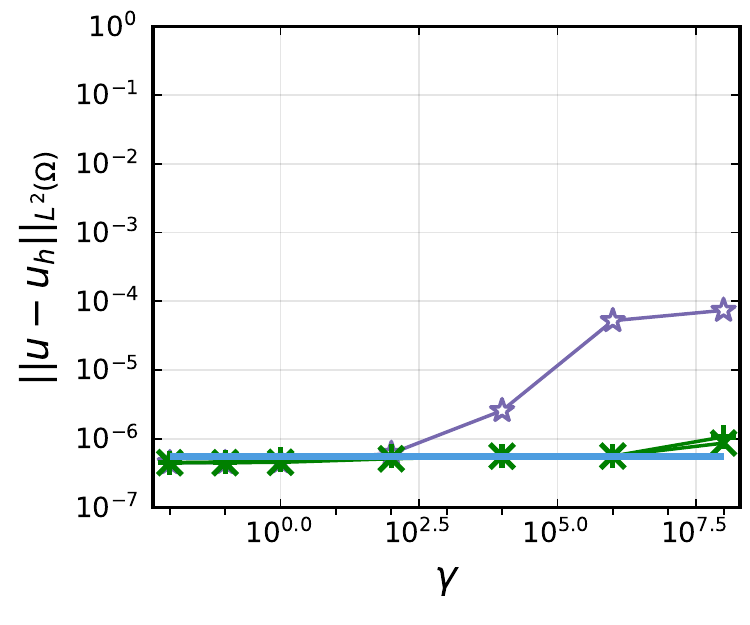}
    \caption{$\|u - u_h\|_{L^2(\Omega)}$ for order 2}
	\end{subfigure}
	\begin{subfigure}[t]{0.32\textwidth}
\includegraphics[width=1.00\textwidth]{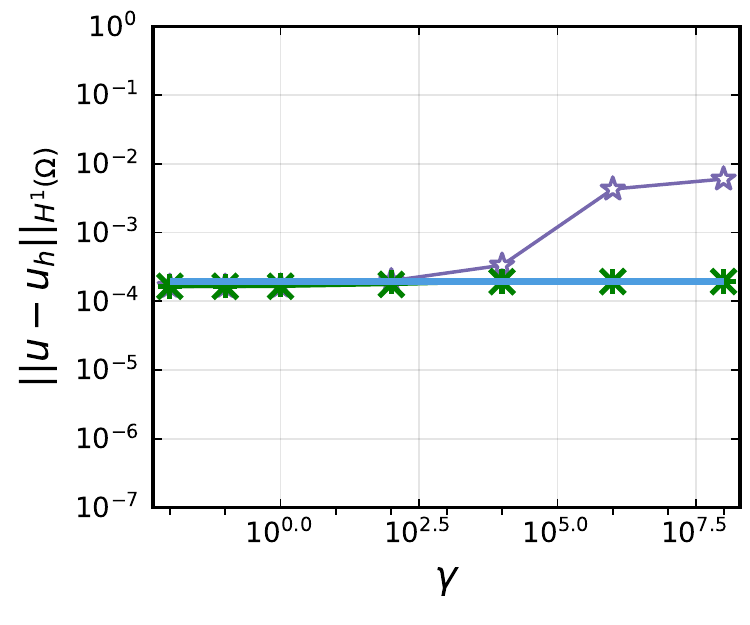}
    \caption{$\|u - u_h\|_{H^1(\Omega)}$ for order 2}
	\end{subfigure}
	\caption[]{Condition number $\kappa_1(A)$, and error norms $\|u - u_h\|_{L^2(\Omega)}$ and $\|u - u_h \|_{H^1(\Omega)}$ vs. penalty parameter $\gamma$ for the linear elasticity problem with with Neumann and strong Dirichlet boundary conditions on the sphere using linear and quadratic elements {in a $40\times40\times40$ structured tetrahedral mesh}.}
	\label{fig:k-l2e-h1e-vs-gamma-sphere-k1-k2-dir-elasticity}
\end{figure}

\subsection{Results under refinement} 
\label{sub:Results under refinement}

In this section, we analyse results in a more standard way. We consider the variation of the condition number and $L^2$ and $H^1$ errors as we refine the mesh. We show the results for three different values, $\gamma \in \left\{ 1, 100, 10^{8} \right\}$. In Figure~\ref{fig:k-l2e-h1e-vs-h-cube-k1-dir-poisson}, we consider the Poisson problem with Nitsche's method for linear elements on the cube. For the condition number bounds, we observe that the weak methods are in some cases still in a pre-asymptotic regime and in all cases above \ac{sag}. The loss of convergence of \ac{fgp} with increasing values of $\gamma$ is very clear. As indicated above, the results for \ac{agp} and \ac{bgpi} are slightly better, but they still show the loss of convergence in this limit. W-Ag-* and \ac{sag} practically show the same convergence behaviour in all cases. We show in Figure~\ref{fig:k-l2e-h1e-vs-h-8-ball-k2-neu-elasticity} the same plots for elasticity, Neumann boundary conditions and the 8-norm sphere using quadratic elements. Despite all the differences in the numerical setup, the results are strikingly similar and the same conclusions apply.

\begin{figure}[!h]
	\centering
	\addlegendref
	
	\vspace*{0.5em}	
	\begin{subfigure}[t]{0.32\textwidth}
\includegraphics[width=1.0\textwidth]{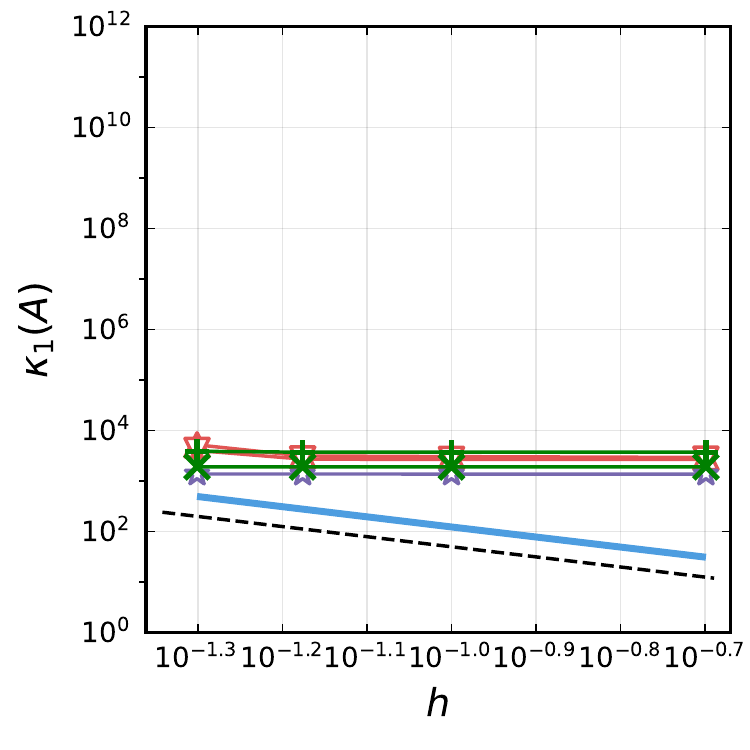}
		\caption{$\kappa_1(A)$ for $\gamma = 1$}
	\end{subfigure}
	\begin{subfigure}[t]{0.32\textwidth}
\includegraphics[width=1.0\textwidth]{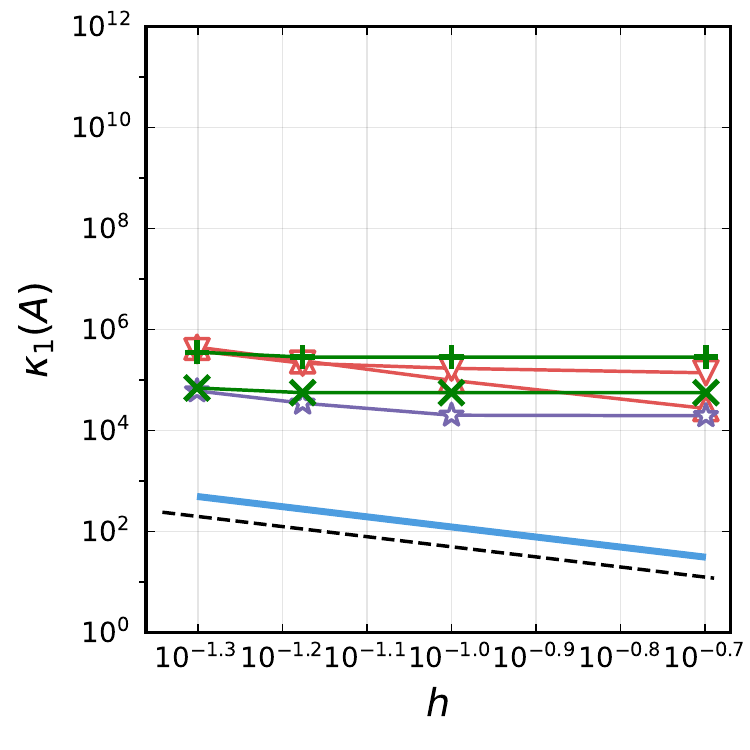}
		\caption{$\kappa_1(A)$ for $\gamma = 100$}
	\end{subfigure}
	\begin{subfigure}[t]{0.32\textwidth}
\includegraphics[width=1.0\textwidth]{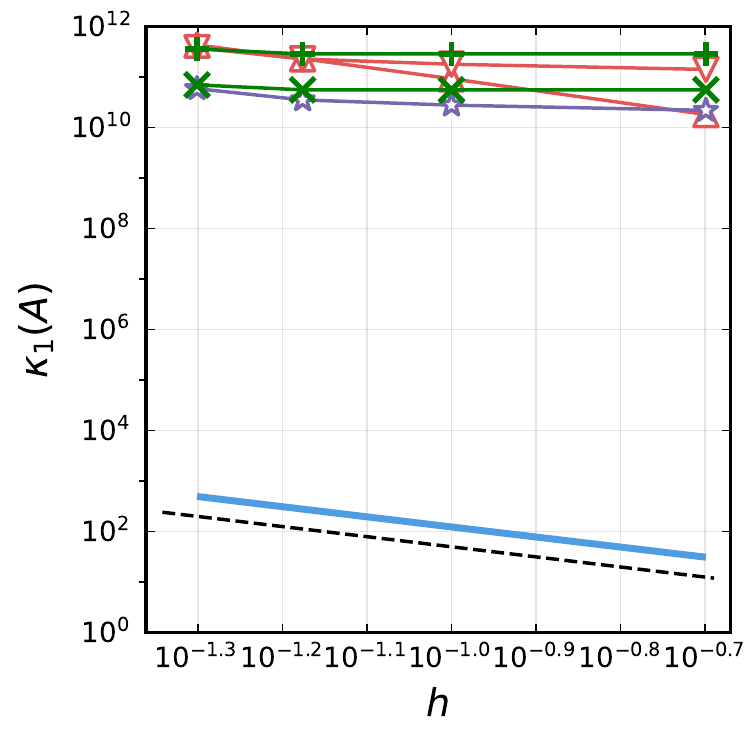}
		\caption{$\kappa_1(A)$ for $\gamma = 10^8$}
	\end{subfigure} \\
\vspace{0.6em}
\begin{subfigure}[t]{0.32\textwidth}
\includegraphics[width=1.00\textwidth]{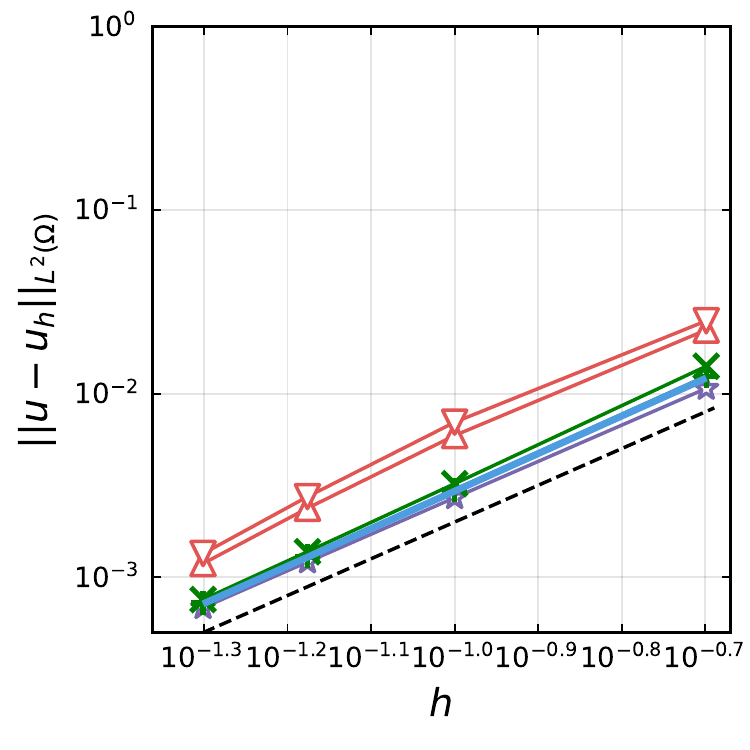}
		\caption{$\|u - u_h\|_{L^2(\Omega)}$ for $\gamma = 1$}
	\end{subfigure}
	\begin{subfigure}[t]{0.32\textwidth}
\includegraphics[width=1.00\textwidth]{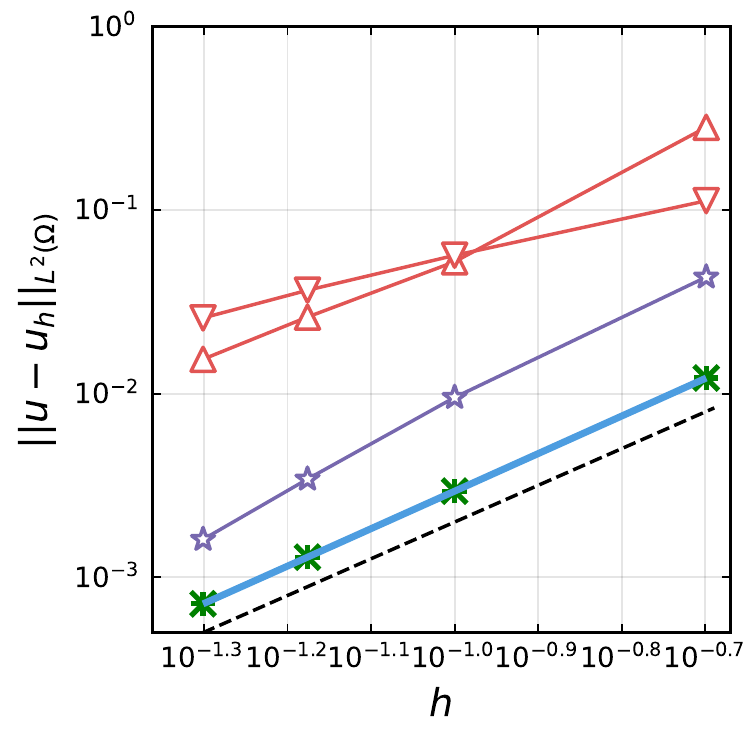}
		\caption{$\|u - u_h\|_{L^2(\Omega)}$ for $\gamma = 100$}
	\end{subfigure}
	\begin{subfigure}[t]{0.32\textwidth}
\includegraphics[width=1.00\textwidth]{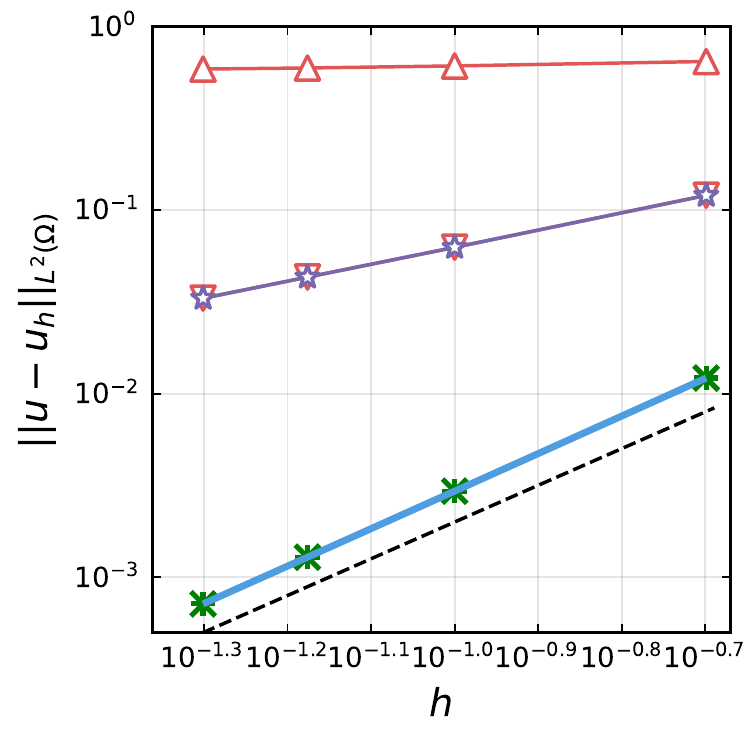}
		\caption{$\|u - u_h\|_{L^2(\Omega)}$ for $\gamma = 10^8$}
	\end{subfigure}\\ 
\vspace{0.6em}
\begin{subfigure}[t]{0.32\textwidth}
\includegraphics[width=1.00\textwidth]{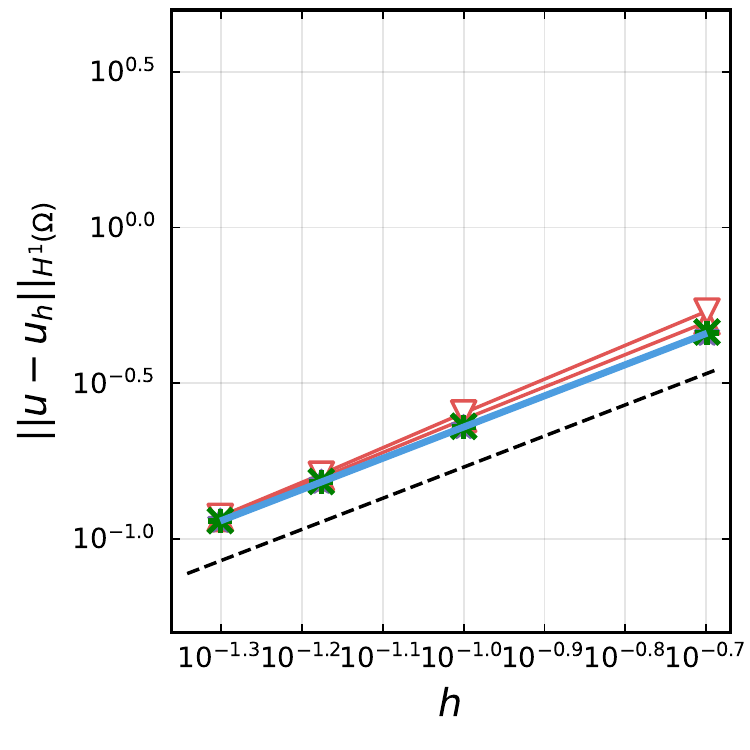}
		\caption{$\|u - u_h\|_{H^1(\Omega)}$ for $\gamma = 1$}
	\end{subfigure}
	\begin{subfigure}[t]{0.32\textwidth}
\includegraphics[width=1.00\textwidth]{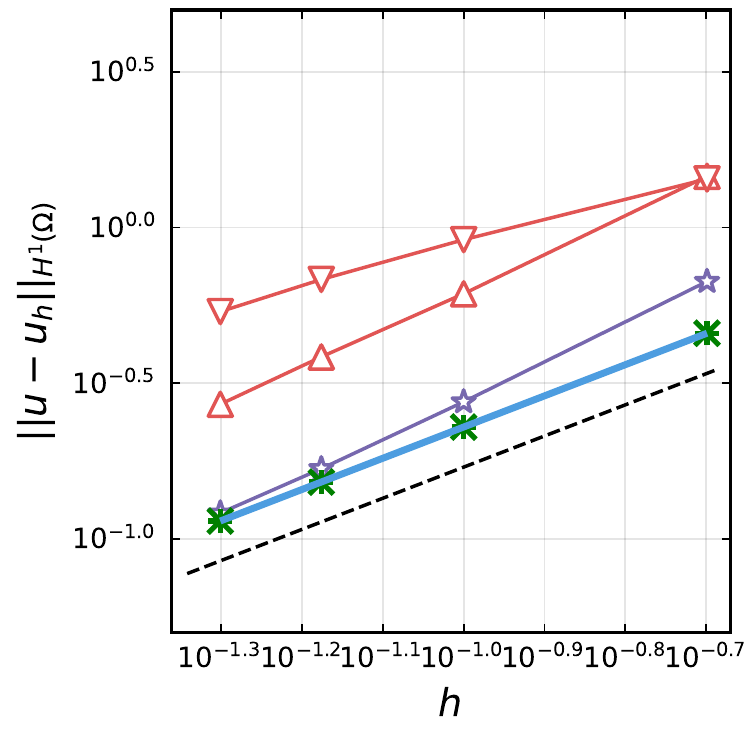}
		\caption{$\|u - u_h\|_{H^1(\Omega)}$ for $\gamma = 100$}
	\end{subfigure}
	\begin{subfigure}[t]{0.32\textwidth}
\includegraphics[width=1.00\textwidth]{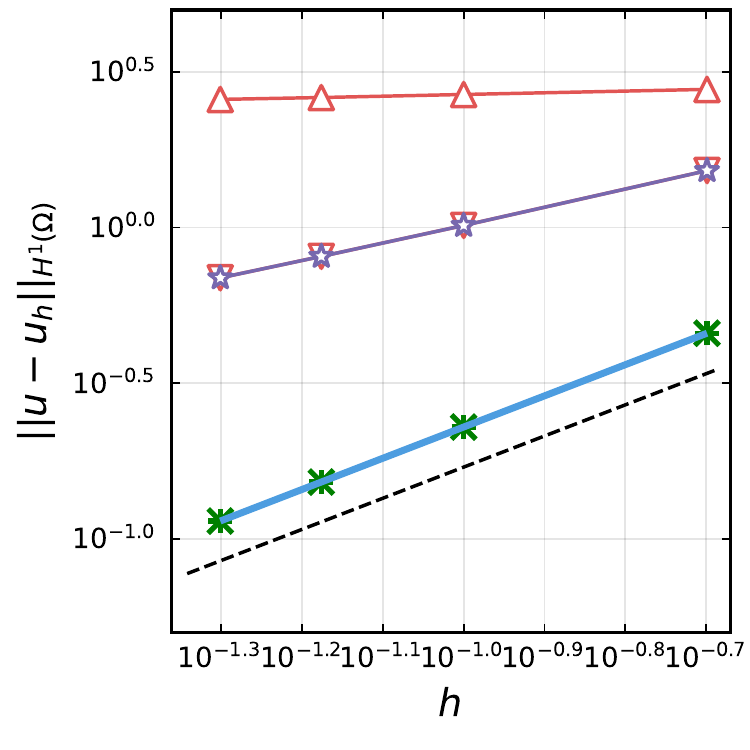}
		\caption{$\|u - u_h\|_{H^1(\Omega)}$ for $\gamma = 10^8$}
	\end{subfigure} 
	\caption[]{Condition number $\kappa_1(A)$, and error norms $\|u - u_h\|_{L^2(\Omega)}$ and $\|u - u_h \|_{H^1(\Omega)}$ vs. mesh size $h$ for the Poisson problem with Nitsche's method on the cube using linear elements. The dashed lines indicate the theoretical bound, i.e., $h^{-2}$ for $\kappa_1(A)$, $h^2$ for $\|u - u_h\|_{L^2(\Omega)}$ and $h$ for $\|u - u_h\|_{H^1(\Omega)}$.}
	\label{fig:k-l2e-h1e-vs-h-cube-k1-dir-poisson}
\end{figure}

\begin{figure}[!h]
	\centering
    \addlegendref
	
	\vspace*{0.5em}	
	\begin{subfigure}[t]{0.32\textwidth}
\includegraphics[width=1.0\textwidth]{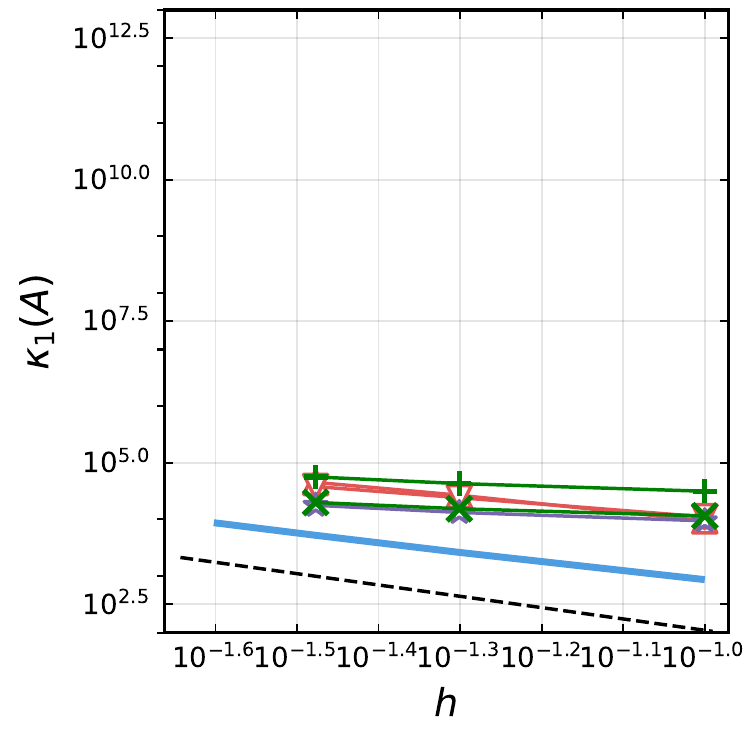}
		\caption{$\kappa_1(A)$ for $\gamma = 1$}
	\end{subfigure}
	\begin{subfigure}[t]{0.32\textwidth}
\includegraphics[width=1.0\textwidth]{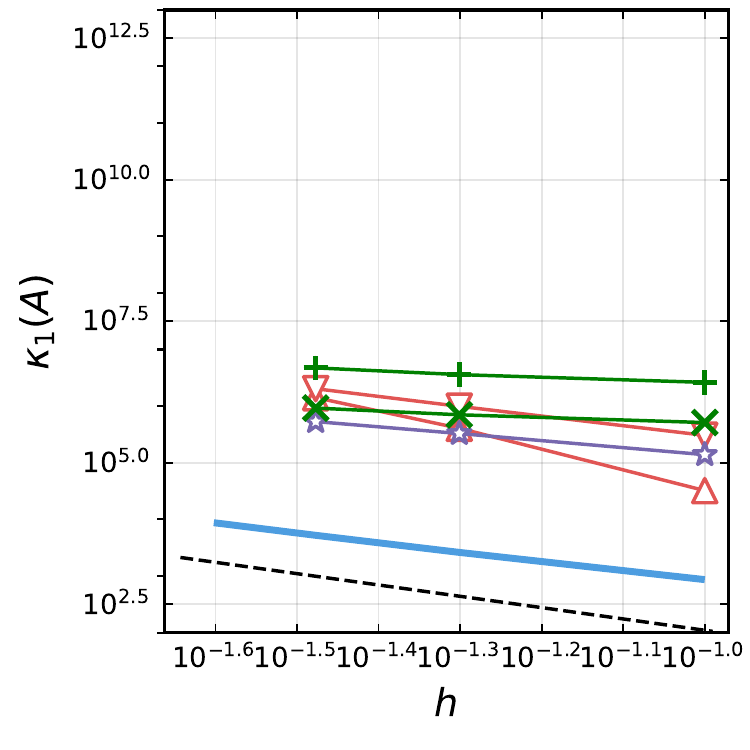}
		\caption{$\kappa_1(A)$ for $\gamma = 100$}
	\end{subfigure}
	\begin{subfigure}[t]{0.32\textwidth}
\includegraphics[width=1.0\textwidth]{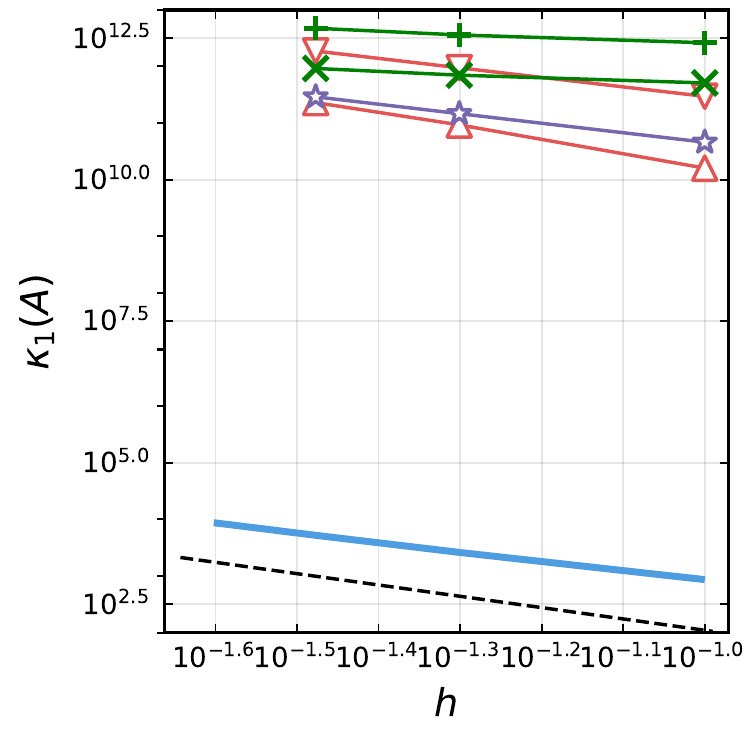}
		\caption{$\kappa_1(A)$ for $\gamma = 10^8$}
	\end{subfigure} \\
\vspace{0.6em}
\begin{subfigure}[t]{0.32\textwidth}
\includegraphics[width=1.00\textwidth]{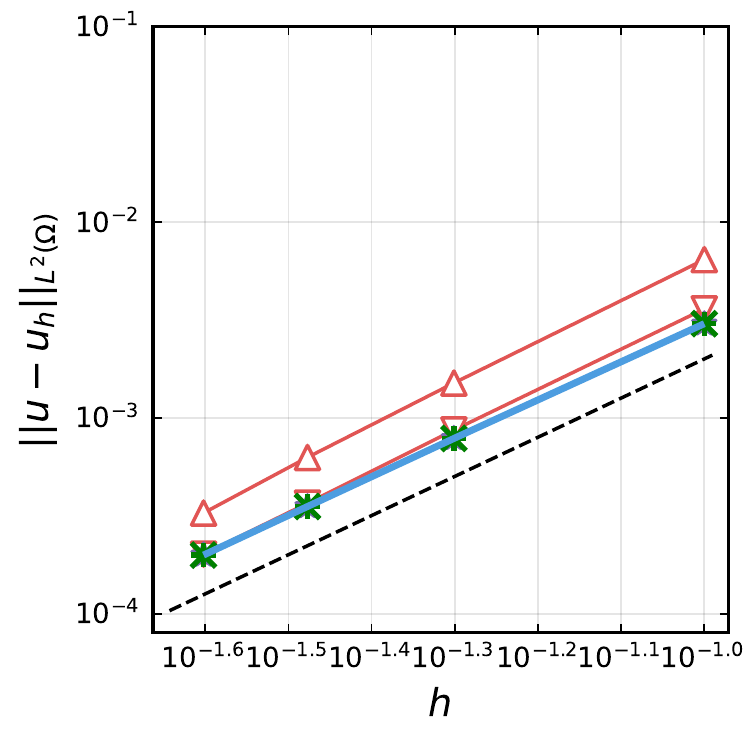}
		\caption{$\|u - u_h\|_{L^2(\Omega)}$ for $\gamma = 1$}
	\end{subfigure}
	\begin{subfigure}[t]{0.32\textwidth}
\includegraphics[width=1.00\textwidth]{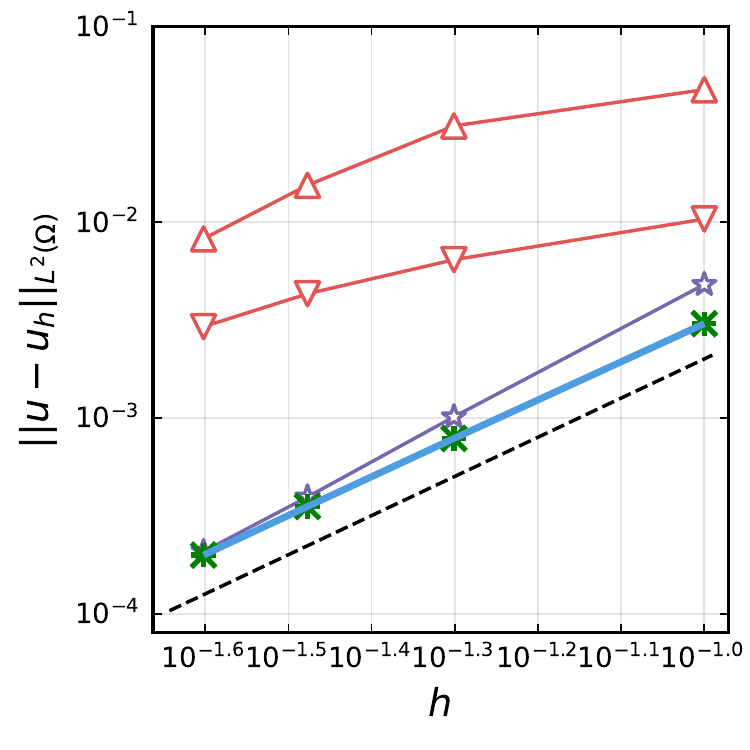}
		\caption{$\|u - u_h\|_{L^2(\Omega)}$ for $\gamma = 100$}
	\end{subfigure}
	\begin{subfigure}[t]{0.32\textwidth}
\includegraphics[width=1.00\textwidth]{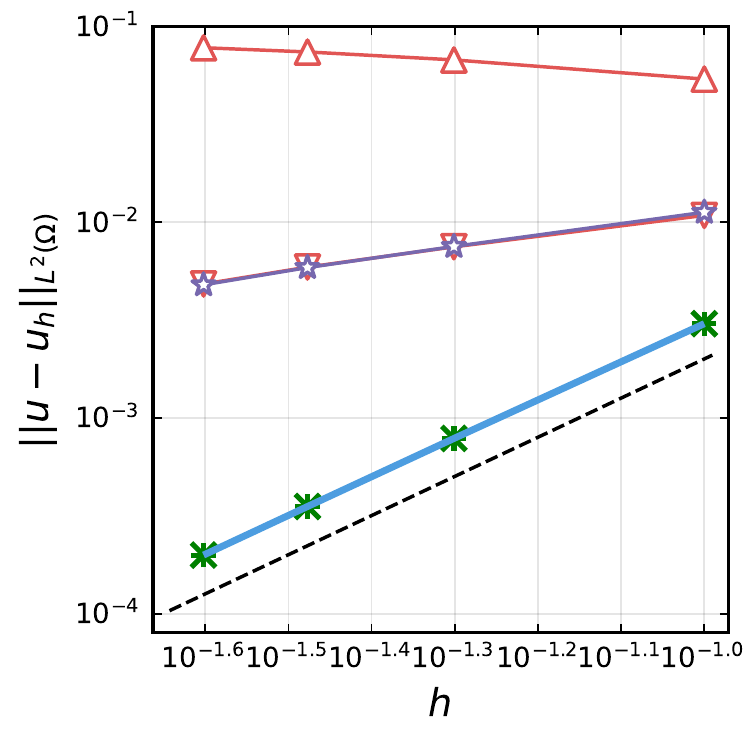}
		\caption{$\|u - u_h\|_{L^2(\Omega)}$ for $\gamma = 10^8$}
	\end{subfigure}\\ 
\vspace{0.6em}
\begin{subfigure}[t]{0.32\textwidth}
\includegraphics[width=1.00\textwidth]{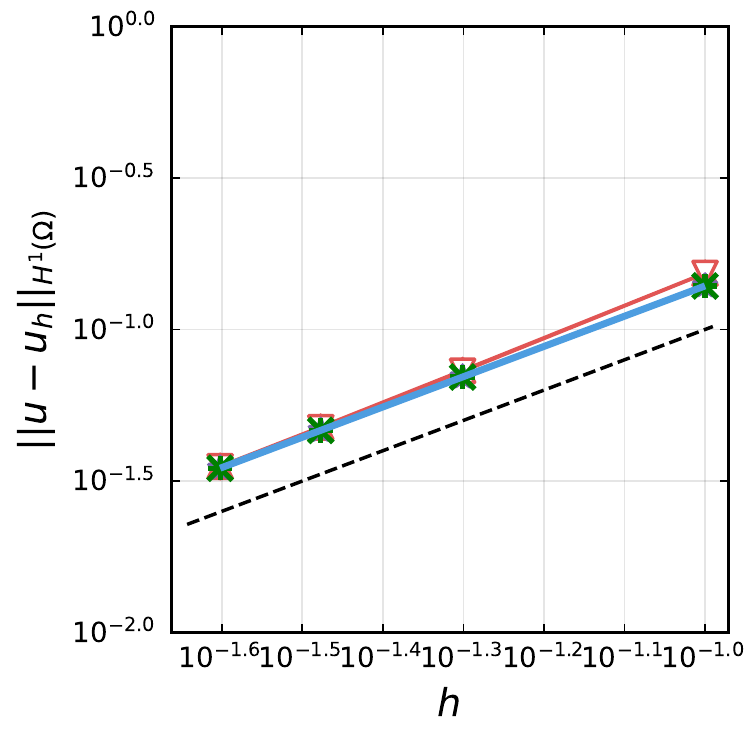}
		\caption{$\|u - u_h\|_{H^1(\Omega)}$ for $\gamma = 1$}
	\end{subfigure}
	\begin{subfigure}[t]{0.32\textwidth}
\includegraphics[width=1.00\textwidth]{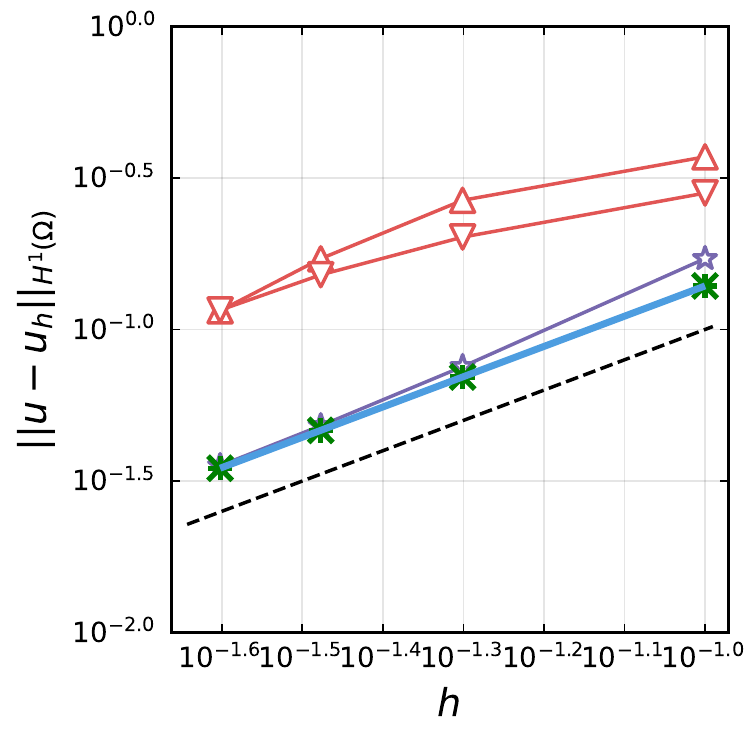}
		\caption{$\|u - u_h\|_{H^1(\Omega)}$ for $\gamma = 100$}
	\end{subfigure}
	\begin{subfigure}[t]{0.32\textwidth}
\includegraphics[width=1.00\textwidth]{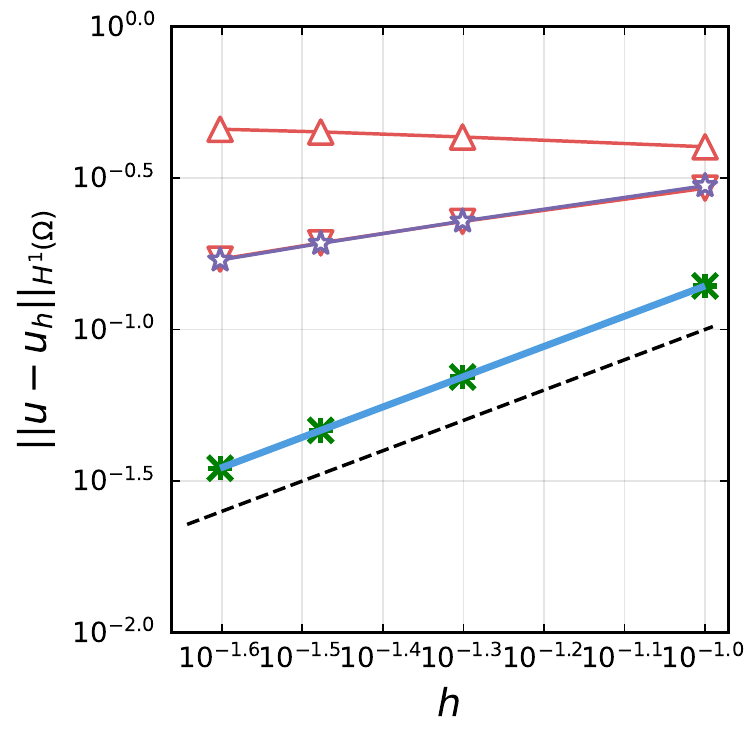}
		\caption{$\|u - u_h\|_{H^1(\Omega)}$ for $\gamma = 10^8$}
	\end{subfigure} 
	\caption[]{Condition number $\kappa_1(A)$, and error norms $\|u - u_h\|_{L^2(\Omega)}$ and $\|u - u_h \|_{H^1(\Omega)}$ vs. mesh size $h$ for the elasticity problem with Neumann boundary conditions on the 8-norm sphere using quadratic elements. The dashed lines indicate the theoretical bound, i.e., $h^{-2}$ for $\kappa_1(A)$, $h^3$ for $\|u - u_h\|_{L^2(\Omega)}$ and $h^2$ for $\|u - u_h\|_{H^1(\Omega)}$.}
	\label{fig:k-l2e-h1e-vs-h-8-ball-k2-neu-elasticity}
\end{figure}

\section{Conclusions}\label{sec:conclusions}

In this work, we have analysed the link between \emph{weak}  ghost penalty methods, which rely on a stabilisation term, and \emph{strong}  \acp{agfem}, which rely on discrete extension operators for the definition of the \ac{fe} spaces. When comparing these two families of methods, we observe that standard \ac{gp} formulations, e.g., the one being used in the CutFEM method, do not lead to acceptable strong versions. The problem is that the kernel of the penalty term does not enjoy the right approximability properties. The stabilisation is rigidising too much the solution, and exhibits a \emph{locking} phenomenon as the penalty parameter increases. As a result, we propose novel \ac{gp} methods that explicitly use the discrete extension operator in \ac{agfem}. The idea is to penalise the distance between the standard \ac{fe} space and the \ac{agfe} space. The penalty term has a local stabilisation structure in which the distance between these two spaces is expressed in terms of the difference between the solution and the discrete extension of its value on interior cells. The kernel of the penalty term is the \ac{agfe} space. Thus, these methods converge to the strong \ac{agfem} as the penalty parameters increases and are \emph{locking}-free.

We have carried out a thorough numerical analysis, in which we have compared all the methods with respect to condition number, $L^2$ and $H^1$ error. We have considered linear and quadratic \acp{fe}, Poisson and linear elasticity problems, weak imposition of Dirichlet conditions using Nitsche's method and Neumann boundary conditions, and a set of geometries in 2D and 3D. Due to the problems commented above, standard \ac{gp} formulations are strongly sensitive to the penalty parameter, losing convergence properties for moderate to large values of this parameter. On the contrary, the new \ac{gp} methods that are weak versions of \ac{agfem} are much less sensitive to the penalty parameter and their convergence is preserved in all cases. They are a good alternative to the strong formulation of \ac{agfem} and are systematically superior to existing \ac{gp} formulations.

\section*{Acknowledgments}

\newcommand{\thethanks}{This research was partially funded by the Australian Government through the Australian Research Council (project number DP210103092), the European Commission under the FET-HPC ExaQUte project (Grant agreement ID: 800898) within the Horizon 2020 Framework Programme and the project RTI2018-096898-B-I00 from the ``FEDER/Ministerio de Ciencia e Innovación – Agencia Estatal de Investigación''. F. Verdugo acknowledges support from the Spanish Ministry of Economy and Competitiveness through the ``Severo Ochoa Programme for Centers of Excellence in R\&D (CEX2018-000797-S)". F. Verdugo acknowledges support from the \emph{Secretaria d'Universitats i Recerca} of the Catalan Government in the framework of the Beatriu Pinós Program (Grant Id.: 2016 BP 00145). This work was also supported by computational resources provided by the Australian Government through NCI under the National Computational Merit Allocation Scheme.}

\thethanks

\setlength{\bibsep}{0.0ex plus 0.00ex}
\bibliographystyle{myabbrvnat}
\bibliography{refs} 

\begin{thebibliography}{55}
\providecommand{\natexlab}[1]{#1}
\providecommand{\url}[1]{\texttt{#1}}
\expandafter\ifx\csname urlstyle\endcsname\relax
  \providecommand{\doi}[1]{doi: #1}\else
  \providecommand{\doi}{doi: \begingroup \urlstyle{rm}\Url}\fi

\bibitem[Waisman and Berger-Vergiat(2013)]{Waisman2013}
H.~Waisman and L.~Berger-Vergiat.
\newblock An adaptive domain decomposition preconditioner for crack propagation
  problems modeled by {XFEM}.
\newblock \emph{International Journal for Multiscale Computational
  Engineering}, 11\penalty0 (6):\penalty0 633--654, 2013.
\newblock \doi{10.1615/IntJMultCompEng.2013006012}.

\bibitem[Berger-Vergiat et~al.(2012)Berger-Vergiat, Waisman, Hiriyur, Tuminaro,
  and Keyes]{berger-vergiat_inexact_2012}
L.~Berger-Vergiat, H.~Waisman, B.~Hiriyur, R.~Tuminaro, and D.~Keyes.
\newblock {Inexact Schwarz-algebraic multigrid preconditioners for crack
  problems modeled by extended finite element methods}.
\newblock \emph{International Journal for Numerical Methods in Engineering},
  90\penalty0 (3):\penalty0 311--328, 2012.
\newblock \doi{10.1002/nme.3318}.

\bibitem[Schott et~al.(2019)Schott, Ager, and Wall]{schott2019monolithic}
B.~Schott, C.~Ager, and W.~A. Wall.
\newblock Monolithic cut finite element--based approaches for fluid-structure
  interaction.
\newblock \emph{International Journal for Numerical Methods in Engineering},
  119\penalty0 (8):\penalty0 757--796, 2019.
\newblock \doi{10.1002/nme.6072}.

\bibitem[Alauzet et~al.(2016)Alauzet, Fabr{\`e}ges, Fern{\'a}ndez, and
  Landajuela]{alauzet2016nitsche}
F.~Alauzet, B.~Fabr{\`e}ges, M.~A. Fern{\'a}ndez, and M.~Landajuela.
\newblock {N}itsche-{XFEM} for the coupling of an incompressible fluid with
  immersed thin-walled structures.
\newblock \emph{Computer Methods in Applied Mechanics and Engineering},
  301:\penalty0 300--335, 2016.
\newblock \doi{10.1016/j.cma.2015.12.015}.

\bibitem[Zonca et~al.(2018)Zonca, Vergara, and Formaggia]{zonca2018unfitted}
S.~Zonca, C.~Vergara, and L.~Formaggia.
\newblock An unfitted formulation for the interaction of an incompressible
  fluid with a thick structure via an {XFEM/DG} approach.
\newblock \emph{SIAM Journal on Scientific Computing}, 40\penalty0
  (1):\penalty0 B59--B84, 2018.
\newblock \doi{10.1137/16M1097602}.

\bibitem[Massing et~al.(2015)Massing, Larson, Logg, and Rognes]{Massing2015}
A.~Massing, M.~G. Larson, A.~Logg, and M.~E. Rognes.
\newblock {A {N}itsche-based cut finite element method for a fluid-structure
  interaction problem}.
\newblock \emph{Communications in Applied Mathematics and Computational
  Science}, 10\penalty0 (2):\penalty0 97--120, 2015.
\newblock \doi{10.2140/camcos.2015.10.97}.

\bibitem[Sauerland and Fries(2011)]{Sauerland2011}
H.~Sauerland and T.~P. Fries.
\newblock {The extended finite element method for two-phase and free-surface
  flows: A systematic study}.
\newblock \emph{Journal of Computational Physics}, 230\penalty0 (9):\penalty0
  3369--3390, 2011.
\newblock \doi{10.1016/j.jcp.2011.01.033}.

\bibitem[Kirchhart et~al.(2016)Kirchhart, Gross, and
  Reusken]{kirchhart2016analysis}
M.~Kirchhart, S.~Gross, and A.~Reusken.
\newblock Analysis of an {XFEM} discretization for {S}tokes interface problems.
\newblock \emph{SIAM Journal on Scientific Computing}, 38\penalty0
  (2):\penalty0 A1019--A1043, 2016.
\newblock \doi{10.1137/15M1011779}.

\bibitem[Badia et~al.(2020)Badia, Caicedo, Mart{\'{i}}n, and
  Principe]{Badia2020Dec}
S.~Badia, M.~Caicedo, A.~F. Mart{\'{i}}n, and J.~Principe.
\newblock {A robust and scalable unfitted adaptive finite element framework for
  nonlinear solid mechanics}.
\newblock \emph{ArXiv e-prints}, 2020.
\newblock \href{https://arxiv.org/abs/2012.00280v2}{arXiv:2012.00280v2}.

\bibitem[Burman et~al.(2018)Burman, Elfverson, Hansbo, Larson, and
  Larsson]{Burman2018}
E.~Burman, D.~Elfverson, P.~Hansbo, M.~G. Larson, and K.~Larsson.
\newblock {Shape optimization using the cut finite element method}.
\newblock \emph{Computer Methods in Applied Mechanics and Engineering},
  328:\penalty0 242--261, 2018.
\newblock \doi{10.1016/j.cma.2017.09.005}.

\bibitem[Neiva et~al.(2020)Neiva, Chiumenti, Cervera, Salsi, Piscopo, Badia,
  Mart{\'\i}n, Chen, Lee, and Davies]{neiva2020numerical}
E.~Neiva, M.~Chiumenti, M.~Cervera, E.~Salsi, G.~Piscopo, S.~Badia, A.~F.
  Mart{\'\i}n, Z.~Chen, C.~Lee, and C.~Davies.
\newblock Numerical modelling of heat transfer and experimental validation in
  powder-bed fusion with the virtual domain approximation.
\newblock \emph{Finite Elements in Analysis and Design}, 168:\penalty0 103343,
  2020.
\newblock \doi{10.1016/j.finel.2019.103343}.

\bibitem[Carraturo et~al.(2020)Carraturo, Jomo, Kollmannsberger, Reali,
  Auricchio, and Rank]{carraturo2020modeling}
M.~Carraturo, J.~Jomo, S.~Kollmannsberger, A.~Reali, F.~Auricchio, and E.~Rank.
\newblock Modeling and experimental validation of an immersed thermo-mechanical
  part-scale analysis for laser powder bed fusion processes.
\newblock \emph{Additive Manufacturing}, 36:\penalty0 101498, 2020.
\newblock \doi{10.1016/j.addma.2020.101498}.

\bibitem[Badia et~al.(in press)Badia, Hampton, and Principe]{badia2019embedded}
S.~Badia, J.~Hampton, and J.~Principe.
\newblock Embedded multilevel monte carlo for uncertainty quantification in
  random domains.
\newblock \emph{International Journal for Uncertainty Quantification}, in
  press.
\newblock \doi{10.1615/Int.J.UncertaintyQuantification.2021032984}.

\bibitem[Belytschko et~al.(2001)Belytschko, Mo{\"{e}}s, Usui, and
  Parimi]{belytschko_arbitrary_2001}
T.~Belytschko, N.~Mo{\"{e}}s, S.~Usui, and C.~Parimi.
\newblock {Arbitrary discontinuities in finite elements}.
\newblock \emph{International Journal for Numerical Methods in Engineering},
  50\penalty0 (4):\penalty0 993--1013, 2001.
\newblock \doi{10.1002/1097-0207(20010210)50:4<993::AID-NME164>3.0.CO;2-M}.

\bibitem[Burman et~al.(2015)Burman, Claus, Hansbo, Larson, and
  Massing]{burman_cutfem_2015}
E.~Burman, S.~Claus, P.~Hansbo, M.~G. Larson, and A.~Massing.
\newblock {CutFEM: Discretizing Geometry and Partial Differential Equations}.
\newblock \emph{International Journal for Numerical Methods in Engineering},
  104\penalty0 (7):\penalty0 472--501, 2015.
\newblock \doi{10.1002/nme.4823}.

\bibitem[Badia et~al.(2018)Badia, Verdugo, and Mart{\'{i}}n]{Badia2018}
S.~Badia, F.~Verdugo, and A.~F. Mart{\'{i}}n.
\newblock {The aggregated unfitted finite element method for elliptic
  problems}.
\newblock \emph{Computer Methods in Applied Mechanics and Engineering},
  336:\penalty0 533--553, 2018.
\newblock \doi{10.1016/j.cma.2018.03.022}.

\bibitem[Elfverson et~al.(2018)Elfverson, Larson, and Larsson]{Elfverson2018}
D.~Elfverson, M.~G. Larson, and K.~Larsson.
\newblock {CutIGA with basis function removal}.
\newblock \emph{Advanced Modeling and Simulation in Engineering Sciences},
  5\penalty0 (1):\penalty0 6, 2018.
\newblock \doi{10.1186/s40323-018-0099-2}.

\bibitem[Mittal and Iaccarino(2005)]{Mittal2005}
R.~Mittal and G.~Iaccarino.
\newblock {Immersed Boundary Methods}.
\newblock \emph{Annual Review of Fluid Mechanics}, 37\penalty0 (1):\penalty0
  239--261, 2005.
\newblock \doi{10.1146/annurev.fluid.37.061903.175743}.

\bibitem[Schillinger and Ruess(2015)]{Schillinger2015}
D.~Schillinger and M.~Ruess.
\newblock The {F}inite {C}ell {M}ethod: {A} review in the context of
  higher-order structural analysis of {CAD} and image-based geometric models.
\newblock \emph{Archives of Computational Methods in Engineering}, 22\penalty0
  (3):\penalty0 391--455, 2015.
\newblock \doi{10.1007/s11831-014-9115-y}.

\bibitem[Main and Scovazzi(2018)]{main2018shifted}
A.~Main and G.~Scovazzi.
\newblock The {S}hifted {B}oundary {M}ethod for embedded domain computations.
  {P}art {I}: Poisson and {S}tokes problems.
\newblock \emph{Journal of Computational Physics}, 372:\penalty0 972--995,
  2018.
\newblock \doi{10.1016/j.jcp.2017.10.026}.

\bibitem[Kamensky et~al.(2015)Kamensky, Hsu, Schillinger, Evans, Aggarwal,
  Bazilevs, Sacks, and Hughes]{kamensky2015immersogeometric}
D.~Kamensky, M.-C. Hsu, D.~Schillinger, J.~A. Evans, A.~Aggarwal, Y.~Bazilevs,
  M.~S. Sacks, and T.~J. Hughes.
\newblock An immersogeometric variational framework for fluid--structure
  interaction: Application to bioprosthetic heart valves.
\newblock \emph{Computer methods in applied mechanics and engineering},
  284:\penalty0 1005--1053, 2015.
\newblock \doi{10.1016/j.cma.2014.10.040}.

\bibitem[Navarro-Jim{\ifmmode\acute{e}\else\'{e}\fi}nez
  et~al.(2020)Navarro-Jim{\ifmmode\acute{e}\else\'{e}\fi}nez, Nadal, Tur,
  Mart{\ifmmode\acute{\imath}\else\'{\i}\fi}nez-Casas, and
  R{\ifmmode\acute{o}\else\'{o}\fi}denas]{Navarro-Jimenez2020Jul}
J.~M. Navarro-Jim{\ifmmode\acute{e}\else\'{e}\fi}nez, E.~Nadal, M.~Tur,
  J.~Mart{\ifmmode\acute{\imath}\else\'{\i}\fi}nez-Casas, and J.~J.
  R{\ifmmode\acute{o}\else\'{o}\fi}denas.
\newblock {On the use of stabilization techniques in the Cartesian grid finite
  element method framework for iterative solvers}.
\newblock \emph{International Journal for Numerical Methods in Engineering},
  121\penalty0 (13):\penalty0 3004--3020, 2020.
\newblock \doi{10.1002/nme.6344}.

\bibitem[Saye(2017)]{saye2017implicit}
R.~Saye.
\newblock Implicit mesh discontinuous {G}alerkin methods and interfacial gauge
  methods for high-order accurate interface dynamics, with applications to
  surface tension dynamics, rigid body fluid--structure interaction, and free
  surface flow: {P}art {I}.
\newblock \emph{Journal of Computational Physics}, 344:\penalty0 647--682,
  2017.
\newblock \doi{10.1016/j.jcp.2017.04.076}.

\bibitem[Engwer and Heimann(2012)]{engwer2012dune}
C.~Engwer and F.~Heimann.
\newblock {Dune-UDG}: a cut-cell framework for unfitted discontinuous
  {G}alerkin methods.
\newblock In \emph{Advances in DUNE}, pages 89--100. Springer, 2012.
\newblock \doi{10.1007/978-3-642-28589-9_7}.

\bibitem[Johansson and Larson(2013)]{johansson2013high}
A.~Johansson and M.~G. Larson.
\newblock A high order discontinuous {G}alerkin {N}itsche method for elliptic
  problems with fictitious boundary.
\newblock \emph{Numerische Mathematik}, 123\penalty0 (4):\penalty0 607--628,
  2013.
\newblock \doi{10.1007/s00211-012-0497-1}.

\bibitem[M{\"u}ller et~al.(2017)M{\"u}ller, Kr{\"a}mer-Eis, Kummer, and
  Oberlack]{muller2017high}
B.~M{\"u}ller, S.~Kr{\"a}mer-Eis, F.~Kummer, and M.~Oberlack.
\newblock A high-order discontinuous {G}alerkin method for compressible flows
  with immersed boundaries.
\newblock \emph{International Journal for Numerical Methods in Engineering},
  110\penalty0 (1):\penalty0 3--30, 2017.
\newblock \doi{10.1002/nme.5343}.

\bibitem[de~Prenter et~al.(2017)de~Prenter, Verhoosel, van Zwieten, and van
  Brummelen]{DePrenter2017}
F.~de~Prenter, C.~V. Verhoosel, G.~J. van Zwieten, and E.~H. van Brummelen.
\newblock {Condition number analysis and preconditioning of the finite cell
  method}.
\newblock \emph{Computer Methods in Applied Mechanics and Engineering},
  316:\penalty0 297--327, 2017.
\newblock \doi{10.1016/j.cma.2016.07.006}.

\bibitem[Neiva and Badia(2021)]{Neiva2021}
E.~Neiva and S.~Badia.
\newblock Robust and scalable h-adaptive aggregated unfitted finite elements
  for interface elliptic problems.
\newblock \emph{Computer Methods in Applied Mechanics and Engineering},
  380:\penalty0 113769, July 2021.
\newblock \doi{10.1016/j.cma.2021.113769}.

\bibitem[Kummer(2017)]{Kummer2017}
F.~Kummer.
\newblock {Extended discontinuous {G}alerkin methods for two-phase flows: the
  spatial discretization}.
\newblock \emph{International Journal for Numerical Methods in Engineering},
  109\penalty0 (2):\penalty0 259--289, 2017.
\newblock \doi{10.1002/nme.5288}.

\bibitem[Lehrenfeld(2016)]{lehrenfeld2016high}
C.~Lehrenfeld.
\newblock High order unfitted finite element methods on level set domains using
  isoparametric mappings.
\newblock \emph{Computer Methods in Applied Mechanics and Engineering},
  300:\penalty0 716--733, 2016.
\newblock \doi{10.1016/j.cma.2015.12.005}.

\bibitem[Guzm{\'a}n et~al.(2017)Guzm{\'a}n, S{\'a}nchez, and
  Sarkis]{guzman2017finite}
J.~Guzm{\'a}n, M.~A. S{\'a}nchez, and M.~Sarkis.
\newblock A finite element method for high-contrast interface problems with
  error estimates independent of contrast.
\newblock \emph{Journal of Scientific Computing}, 73\penalty0 (1):\penalty0
  330--365, 2017.
\newblock \doi{10.1007/s10915-017-0415-x}.

\bibitem[Li et~al.(2019)Li, Atallah, Main, and Scovazzi]{li2019shifted}
K.~Li, N.~M. Atallah, G.~A. Main, and G.~Scovazzi.
\newblock The {S}hifted {I}nterface {M}ethod: A flexible approach to embedded
  interface computations.
\newblock \emph{International Journal for Numerical Methods in Engineering},
  2019.
\newblock \doi{10.1002/nme.6231}.

\bibitem[Burman(2010)]{burman2010ghost}
E.~Burman.
\newblock Ghost penalty.
\newblock \emph{Comptes Rendus Mathematique}, 348\penalty0 (21-22):\penalty0
  1217--1220, 2010.
\newblock \doi{10.1016/j.crma.2010.10.006}.

\bibitem[Helzel et~al.(2005)Helzel, Berger, and
  Leveque]{helzel_high-resolution_2005}
C.~Helzel, M.~Berger, and R.~Leveque.
\newblock A high-resolution rotated grid method for conservation laws with
  embedded geometries.
\newblock \emph{SIAM Journal on Scientific Computing}, 26\penalty0
  (3):\penalty0 785--809, 2005.
\newblock \doi{10.1137/S106482750343028X}.

\bibitem[Bastian and Engwer(2009)]{bastian2009unfitted}
P.~Bastian and C.~Engwer.
\newblock An unfitted finite element method using discontinuous {G}alerkin.
\newblock \emph{International journal for numerical methods in engineering},
  79\penalty0 (12):\penalty0 1557--1576, 2009.
\newblock \doi{10.1002/nme.2631}.

\bibitem[Badia et~al.(2018)Badia, Mart{\'{i}}n, and Verdugo]{Badia2018a}
S.~Badia, A.~F. Mart{\'{i}}n, and F.~Verdugo.
\newblock {Mixed aggregated finite element methods for the unfitted
  discretization of the {S}tokes problem}.
\newblock \emph{SIAM Journal on Scientific Computing}, 40\penalty0
  (6):\penalty0 B1541--B1576, 2018.
\newblock \doi{10.1137/18M1185624}.

\bibitem[Verdugo et~al.(2019)Verdugo, Mart{\'{i}}n, and Badia]{Verdugo2019}
F.~Verdugo, A.~F. Mart{\'{i}}n, and S.~Badia.
\newblock {Distributed-memory parallelization of the aggregated unfitted finite
  element method}.
\newblock \emph{Computer Methods in Applied Mechanics and Engineering},
  357:\penalty0 112583, 2019.
\newblock \doi{10.1016/j.cma.2019.112583}.

\bibitem[Badia et~al.(2021)Badia, Mart{\'{\i}}n, Neiva, and
  Verdugo]{Badia2020Jun}
S.~Badia, A.~F. Mart{\'{\i}}n, E.~Neiva, and F.~Verdugo.
\newblock The aggregated unfitted finite element method on parallel tree-based
  adaptive meshes.
\newblock \emph{{SIAM} Journal on Scientific Computing}, 43\penalty0
  (3):\penalty0 C203--C234, Jan. 2021.
\newblock \doi{10.1137/20m1344512}.

\bibitem[Burman et~al.(2020)Burman, Hansbo, and Larson]{Burman2020Nov}
E.~Burman, P.~Hansbo, and M.~G. Larson.
\newblock Explicit time stepping for the wave equation using cutfem with
  discrete extension.
\newblock \emph{ArXiv e-prints}, 2020.
\newblock \href{https://arxiv.org/abs/2011.05386}{arXiv:2011.05386}.

\bibitem[{N}itsche(1971)]{nitsche_uber_2013}
J.~{N}itsche.
\newblock {{\"{U}}ber ein Variationsprinzip zur L{\"{o}}sung von
  Dirichlet-Problemen bei Verwendung von Teilr{\"{a}}umen, die keinen
  Randbedingungen unterworfen sind}.
\newblock \emph{Abhandlungen aus dem Mathematischen Seminar der
  Universit{\"{a}}t Hamburg}, 36\penalty0 (1):\penalty0 9--15, 1971.
\newblock \doi{10.1007/BF02995904}.

\bibitem[Freund and Stenberg(1995)]{Freund1995}
J.~Freund and R.~Stenberg.
\newblock {On weakly imposed boundary conditions for second order problems}.
\newblock In \emph{{Finite elements in fluids, Italia, 15-21.10.1995}}, pages
  327--336. Padovan yliopisto, 1995.

\bibitem[Badia and Verdugo(2018)]{badia_robust_2017}
S.~Badia and F.~Verdugo.
\newblock {Robust and scalable domain decomposition solvers for unfitted finite
  element methods}.
\newblock \emph{Journal of Computational and Applied Mathematics},
  344:\penalty0 740--759, 2018.
\newblock \doi{10.1016/j.cam.2017.09.034}.

\bibitem[Burman and Hansbo(2012)]{Burman2012}
E.~Burman and P.~Hansbo.
\newblock Fictitious domain finite element methods using cut elements: {II}. a
  stabilized nitsche method.
\newblock \emph{Applied Numerical Mathematics}, 62\penalty0 (4):\penalty0
  328--341, Apr. 2012.
\newblock \doi{10.1016/j.apnum.2011.01.008}.

\bibitem[Hansbo et~al.(2017)Hansbo, Larson, and Larsson]{Hansbo2017}
P.~Hansbo, M.~G. Larson, and K.~Larsson.
\newblock {Cut Finite Element Methods for Linear Elasticity Problems}.
\newblock In \emph{{Geometrically Unfitted Finite Element Methods and
  Applications}}, pages 25--63. Springer, 2017.
\newblock \doi{10.1007/978-3-319-71431-8_2}.

\bibitem[Hansbo and Hansbo(2002)]{hansbo2002unfitted}
A.~Hansbo and P.~Hansbo.
\newblock An unfitted finite element method, based on {N}itsche’s method, for
  elliptic interface problems.
\newblock \emph{Computer methods in applied mechanics and engineering},
  191\penalty0 (47-48):\penalty0 5537--5552, 2002.
\newblock \doi{10.1016/S0045-7825(02)00524-8}.

\bibitem[Becker and Braack(2001)]{Becker2001Dec}
R.~Becker and M.~Braack.
\newblock {A finite element pressure gradient stabilization for the Stokes
  equations based on local projections}.
\newblock \emph{CALCOLO}, 38\penalty0 (4):\penalty0 173--199, 2001.
\newblock \doi{10.1007/s10092-001-8180-4}.

\bibitem[Badia(2012)]{Badia2012Nov}
S.~Badia.
\newblock {On stabilized finite element methods based on the
  Scott{\textendash}Zhang projector. Circumventing the inf{\textendash}sup
  condition for the Stokes problem}.
\newblock \emph{Computer Methods in Applied Mechanics and Engineering},
  247-248:\penalty0 65--72, 2012.
\newblock \doi{10.1016/j.cma.2012.07.020}.

\bibitem[Scott and Zhang(1990)]{scott1990finite}
L.~R. Scott and S.~Zhang.
\newblock Finite element interpolation of nonsmooth functions satisfying
  boundary conditions.
\newblock \emph{Mathematics of Computation}, 54\penalty0 (190):\penalty0
  483--483, May 1990.
\newblock \doi{10.1090/s0025-5718-1990-1011446-7}.

\bibitem[Nguyen et~al.(2017)Nguyen, Stoter, Baum, Kirschke, Ruess, Yosibash,
  and Schillinger]{Nguyen2017}
L.~Nguyen, S.~Stoter, T.~Baum, J.~Kirschke, M.~Ruess, Z.~Yosibash, and
  D.~Schillinger.
\newblock {Phase-field boundary conditions for the voxel finite cell method:
  Surface-free stress analysis of CT-based bone structures}.
\newblock \emph{International Journal for Numerical Methods in Biomedical
  Engineering}, 33\penalty0 (12):\penalty0 e2880, 2017.
\newblock \doi{10.1002/cnm.2880}.

\bibitem[Badia et~al.(2018)Badia, Mart{\'{i}}n, and Principe]{badia-fempar}
S.~Badia, A.~F. Mart{\'{i}}n, and J.~Principe.
\newblock {FEMPAR: An Object-Oriented Parallel Finite Element Framework}.
\newblock \emph{Archives of Computational Methods in Engineering}, 25\penalty0
  (2):\penalty0 195--271, 2018.
\newblock \doi{10.1007/s11831-017-9244-1}.

\bibitem[Arndt et~al.(2021)Arndt, Bangerth, Davydov, Heister, Heltai,
  Kronbichler, Maier, Pelteret, Turcksin, and Wells]{dealii2019design}
D.~Arndt, W.~Bangerth, D.~Davydov, T.~Heister, L.~Heltai, M.~Kronbichler,
  M.~Maier, J.-P. Pelteret, B.~Turcksin, and D.~Wells.
\newblock The {deal.II} finite element library: Design, features, and insights.
\newblock \emph{Computers \& Mathematics with Applications}, 81:\penalty0
  407--422, 2021.
\newblock \doi{10.1016/j.camwa.2020.02.022}.

\bibitem[Badia and Verdugo(2020)]{Badia2020Aug}
S.~Badia and F.~Verdugo.
\newblock {Gridap: An extensible Finite Element toolbox in Julia}.
\newblock \emph{Journal of Open Source Software}, 5\penalty0 (52):\penalty0
  2520, 2020.
\newblock \doi{10.21105/joss.02520}.

\bibitem[Aln{\ae}s et~al.(2015)Aln{\ae}s, Blechta, Hake, Johansson, Kehlet,
  Logg, Richardson, Ring, Rognes, and Wells]{Alnaes2015Dec}
M.~Aln{\ae}s, J.~Blechta, J.~Hake, A.~Johansson, B.~Kehlet, A.~Logg,
  C.~Richardson, J.~Ring, M.~E. Rognes, and G.~N. Wells.
\newblock {The FEniCS Project Version 1.5}.
\newblock \emph{Archive of Numerical Software}, 3\penalty0 (100), 2015.
\newblock \doi{10.11588/ans.2015.100.20553}.

\bibitem[de~Prenter et~al.(2018)de~Prenter, Lehrenfeld, and
  Massing]{de2018note}
F.~de~Prenter, C.~Lehrenfeld, and A.~Massing.
\newblock {A note on the stability parameter in Nitsche{'}s method for unfitted
  boundary value problems}.
\newblock \emph{Computers {\&} Mathematics with Applications}, 75\penalty0
  (12):\penalty0 4322--4336, 2018.
\newblock \doi{10.1016/j.camwa.2018.03.032}.

\bibitem[NCI-Gadi {W}eb site()]{NCIgadi}
NCI-Gadi {W}eb site.
\newblock \url{https://nci.org.au/our-systems/hpc-systems}.
\newblock Accessed: 2020-06-22.

\end{thebibliography}
  
\end{document}